\documentclass[12pt]{amsart}
\usepackage{latexsym}
\usepackage{amssymb, amsmath}
\usepackage[left=2.5cm,right=2.5cm,top=2.5cm, bottom=2.5cm ]{geometry}
\usepackage{graphicx}
\usepackage{tikz}
\usepackage{pgfplots}

\usepackage{float}

\usepackage{setspace}
\usepackage[pagebackref,hypertexnames=false, colorlinks, citecolor=red, linkcolor=red]{hyperref}

\newcommand{\McC}{\raise.5ex\hbox{c}}

\newcommand{\D}{\mathbb{D}}
\newcommand{\T}{\mathbb{T}}

\newcommand{\C}{\mathbb{C}}

\newcommand{\R}{\mathbb{R}}
\newcommand{\p}{\mathfrak{p}}
\newcommand{\q}{\mathfrak{q}}

\newtheorem{theorem}{Theorem}[section]

\newcommand{\Da}{{\mathfrak D}_{\alpha}}
\newcommand{\Dveca}{{\mathfrak D}_{\vec{\alpha}}}

\newcommand{\veca}{{\vec{\alpha}}}

\newtheorem{lemma}[theorem]{Lemma}

\newtheorem*{theorem*}{Theorem}
\newtheorem*{conjecture*}{Conjecture}
\newtheorem{corollary}[theorem]{Corollary}
\newtheorem*{corollary*}{Corollary}
\newtheorem{proposition}[theorem]{Proposition}

\theoremstyle{remark}
\newtheorem{remark}[theorem]{Remark}
\newtheorem{question}{Question}
\newtheorem{definition}[theorem]{Definition}
\newtheorem{example}[theorem]{Example}

\author[Bickel]{Kelly Bickel$^\dagger$}
\address{Department of Mathematics, Bucknell University, 360 Olin Science Building, Lewisburg, PA 17837, USA.}
\email{kelly.bickel@bucknell.edu}
\thanks{$\dagger$ Research supported in part by National Science Foundation
DMS grant \#1448846.}

\author[Pascoe]{James Eldred Pascoe$^\ddagger$}
\address{Department of Mathematics, Washington University in St. Louis, 1 Brookings Drive, Campus Box 1146, St. Louis, MO 63130, USA.}
\email{pascoej@wustl.edu}
\thanks{$\ddagger$ Research supported by NSF Mathematical Science Postdoctoral Research Fellowship DMS 1606260.}

\author[Sola]{Alan Sola}
\address{Department of Mathematics, Stockholm University, 106 91 Stockholm, Sweden.}
\email{sola@math.su.se}

\begin{document}
\title[Derivatives of rational inner functions]{Derivatives of rational inner functions: geometry of singularities and integrability at the boundary}
\date{\today}

\maketitle
\begin{abstract}
We analyze the singularities of rational inner functions on the unit bidisk and study both when these functions belong to Dirichlet-type spaces and when their partial derivatives belong to Hardy spaces. We characterize derivative $H^\p$ membership purely in terms of contact order, a measure of the rate at which the zero set of a rational inner function approaches the distinguished boundary of the bidisk. We also show that derivatives of rational inner functions with singularities fail to be in $H^\p$ for $\p\ge\frac{3}{2}$ and that 
higher non-tangential regularity of a rational inner function paradoxically reduces the $H^\p$ integrability of its derivative.
We derive inclusion results for Dirichlet-type spaces from derivative inclusion for $H^\p$. Using Agler decompositions and local Dirichlet integrals, we further prove that a restricted class of rational inner functions fails to belong to the unweighted Dirichlet space.

\end{abstract}

\tableofcontents


 \section{Introduction}
\subsection{Rational inner functions on the bidisk}
The study of the boundary behavior of holomorphic functions on domains in $\C^n$ is a deep and classical branch of complex analysis. One of the most natural problems in this context is to relate the properties of a function at a boundary point to those of its derivatives close to that point.
In this paper, we address such questions for certain rational functions in two complex variables having additional algebraic structure. As we shall see, these investigations involve a rich and fascinating interplay betwen function theory, harmonic analysis of several complex variables, and classical algebraic geometry in the guise of algebraic curves. 
 
In several variables, a complicating phenomenon without counterpart in the one-variable setting manifests itself; namely, a rational function can possess a ``non-essential singularity of the second kind,'' meaning its numerator and denominator vanish at the same point without sharing a common factor. From an analytic standpoint, such a boundary singularity affects the boundedness, smoothness, and integrability properties of both the rational function and its derivatives in very nontrivial ways. To understand this problem further, we are led to develop a precise geometric theory of singularities of such functions using a combination of complex analysis and complex geometry. 

We specialize to a specific domain in $\mathbb{C}^2,$ namely the unit bidisk \[\mathbb{D}^2=\{(z_1, z_2)\in \C^2\colon |z_1|<1, |z_2|<1\}.\]  In this setting, Greg Knese  recently initiated an impressive study
of integrability and regularity properties of rational functions, see \cite{Kne15}. Specifically, his theory gives a method for determining if a given rational 
function $q/p$ is square integrable on the distinguished boundary of the bidisk $\T^2 := ( \partial \mathbb{D})^2$, that is if $q/p\in L^2(\T^2)$.  The Knese integrability theory motivates the question of whether or not the derivatives of a
rational function are square-integrable, or in other words, whether they belong to the Hardy space $H^2(\D^2)$.

For general rational functions, such a question is very challenging. However, the arguments and methods in  \cite{Kne15} rest on properties of a special subset of rational functions. Specifically, we say a holomorphic function $\phi\colon \D^2 \to \C$ is {\it inner} if $|\phi(\zeta)|=1$ for almost every $\zeta=(\zeta_1, \zeta_2) \in \mathbb{T}^2.$ Then the subclass of {\it rational inner functions} (RIFs) consists of inner functions $\phi=q/p$, where $p$ and $q$ are polynomials in two complex variables. Rational inner functions on the bidisk are necessarily of the form
\[\phi(z_1, z_2)=cz_1^Mz_2^N\frac{\tilde{p}(z_1,z_2)}{p(z_1,z_2)}\]
where  $p$ is a polynomial of bidegree $(m,n)$ with no zeros in the bidisk, $\tilde{p}(z_1, z_2):=z_1^mz_2^n\overline{p\left(\frac{1}{\bar{z}_1}, \frac{1}{\bar{z}_2}\right)}$
is the reflection of $p$, and $c$ is a unimodular constant, see \cite[Chapter 5.2]{Rud69}. In what follows, we specialize to rational inner functions of the form $\frac{\tilde{p}}{p}$, but the results for RIFs with monomial factors follow easily.

Inner functions on the unit disk play an essential role in single-variable operator and function theory. Indeed, many Hilbert space contractions are unitarily equivalent to the compression of the shift to a model space, namely a Hilbert space associated to an inner function \cite{SNF10}.  Similarly, much of analytic function theory on the disk rests on the classic Beurling-Lax theorem and inner-outer factorizations of one-variable Hardy space functions \cite{Gar07}. 

Inner functions on the bidisk are also quite important. As in the one-variable setting, they induce interesting Hilbert spaces associated to useful operators, see for example \cite{bk13, BLPreprint} and the references therein. Although inner-outer factorization fails on the bidisk,  every bounded holomorphic function on $\mathbb{D}^2$ can still be approximated locally-uniformly by constant multiples of inner functions \cite{Rud69}. Rational inner functions play a particularly critical role in two-variable function theory. For example, RIFs provide key examples of two-variable matrix monotone functions \cite{AMY12b} and provide solutions to every solvable Pick interpolation problem on the bidisk \cite{am99, AglMcC}. Moreover RIFs and their denominators, called \emph{stable polynomials}, have close connections to applications in systems and control engineering, electrical engineering, and statistical mechanics; see \cite{BSV05, GKVVW16, Kum02, Wag11} and the references cited within those papers.

\subsection{Integrability Questions} 
Because of their critical role in function theory on the bidisk and their particularly nice formulas, we restrict attention to RIFs.  Specifically, in this paper, we will address two integrability questions related to rational inner functions and their derivatives. 

First recall that a holomorphic function $f$ on the bidisk belongs to the {\it Hardy space} $H^\p(\D^2)$ for $0<\p<\infty$ if
\[\|f\|^\p_{H^\p(\mathbb{D}^2)}:=\sup_{0<r<1}\frac{1}{4\pi^2}\int_{\T^2}|f(r\zeta)|^\p|d\zeta|<\infty.\]
Here, $|d\zeta|$, denotes Lebesgue measure on $\T^2$. The space $H^{\infty}(\D^2)$ is the Banach algebra of bounded holomorphic functions on the bidisk. 
See \cite{Rud69} for the basic theory of $H^\p$-spaces on polydisks. Since RIFs are bounded, they immediately belong to every $H^\p(\D^2)$ for $0 <\p \le \infty.$
We can now pose the following question.
\begin{question} \label{Q1}
If $\phi$ is a rational inner function on $\mathbb{D}^2$, under what conditions is $\frac{\partial\phi}{\partial z_j} \in H^\p(\mathbb{D}^2)$ for $j=1,2$ and $0 < \p \le \infty$?
\end{question}
In the unit disk $\mathbb{D},$ analogous and more general forms of Question \ref{Q1} have been considered in numerous publications going back to work of Ahern and Clark \cite{AhCla74}; see for instance \cite{GrNic17} and the references therein for subsequent developments. In higher dimensions, see \cite{BedTay80}, it is known that an inner function on the unit ball $\mathbb{B}^2$ whose gradient is in $L^2(\mathbb{B}^2)$ must be constant. Furthermore, a number of authors have examined the related question of boundary regularity of bounded holomorphic functions on $\mathbb{D}^2$, see \cite{Aba98, AMY11, AMY12} and the references provided there.  However as far as the authors are aware,  Question \ref{Q1} has been open in the setting of the bidisk. To answer this question, we study the geometry of boundary singularities of rational inner functions and characterize derivative $H^\p$-inclusion in terms of how the RIF's zero set makes contact with $\T^2.$

We are also interested in the integrability of higher derivatives, in particular the mixed partials $\frac{\partial^2 \phi}{\partial z_1 \partial z_2}$.  
One reason for this is that weighted area integrability of mixed partials is an equivalent condition for membership in {\it Dirichlet-type spaces} $\Da$. These are spaces of functions $f=\sum_{k,\ell}a_{k\ell}z_1^kz_2^{\ell}$ that are holomorphic in the bidisk and satisfy, for $\alpha \in \R$ fixed, the norm 
boundedness condition
\[\|f\|^2_{\Da}:=\sum_{k, \ell =0}^{\infty}(k+1)^{\alpha}(\ell+1)^{\alpha}|a_{k\ell}|^2<\infty.\]
Note that $\mathfrak{D}_0$ coincides with the Hardy space $H^2$, while $\mathfrak{D}_{-1}$ is the {\it Bergman space} $A^2$, and $\mathfrak{D}=\mathfrak{D}_{1}$ is the {\it Dirichlet space} of the bidisk. We will refer to the spaces $\Da$ as {\it weighted Dirichlet spaces} when we wish to emphasize that we are considering parameters $0<\alpha\leq 1$.

Dirichlet-type spaces are important for several reasons. The unweighted
Dirichlet norm leads to a Hilbert function space invariant under the action of automorphisms of the underlying domain \cite{Kap94}, Dirichlet-type 
spaces are connected with potential theory through the exceptional sets of their elements \cite{EKMRBook, Kap94}, and  these spaces provide a natural setting for multivariate operator theory in the form of shift operators \cite{BKKLSS15}. A characterization of cyclic polynomials for shifts acting on $\Da$ spaces was obtained in a recent series of papers \cite{BKKLSS15, KKRS}, but the contents and structures of Dirichlet spaces on the bidisk remain somewhat obscure. This leads to the following natural question:

\begin{question} \label{Q2} If $\phi$ is a rational inner function on $\mathbb{D}^2$, under what conditions is $\phi \in \Da$ for $0< \alpha <\infty$?
\end{question}

One can easily identify some rational inner functions that belong to all $\Da$ and whose derivatives are in $H^\p(\D^2)$ for every $\p>0$. The prime examples are 
\begin{itemize}
\item monomials $z_1^k z_2^{\ell}$, $k,\ell\in \mathbb{N}$,
\item finite Blaschke products in one variable and products of such functions,
\item rational inner functions that are continuous on the closed bidisk, such as
\[\psi(z_1,z_2)=\frac{3z_1z_2-z_1-z_2}{3-z_1-z_2}.\]
\end{itemize}
These examples illustrate the fact that RIFs can be thought of as generalizations of finite Blaschke products. Since finite Blaschke products are smooth on the closed unit disk they are trivially integrable along with their derivatives. Questions \ref{Q1} and \ref{Q2} become interesting when we consider rational inner functions with singularities such as
\begin{equation} \label{eqn:fav2} \phi(z_1,z_2)=\frac{2z_1z_2-z_1-z_2}{2-z_1-z_2}.\end{equation}
This function cannot be continuously extended to a neighborhood of the point $(1,1)\in  \T^2$, yet remains bounded due to the simultaneous vanishing of its numerator and denominator.  The results we obtain in this paper, when applied to this $\phi$, show that the partial derivatives $\frac{\partial \phi}{\partial z_1}, \frac{\partial \phi}{\partial z_2}  \in H^\p(\mathbb{D}^2)$ precisely when $\p < \frac{3}{2}$ and $\phi \in \Da$ precisely when $\alpha < \frac{3}{4}.$  Different RIFs will have different $\p$ and $\alpha$ intervals; the exact cut-off points will depend in a subtle way on the geometric nature of the function's boundary singularities.

\subsection{Overview and statement of results}
We now present the structure of our paper and summarize our main results. 
First in Section \ref{sec:prelim}, we collect information about RIFs on the bidisk and their counterparts, called \emph{rational inner Pick functions}, on the bi-upper half plane. In particular, we review key functional decompositions on the bidisk called \emph{Agler decompositions} and useful formulas on the bi-upper half plane called \emph{Nevanlinna representations}. 
Both types of formulas stem from Agler model theory, which originated in \cite{Ag90} and has been further developed in many recent papers including \cite{AMY12b, ATDY12, bk13, am14}.  We also review some standard ways to measure non-tangential regularity at points on the distinguished boundary $\T^2$. One theme of our paper is the usefulness of Agler model theory and its related formulas in understanding the singularities and associated boundary behavior of RIFs. Agler model theory actually fails for general rational inner functions on the polydisk $\mathbb{D}^n$ for $n \ge 3$, see \cite{Par70, Var71}, which is why we restrict our analysis to $\mathbb{D}^2$.

In Section \ref{sec:contact}, we review some important results from the theory of algebraic curves, including a geometric version of Puiseux's theorem. We use these to define the contact order of a RIF; this is the essential new object in our paper and provides a measurement of how singular a RIF is at points on $\mathbb{T}^2$.
Intuitively speaking, the contact order of a RIF measures the rate at which its zero set tends to $\T^2$ along one of the coordinate directions. Specifically,  denote the zero set of our rational inner function $\phi$
by $V$ and define the facial varieties
\[V_1=V\cap \left(\overline{\D}\times \T\right)  \quad \textrm{and}\quad V_2=V\cap \left(\T\times \overline{\D}\right).\]
Then the $z_i$-{\it contact order of $\phi$}  is the maximum number $K_i$ such that there exists a
sequence $\{w_k\} \subseteq V_i$ converging  to a singular point $\tau \in \T^2$ of $\phi$ and a positive constant $C$ such that
\[ \mathrm{dist}\big (w_k, \T^2 \big )\leq C \ \mathrm{dist} \big(w_k, \tau \big)^{K_i} \qquad \forall k \in \mathbb{N}.\]
If $\phi$ is continuous on $\overline{\mathbb{D}^2}$, its $z_1$- and $z_2$-contact orders are defined to be zero. Some work, detailed in Theorem \ref{lem:contact}, is required to show that the contact orders of a RIF are well defined.

In Sections \ref{sec:Dcontact}-\ref{sec:reg}, we study the $H^\p$-membership of RIF derivatives, in essence giving a complete answer to Question \ref{Q1}. Our first characterization rests on the geometry of the zero set of $\phi$ via contact order:

\begin{theorem*}[\ref{thm:contact}] Let $\phi = \frac{\tilde{p}}{p}$ be a rational inner function on $\mathbb{D}^2$. Then for $1 \le \p < \infty$, $\frac{\partial \phi}{\partial z_i} \in H^\p(\mathbb{D}^2)$ if and only if the $z_i$-contact order of $\phi$, denoted $K_i$, satisfies $K_i < \frac{1}{\p-1}.$  
\end{theorem*}

We then prove in Theorem \ref{thm:Hp} that no rational inner function with singularities on $\T^2$ belongs to $H^\p(\D^2)$ for $\p\geq \frac{3}{2}$, a result which is sharp by the example in \eqref{eqn:fav2}.  Then in a short section, we show that Agler decompositions can also be used to establish the non-integrability of RIF derivatives in the Hilbert space setting. Finally, Section \ref{sec:reg} details the surprising and counterintuitive observation that imposing higher non-tangential regularity on a RIF at a singularity (called being a $B^J$ point) decreases the range of exponents for which $\frac{\partial \phi}{\partial z_i}\in H^\p(\D^2)$:

\begin{corollary*}[\ref{cor:BJc}] Let $\phi = \frac{\tilde{p}}{p}$ be a rational inner function on $\mathbb{D}^2$ with a singularity at $\tau \in \mathbb{T}^2$.  If $\tau$ is a $B^J$ point of $\phi$, then at least one of $\frac{\partial \phi}{\partial z_1}$ and  $\frac{\partial \phi}{\partial z_2}$ fails to belong to $H^\p(\D^2)$ for $\p \ge \frac{1}{2J} +1.$
\end{corollary*}

 This is a corollary of Theorem \ref{thm:BJ}, which characterizes non-tangential regularity via contact order. Importantly, we also derive several results from this $H^\p$-theory  that give statements purely about the geometry of a RIF's singularities; for example, Corollary \ref{cor: geojulia}, which is a Geometric Julia Inequality of the same flavor as results in \cite{AMY12}, says that the $z_1$- and $z_2$-contact orders of a singular RIF are both at least $2$.

Sections \ref{sec:dirichlet}-\ref{sec:Dincl} are concerned with RIF membership in Dirichlet-type spaces. In Section  \ref{sec:dirichlet}, we review preliminary information about the structure of Dirichlet-type spaces. 
Then in Theorem \ref{thm:Dp}, we  
provide a partial answer to Question \ref{Q2} by showing how membership in $\Da$ for a certain range of exponents $\alpha$ can be deduced from $H^\p$-integrability, again via contact order. Specifically,

\begin{corollary*} [\ref{cor:Da}] Let $\phi = \frac{\tilde{p}}{p}$ be a rational inner function on $\mathbb{D}^2$. Then for $0 < \p < \infty$, if $\frac{\partial \phi}{\partial z_1} \in H^\p(\mathbb{D}^2)$ and $\frac{\partial \phi}{\partial z_2} \in H^\p(\mathbb{D}^2)$, then $\phi \in \mathfrak{D}_{\frac{p}{2}}$.
\end{corollary*}

This result follows from the more general Theorem \ref{thm:Dp}, which is stated for anisotropic Dirichlet spaces. Then by combining Agler decompositions with a two-variable version of the local Dirichlet integrals of Richter and Sundberg, we are able to prove in Theorem \ref{thm:1D} that members of a restricted class of RIFs fail to be in $\mathfrak{D}$, the Dirichlet space of the bidisk. Finally, Section \ref{sec:Dincl} gives some simple sufficient conditions for placing a RIF in $\Da$.  

Section \ref{sect: examples} provides a detailed treatment of several explicit RIFs  to illustrate our results, and the paper closes with a discussion of possible avenues for further research.

\section{Preliminaries: The Structure of Rational Inner Functions} \label{sec:prelim}

\subsection{Basics of RIFS on $\mathbb{D}^2$ and $\Pi^2$}
Let $\phi$ be a rational inner function on $\D^2$. Then 
\[\phi(z_1, z_2)=c z_1^Mz_2^N\frac{\tilde{p}(z_1,z_2)}{p(z_1,z_2)},\]
where  $p$ is a polynomial of degree $(m,n)$ with no zeros in the bidisk, $\tilde{p}(z_1, z_2):=z_1^mz_2^n\overline{p\left(\frac{1}{\bar{z}_1}, \frac{1}{\bar{z}_2}\right)}$
is the reflection of $p$, and $c$ is a unimodular constant. We can further assume that $p$ is atoral and hence, has finitely many zeros on $\mathbb{T}^2$. See \cite{AMS06} for more details. One can further deduce that $p$ has no zeros on $(\D \times \T) \cup (\T \times \D).$  Since $\phi$ can only have singularities at the zeros of $p$, it will only have finitely many singularities on $\overline{\D^2}.$

\begin{remark} In what follows, we study the $H^\p(\mathbb{D}^2)$ and the $\mathfrak{D}_{\vec{\alpha}}$ membership of rational inner functions and the geometry of their singularities. These properties are not affected by the existence (or nonexistence) of a monomial factor $z_1^Mz_2^N$. Thus, in the remainder of the paper, we prove statements about rational inner functions of the form $\phi = \frac{\tilde{p}}{p}$, where $\deg p=(m,n)$. Since $p$ and $\tilde{p}$ share no common factors (as $p$ is atoral), we will also say $\deg \phi = (m,n)$. The corresponding results for general rational inner functions will follow immediately from our results about $\frac{\tilde{p}}{p}$ rational inner functions.
\end{remark}

Let $\Pi$ denote the upper half plane. Then a two-variable \emph{Pick function} is a holomorphic function that maps $\Pi^2$ into $\Pi$. A Pick function is called \emph{inner} if $f(x_1,x_2) \in \mathbb{R}$ for almost every $(x_1, x_2) \in \mathbb{R}^2$. We will move between $\Pi$ and $\mathbb{D}$ via the following useful conformal maps:
\[
\begin{aligned}
& \alpha\colon \mathbb{D} \rightarrow \Pi, \ \alpha(z) := i \left[ \frac{1+z}{1-z} \right];  \quad
&& \widetilde{\alpha}\colon  \mathbb{D} \rightarrow \Pi, \ \widetilde{\alpha}(z) := i \left[ \frac{1-z}{1+z} \right]; \\
&\beta\colon \Pi \rightarrow \mathbb{D}, \ \beta(w):=\frac{w-i}{w+i};\quad
&& \tilde{\beta}\colon \Pi \rightarrow \mathbb{D}, \ \tilde{\beta}(w) := \frac{1+iw}{1-iw}. 
\end{aligned}
\]
 Then $\beta = \alpha^{-1}$ and $\tilde{\beta} = \tilde{\alpha}^{-1}$. 
If $\phi$ is a rational inner function on $\mathbb{D}^2$, then we can use these maps to transform $\phi$ into a rational inner Pick function $f$ on $\Pi^2$  
as follows:
\begin{equation} \label{eqn:pick} f(w):= \widetilde{\alpha}\big(\phi(\beta(w)\big) = i \left[ \frac{1-\phi(\beta(w))}{1+\phi(\beta(w))} \right],\end{equation}
where $\beta(w):=(\beta(w_1), \beta(w_2)).$  If $\tau_1, \dots, \tau_L$ are the singular points of $\phi$ on $\T^2$, then on $\mathbb{R}^2$, $f$ only has singularities at the points $\beta^{-1}(\tau_1), \dots, \beta^{-1}(\tau_L)$ and on the set $\{ \beta^{-1}(\zeta) : \phi(\zeta) =-1\}.$ Moreover, it is easy to see that on $\mathbb{R}^2 \setminus \{ \beta^{-1}(\tau_1), \dots, \beta^{-1}(\tau_L)\},$ the function $f$ is continuous as a map into $\mathbb{C}_{\infty}.$

Sometimes it is convenient to work directly with $\phi$, but other times, it will be convenient to work with $f$ via \eqref{eqn:pick}.  Heuristically, $f$ has the advantage that the boundary is flat and the range space is a cone,  and $\phi$ has the advantage that the boundary is compact and the range space is closed under multiplication. 


\subsection{Agler Decompositions}

Every rational inner function $\phi=\frac{\tilde{p}}{p}$ with $\deg p=(m,n)$ possesses a pair of \emph{Agler kernels}, namely a pair of 
positive semidefinite kernels $K_1, K_2\colon \mathbb{D}^2 \times \mathbb{D}^2 \rightarrow \mathbb{C}$ satisfying
\begin{equation} \label{eqn:agdecomp} 1 - \phi(z) \overline{ \phi(\lambda)} = (1-z_1 \bar{\lambda}_1) K_1(z,\lambda) + (1-z_2 \bar{\lambda}_2) K_2(z,\lambda) \quad \text{ for } z,\lambda \in \mathbb{D}^2.\end{equation}
Then \eqref{eqn:agdecomp} is called an \emph{Agler decomposition of $\phi$}.  These can be obtained from sums of squares of polynomials.  By \cite[Sections 5-6]{Kne15}, there are vectors of polynomials $\vec{E}_j$, $\vec{F}_j$, $\vec{G}$ (unique up to left multiplication by a unitary matrix) that satisfy both
\[ 
\begin{aligned} 
|p(z)|^2 -|\tilde{p}(z)|^2 &=& (1-|z_1|^2) \vec{E}_1(z)^*  \vec{E}_1(z)  + (1-|z_2|^2) \vec{F}_2(z)^*  \vec{F}_2(z) \\
&=& (1-|z_1|^2) \vec{F}_1(z)^*  \vec{F}_1(z)  + (1-|z_2|^2) \vec{E}_2(z)^*  \vec{E}_2(z)\\
&=& (1-|z_1|^2) \vec{F}_1(z)^*  \vec{F}_1(z)  + (1-|z_2|^2) \vec{F}_2(z)^*  \vec{F}_2(z) \\&& + (1-|z_1|^2)(1-|z_2|^2) \vec{G}(z)^*\vec{G}(z),
\end{aligned} 
\]
and the following condition:~if we write $\vec{E}_1(z) = E_1(z_2) \vec{\Lambda}_m(z_1)$ and $\vec{E}_2(z) = E_1(z_1) \vec{\Lambda}_n(z_2)$, where 
 each $\vec{\Lambda}_{\ell}(z_j)=(1, z_j, \ldots, z_j^{\ell-1})^T$ and each $E_j(z_j)$ is a square matrix-valued polynomial, then $\det E_1(z_1)$ and $\det E_2(z_2)$ are non-vanishing on $\mathbb{D}$. Then it is easy to see that the functions
\begin{eqnarray} \label{eqn:Akernels} K_1(z,\lambda): = \tfrac{1}{p(z)} \tfrac{1}{\overline{p(\lambda)}} \vec{E}_1(\lambda)^* \vec{E}_1(z) \ \text{ and } \ K_2(z,\lambda):=  \tfrac{1}{p(z)} \tfrac{1}{\overline{p(\lambda)}}  \vec{F}_2(\lambda)^* \vec{F}_2(z) \end{eqnarray}
are Agler kernels of  $\phi$. Moreover, as $ \vec{E}_1$ and $\vec{F}_2$ satisfy additional term and degree bounds, see the beginning of Section $5$ in \cite{Kne15}, we can expand these kernels and write
\begin{equation} \label{eqn:Akernels2} K_1(z,\lambda) =  \sum_{k=1}^m \tfrac{r_k}{p}(z)\overline{\tfrac{r_k}{p}(\lambda)} \ \ \text{ and } \ \ K_2(z,\lambda) = \sum_{j=1}^n \tfrac{q_j}{p}(z)\overline{\tfrac{q_j}{p}(\lambda)} ,\end{equation}
for polynomials $r_k$ with $\deg r_k \le (m-1, n)$ and polynomials $q_j$ with $\deg q_j \le (m, n-1)$. Here, one can take $\{ \frac{r_1}{p}, \dots, \frac{r_m}{p} \}$ and $\{ \frac{q_1}{p}, \dots, \frac{q_n}{p} \}$ to be orthonormal bases of $\mathcal{H}(K_1)$ and $\mathcal{H}(K_2)$ respectively, where $\mathcal{H}(K_j)$ is the Hilbert function space with reproducing kernel $K_j.$
See \cite{Ag90, Kne10, Bic12} and the references therein for additional material concerning Agler decompositions.


\subsection{B-points} All rational inner functions $\phi$ possess some non-tangential regularity at points on $\mathbb{T}^2$. Specifically, we define the following behavior:

\begin{definition} Let $\phi \in H^{\infty}(\mathbb{D}^2)$. Then a point $\tau = (\tau_1, \tau_2) \in \T^2$ is a called \emph{B-point for $\phi$} if there is a sequence $\{\lambda_n \} \subseteq \mathbb{D}^2$ converging to $\tau$ such that the corresponding sequence
\[ \frac{1-|\phi(\lambda_n)|}{1-\| \lambda_n \|} \ \ \ \text{ is bounded},\]
where $\| z \| = \max \{ |z_1|, |z_2|\}$ for any $z \in \mathbb{D}^2.$
\end{definition}

Agler, M\McC Carthy, and Young introduced B-points in \cite{AMY12} and used them, along with the related notion of a C-point,  to generalize the Julia-Carath\'eodory Theorem to $\mathbb{D}^2$. We also require the following (likely well-known) lemma:

\begin{lemma} \label{lem:Bpoint} If $\phi$ is a rational inner function on $\mathbb{D}^2$, then every $\tau \in \mathbb{T}^2$ is a $B$-point for $\phi.$
\end{lemma}

\begin{proof} Let $\tau=(\tau_1, \tau_2) \in \mathbb{T}^2$. Then $b(z): = \phi(z\tau_1, z\tau_2)$ is a finite Blaschke product on $\mathbb{D}$. 
By Carath\'eodory's Theorem in \cite[VI-4]{Sar94}, 
\[ \frac{ 1 - |b(z_n)|}{1-|z_n|} \text{ is bounded } \]
for every sequence $\{z_n\} \subseteq \mathbb{D}$ approaching one non-tangentially. Let $\{r_n\}$ with $0<r_n<1$ be a sequence approaching one, and
set $\lambda_n = (r_n \tau_1, r_n \tau_2) \subseteq \mathbb{D}^2$. Then $\{ \lambda_n\} $ converges to $\tau$ and 
\[  \frac{1-|\phi(\lambda_n)|}{1-\| \lambda_n \|}  =  \frac{ 1 - |b(r_n)|}{1-|r_n|} \ \ \text{ is bounded,}\]
so $\tau$ is a B-point of $\phi.$ 
\end{proof}


\subsection{Nevanlinna Representations and Spoke Regions} \label{ss:spokes} The Pick functions $f$ given by \eqref{eqn:pick} have very useful representations. By Lemma \ref{lem:Bpoint}, $(1,1)$ is a B-point of $\phi$ and so, by Proposition $2.8$ in \cite{AMY12},  there is a $\tau \in \mathbb{T}$ such that if $\{\lambda_n\} \subseteq \mathbb{D}^2$ converges non-tangentially to $(1,1)$, then $\{\phi(\lambda_n)\}$ converges to $\tau.$ Without loss of generality, assume $\tau=1$.

Using the language of \cite{ATDY12}, since $(1,1)$ is a B-point of $\phi$, then  $(\infty, \infty)$ is a carapoint of $f$ and since $\phi(1,1)=1$, 
$f(\infty,\infty) =0,$ where (as in \cite{ATDY12}) we interpret equality in terms of non-tangential limits. Then an application of  Theorem $1.11$ in \cite{ATDY12} gives

\begin{lemma} Let $f$ be defined as above. Then $f$ has a \emph{Type $1$ Nevanlinna Representation}, namely there exists a Hilbert space $\mathcal{H}$, a densely defined self-adjoint operator $A$ on $\mathcal{H}$ and a contraction $Y$ satisfying $0 \le Y \le I$ on $\mathcal{H}$, and a vector $\alpha \in \mathcal{H}$ such that 
\begin{equation} \label{eqn:Type1}  f(w) = \left \langle (A - W_Y)^{-1} \alpha, \alpha\right \rangle_{\mathcal{H}} \qquad \text{ for } w=(w_1, w_2) \in \Pi^2,\end{equation}
where $W_Y = w_1Y + w_2(I-Y).$ Moreover, since $f$ is rational inner, $\mathcal{H}$ can be chosen to be finite dimensional, so $A$ and $Y$ are matrices.
\end{lemma}

\begin{remark} Technically the statement of Theorem $1.11$ in \cite{ATDY12}  does not include the fact that $\mathcal{H}$ is finite dimensional. However, this can be deduced in a straightforward manner. Specifically, as given in \eqref{eqn:Akernels2}, every rational inner function has an Agler decomposition with kernels coming from finite sums of squares or equivalently, kernels that are reproducing kernels of finite-dimensional Hilbert spaces. Using this, one can show that the related Herglotz representation from Theorem $1.8$ in \cite{Ag90} is defined on a finite-dimensional Hilbert space. Then one can follow the computations in \cite{ATDY12} to show that for rational inner Pick functions as given above, these Herglotz representations give Type $1$ Nevanlinna representations corresponding to finite-dimensional Hilbert spaces.
\end{remark}

Then for these rational inner Pick functions, \eqref{eqn:Type1} extends to all $w \in \mathbb{C}^2$. 
We first use Type $1$ Nevanlinna representations to characterize the sets where these Pick functions can be ``large.''

\begin{lemma}[Spoke Lemma] \label{lem:Spoke} Let $f: \Pi^2 \rightarrow \Pi$ be a rational inner Pick function with a Type $1$ Nevanlinna representation as in \eqref{eqn:Type1}. Then for $r>0$
\begin{equation} \label{eqn:fbound} |f(w)| \le r \text{ whenever } \ \max_{1 \le i \le N}  \left|  \frac{1}{t_iw_1 + (1-t_i)w_2}\right| \le \min \left[ \frac{1}{2 \|A \|}, \frac{r}{2 \|\alpha\|^2} \right],\end{equation}
where $A, Y, \alpha$ are from \eqref{eqn:Type1} and $t_1, \dots, t_N$ are the eigenvalues of $Y$.
\end{lemma}

\begin{proof}
Fix $r >0$ and $w \in \mathbb{C}^2$, and assume $f$ satisfies the bounds in \eqref{eqn:fbound}. Then since 
\[ \|W^{-1}_Y \| = \max_{1 \le i \le N} \left|  \frac{1}{t_iw_1 + (1-t_i)w_2}\right|,\]
we can conclude $ \| A \|^k \| W_Y^{-1}\|^{k} \le 2^{-k}$ for $k \in \mathbb{N}.$  Then $(AW_Y^{-1}-I)$ is invertible and its Neumann series converges, so
\[ \begin{aligned}
| f(w)|  &= \left | \left \langle (A - W_Y)^{-1} \alpha, \alpha\right \rangle_{\mathcal{H}} \right |
& = \left | \left \langle -W_Y^{-1} \sum_{k=0}^{\infty} (A W^{-1}_Y)^k \alpha, \alpha \right \rangle_{\mathcal{H}}  \right |
& \le  \sum_{k=0}^{\infty} \| A \|^k \| W_Y^{-1}\|^{k+1} \|\alpha \|^2.
\end{aligned} 
\]
From this, we have $ |f(w)| \le  2 \| W_Y^{-1} \| \cdot \| \alpha \|^2$ and the other bound in \eqref{eqn:fbound} immediately  gives $|f(w)|\leq r.$ 
\end{proof}
The Spoke Lemma motivates the following additional definitions.
\begin{definition} Let $f$ be a rational  inner Pick function as in the Spoke Lemma and let $(x_1,x_2) \in \mathbb{R}^2$. Then,  $|f(x_1,x_2)| \le r$ as long as
\begin{equation} \label{eqn:spoke2} \max_{1 \le i \le N} \frac{1}{|t_ix_1 + (1-t_i)x_2|} \le \min \left[ \frac{1}{2 \|A \|}, \frac{r}{2 \|\alpha\|^2} \right]. \end{equation}
The set of points $(x_1,x_2)$  satisfying \eqref{eqn:spoke2} is called a \emph{spoke region of $f$} and is denoted by $SR(f,r)$. It is the shaded area in Figure \ref{fig:spoke}. 

Similarly, an \emph{$S(f,r)$-spoke} is the set of points $(x_1,x_2)$ satisfying 
\[ \frac{1}{|t_ix_1 + (1-t_i)x_2|} > \min \left[ \frac{1}{2 \|A \|}, \frac{r}{2 \|\alpha\|^2}\right] \text{for any eigenvalue $t_i$  of $Y$}.\] 
If $t_i = 1$ or $t_i=0$, this equation simplifies to $ -B < x_1 < B$ or $-B < x_2 < B$ respectively, for some positive constant $B$. If $t_i \ne 0, 1$, it 
simplifies to $mx_1-B <x_2 < mx_1+B$, for some negative $m$ and positive $B$. A \emph{nontrivial $(f,r)$-spoke} is a spoke that does not correspond to eigenvalues $t_i=0$ or $t_i=1$. This is equivalent to the spoke
not being parallel to one of the coordinate axes. 

Lastly, we say a point $p_0 \in \mathbb{R}^2$ is \emph{on a nontrivial $S(f,r)$-spoke} if $p_0$ lies on a nontrivial $(f,r)$-spoke and if we slice the spoke horizontally and vertically at $p_0$, 
one side of the resulting spoke does not intersect any other spoke (this includes spoke boundaries touching). The non-intersecting side of the spoke is called \emph{the cut-off spoke associated to} $p_0$ and is denoted $S(f,r,p_0)$. In Figure \ref{fig:spoke}, $p_0$ is on a nontrivial $S(f,r)$-spoke, but $q_0$ is not.
\end{definition}

\begin{center}
\begin{figure}[!htbp]
  \includegraphics[scale=.25]{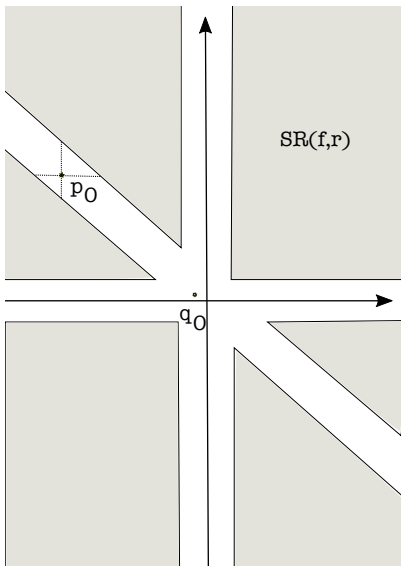}
  \caption{A Spoke Region of $f$}
  \label{fig:spoke}
\end{figure}
\end{center}



\subsection{$B^J$ Points}

Some rational inner functions $\phi$ possess non-tangential regularity at points on $\mathbb{T}^2$ that is stronger than the notion of a B-point. 
To account for this, we recall the definition of the intermediate L\"owner class, denoted $\mathcal{L}^{J_-}$ from \cite{Pas}. Technically, \cite{Pas}
defines $\mathcal{L}^{J_-}$ using behavior at $(\infty, \infty)$. Here, we send $(w_1, w_2) \mapsto (-\frac{1}{w_1}, -\frac{1}{w_2})$  in the conditions and so,
define  $\mathcal{L}^{J_-}$ using behavior at $(0,0)$ as follows: 

\begin{definition}Let $J$ be nonnegative integer, let $j=(j_1,j_2)\in \mathbb{N}^2$, and set $|j|=j_1+j_2$. A two-variable Pick function $g$ is in the \emph{intermediate L\"owner class at $(0,0)$}, denoted  $\mathcal{L}^{J_-}$, if $ \lim_{s\rightarrow 0}|g(is, is )| =0$
 and if there exists a multi-index system of real numbers $\{\rho_j\}_{|j| \le 2J-2}$ such that 
 \[ g(w) = \sum_{|j| \le 2J-2} \rho_j w^j +O \left( \| w \|^{2J-1} \right) \qquad \text{ non-tangentially}.\]
 We follow  \cite{Pas} and say \emph{non-tangentially} means that for each $c \in \R,$ the formula holds for all $w$ sufficiently small that also satisfy $\| \frac{1}{w} \|  \le c  \min \left\{ -\text{Im}(\frac{1}{w_1}) , -\text{Im}(\frac{1}{w_2})\right \}.$
\end{definition}
Now we can define the notion of a $B^J$ point:

\begin{definition}  Let $\phi$ be a rational inner function on $\mathbb{D}^2$ and fix $\tau =(\tau_1, \tau_2) \in \mathbb{T}^2$. Define 
\[  g_{\tau} (w):= \widetilde{\alpha}\left(\phi\left( \tau_1 \tilde{\beta}(w_1), \tau_2  \tilde{\beta}(w_2) \right) \right) = i \left[ \frac{1-\phi( \tau_1  \tilde{\beta}(w_1), \tau_2 \tilde{\beta}(w_2))}{1+\phi(\tau_1  \tilde{\beta}(w_1), \tau_2 \tilde{\beta}(w_2))} \right],\]
so $g_{\tau}(0) = \widetilde{\alpha}(\phi(\tau)).$ Then we say that $\tau$ is a \emph{$B^J$ point of $\phi$} if $g_{\tau}$ is in the intermediate L\"owner class  $\mathcal{L}^{J_-}$ at $(0,0)$.
\end{definition}
Informally speaking, $\tau \in \T^2$ being a $B^J$ point of $\phi$ can be thought of as saying that $\phi$ has a non-tangential expansion to order $2(J-1)$ at $\tau$, with remainder of order $2J-1$. In particular, the notion of $B^1$ point coincides with that of a B-point as defined in \cite{AMY12}. 

To simplify notation, we often assume that the point $\tau$ under consideration is $(1,1)$ so $g(w):= g_{(1,1)}(w) = f(- \frac{1}{w_1}, -\frac{1}{w_2})$, where $f$ is defined in \eqref{eqn:pick}
and has a Type $1$ Nevanlinna representation as in \eqref{eqn:Type1} with associated $\mathcal{H}$, $A$, $Y$, and $\alpha$.  Working through the proof of Proposition 3.2 in \cite{Pas}, one can prove the following:

\begin{lemma} \label{lem:BJ} Let $\phi$ be a rational inner function on $\mathbb{D}^2$ and assume $\tau = (1,1)$ is a $B^J$ point for $\phi$. Then
\[ g (w) = q(w) + \left \langle \left(A+ (W^{-1})_Y \right)^{-1} R_{J-1}(w),  R_{J-1}(\bar{w}) \right \rangle_{\mathcal{H}}, \]
where $q$ is a polynomial with real coefficients satisfying $q(0,0)=0$ and $R_{J-1}$ is a homogeneous vector-valued polynomial of total degree $J-1$.
\end{lemma}

\subsection{Polynomial ideals, intersection multiplicity,  and order of vanishing} \label{ssec:ideal}
Let $\phi = \frac{\tilde{p}}{p}$ be a rational inner function on $\mathbb{D}^2$  with $\deg \phi = (m,n)$. The study of $\frac{\partial \phi}{\partial z_1}, \frac{\partial \phi}{\partial z_2}  \in H^\p(\mathbb{D}^2)$ and $\phi \in \mathfrak{D}_{\alpha}$ for some $\p$ or $\alpha$ leads naturally to ideals of polynomials of this form:
\[ \mathcal{I}_p := \left \{ q \in \mathbb{C}[z_1, z_2]\colon \tfrac{q}{p} \in H^2(\mathbb{D}^2) \right\}\]
and the related ideal $\mathcal{I}_p^{\infty} := \left\{ q \in \mathbb{C}[z_1, z_2] \colon \frac{q}{p} \in H^{\infty}(\mathbb{D}^2) \right\}$.  When the polynomials of interest possess a degree bound, the following set is also useful: 
\[ \mathcal{P}_{p, (j,k)} := \{ q \in \mathbb{C}[z_1, z_2]\colon \tfrac{q}{p}  \in H^2(\mathbb{D}^2) \text{ and } \deg q \le (j,k) \}.\]
Some results about these are known. For example in \cite{Kne15}, Knese showed that the entries of $\vec{E}_1$, $\vec{F}_1$, and $\vec{F}_2$ generate $\mathcal{I}_p$ and computed the dimensions of many of the $\mathcal{P}_{p, (j,k)}$ sets.

The questions at hand rest on the behavior of $\phi$ at singularities on $\mathbb{T}^2,$ which are related to two additional topics. First, recall that a singularity of $\phi$ must occur at a zero of $p$.  Any zero of $p$ on $\mathbb{T}^2$ is also a zero of $\tilde{p}$, which leads to the study of the intersection of the zero sets of $p$ and $\tilde{p}$, denoted $\mathcal{Z}_p$, $\mathcal{Z}_{\tilde{p}}$, as algebraic curves. The ``amount'' of intersection at a common zero $\lambda$ of two polynomials $p$ and $q$ is called the \emph{intersection multiplicity} and is denoted $N_{\lambda}(p, q).$ See \cite{FulBook, Coxetal} for details about intersection multiplicity and its computation. If $\lambda \in \mathbb{T}^2$, one can show $N_{\lambda}(p,\tilde{p})$ is even \cite[Appendix C]{Kne15}. Moreover {B\'ezout's Theorem} implies
\[N(p, \tilde{p})=\sum_{\lambda \in \mathcal{Z}(p)\cap \mathcal{Z}(\tilde{p})}N_{\lambda}(p,  \tilde{p})= 2mn,\]
and so, the sum of the intersection multiplicities of common zeros on $\mathbb{T}^2,$ denoted $N_{\T^2}(p, \tilde{p})$, is  at most $2mn$. Thus, $p$ can have at most $mn$ distinct zeros on $\mathbb{T}^2$.

Recall that a polynomial $p\in \C[z_1, z_2]$ with $\deg p =(m,n)$ is said to \emph{vanish with order $M$ at $\lambda \in \C^2$} if \[p(\lambda-w)=\sum_{j=M}^{m+n}P_j(w),\] where the $P_j$ are homogeneous polynomials of degree $j$ and $P_M \not \equiv 0$. In \cite[Theorem 14.10]{Kne15}, Knese showed that if $p$ vanishes to order $M$ at a point $\tau \in \mathbb{T}^2$, then any $q\in \mathcal{I}_p$ vanishes to at least order $M$ at $\tau$. 

%


\section{Contact Order of a RIF} \label{sec:contact}
 In what follows, we recall some important results from the theory of algebraic curves. We use these to define the contact order of a rational inner function, which provides an important measure of how singular the rational inner function is at points on $\mathbb{T}^2$ and will play a key role in many of our later arguments and results.

\subsection{Puiseux's Theorem}
In what follows, we require information about the structure of the zero sets of two-variable polynomials. We collect the
needed information here. For additional information or proofs of the stated results, see \cite[Chapters $6$ and $7$]{fischer01}.

\begin{definition} Recall the following standard definitions: $\mathbb{C}[z_1, z_2]$ denotes the polynomials in $z_1, z_2$ with complex coefficients, 
$\mathbb{C}\left \langle z_1, z_2\right \rangle$ denotes the power series in $z_1, z_2$ with complex coefficients
that converge in a neighborhood of the origin, and $\mathbb{C}\left \langle z_1 \right \rangle[z_2]$ denotes the elements of 
$\mathbb{C}\left \langle z_1, z_2\right \rangle$ with a largest power in $z_2.$ A function $\alpha \in \mathbb{C}\left \langle z_1, z_2\right \rangle$ 
is called a \emph{unit} if $\alpha(0,0) \ne 0.$

We say that $g \in \mathbb{C}\left \langle z_1, z_2\right \rangle$ is \emph{general in $z_2$ of order $N$} if writing $g(0,z_2) = \sum_{j=0}^{\infty} a_j z_2^j$, 
\[  N = \min \{ j: a_j \ne 0\}.\]
Observe that $N\ge 1$ as long as $g(0,0)=0$ and $z_1$ does not divide $g$.

Furthermore, we say $q \in \mathbb{C}\left \langle z_1 \right \rangle[z_2] $ is a \emph{Weierstrass polynomial of degree $N$} if 
$q(z) = \sum_{j=0}^N q_j(z_1) z_2^j$ satisfies
\[ q_0(0)= \dots =q_{N-1}(0) = 0 \text{ and } q_N \equiv 1.\]
Observe that a Weierstrass polynomial of degree $0$ is a constant function.  
\end{definition}
In what follows, we reduce from general functions  to Weierstrass polynomials using the Weierstrass Preparation Theorem \cite[pp.~$107$]{fischer01}:

\begin{theorem}[Weierstrass Preparation Theorem] \label{thm:WP} Let $g \in \mathbb{C}\left \langle z_1, z_2 \right \rangle$ be general in $z_2$ of order $N$.
Then, there is a neighborhood $U$ of $(0,0)$, a Weierstrass polynomial $q$ of degree $N$,  and a unit $\alpha \in \mathbb{C}\left \langle z_1, z_2\right \rangle$
 such that $g = \alpha q$ in $U$.
\end{theorem}

We will combine this with Puiseux's Theorem. The version presented here appears in \cite[ pp.~$136$]{fischer01}:

\begin{theorem}[Geometric Puiseux's Theorem] \label{thm:P} Let $r \in  \mathbb{C}\left \langle z_1 \right \rangle[z_2] $ be an irreducible Weierstrass polynomial of degree $N \ge 1$. Then there exists a $\delta>0$ such that if 
\[ V: = \left \{ t \in \mathbb{C} : |t| < \delta^{\frac{1}{N}} \right\} \ \ \text{ and } \ \ \widetilde{V} := \left \{ z_1 \in \mathbb{C} : |z_1| < \delta \right \} \]
and if we define the restricted zero set
\[ C := \left \{ (z_1, z_2) \in \widetilde{V} \times \mathbb{C} : r(z_1, z_2) = 0 \right \},
\]
then there is a $\psi \in \mathbb{C}\left \langle t \right \rangle$ that converges in $V$  and has the following properties:
\begin{itemize}
\item[(i)] $r(t^N, \psi(t)) = 0$ for all $t \in V$.
\item[(ii)] $\Psi: V \rightarrow C$, sending $t \mapsto (t^N, \psi(t))$ is a bijection.
\end{itemize}
 \end{theorem}

By combining Theorems \ref{thm:WP} and \ref{thm:P}, one can give a description of the zero sets of two-variable holomorphic functions near $(0,0)$. Here are the details.

\begin{remark} \label{rem:PT} Let $g \in \mathbb{C}\left \langle z_1, z_2 \right \rangle$ be general in $z_2$ of order $N \ge 1$.  By Theorem \ref{thm:WP}, there is a Weierstrass polynomial $q$ 
of degree $N \ge 1$, a neighborhood of $(0,0)$, and a unit $\alpha$ with  $g = \alpha q$ in this neighborhood. By shrinking the neighborhood if necessary, we can assume it is of the form $U_1 \times U_2$, for $U_1, U_2$ neighborhoods of zero. 
By Proposition 2 and the subsequent lemma in \cite[pp. $118$]{fischer01}, there is a unit $\beta$ and a set of irreducible Weierstrass polynomials 
$r_1, \dots, r_{L}$ such that $g = \beta r_1 \cdots r_L$ on some neighborhood where $\beta$ is nonvanishing.  By shrinking $U_1, U_2$ if necessary, we can assume that 
this neighborhood is still $U_1 \times U_2.$
We can also assume each $r_{\ell}$ has positive degree $N_{\ell}$ because otherwise, it could be absorbed into $\beta.$

Now for $\ell=1, \dots, L$,  apply Theorem \ref{thm:P} to $r_{\ell}$ to 
obtain a  radius $\delta_{\ell}>0,$ neighborhoods $V_{\ell}$ and $\widetilde{V}_{\ell}$ containing zero defined using $\delta_{\ell}$, and  a power series $\psi_{\ell}$ converging in $V_{\ell}$ and satisfying (i) and (ii). Set $\delta = \min \{\delta_{\ell}: 1 \le \ell \le L\}$ and define the sets
 \[ \widetilde{V} := \{ z_1 \in \mathbb{C}: |z_1| < \delta\} \text{ and } \widehat{V}_{\ell} := \{ t \in \mathbb{C}: |t|^{N_{\ell}} < \delta \}.\]
 Further set
  \[ C_{\ell} := 
  \left\{ (z_1, z_2) \in \widetilde{V} \times \mathbb{C} :
 r_{\ell}(z_1, z_2) = 0 \right\}.
 \]
 By Theorem \ref{thm:P}, each $C_{\ell}$ is exactly the collection of points $\big(t^{N_{\ell}}, \psi_{\ell}(t)\big)$, for $t \in \widehat{V}_{\ell}$. 
 Set $\widetilde{U}_1 := U_1 \cap \widetilde{V}$ and $\widetilde{U}_2 := U_2.$ Then
 \[
C := \left\{ (z_1, z_2) \in \widetilde{U}_1 \times \widetilde{U}_2:
 g(z_1, z_2) = 0 \right\} =  \bigcup_{\ell=1}^L \left( C_{\ell} \cap \left(U_1 \times U_2 \right) \right).
\]
Then for each $\ell$, we can set $z_1 = t^{N_{\ell}}$ and conclude that all points in $C$ are of the form
\[ \left(z_1, \psi_{1} \left(z_1^{\frac{1}{N_1}}\right) \right), \dots, \left (z_1, \psi_{L}\left(z_1^{\frac{1}{N_L}}\right)\right) \] 
as long as we allow all values for each $z_1^{\frac{1}{N_{\ell }}}.$ Furthermore, if we shrink $\widetilde{U}_1$ and $\widetilde{U}_2$, then the only zeros of $g$ in $\widetilde{U}_1 \times \widetilde{U}_2$ will occur on curves satisfying $\psi_{\ell}(0)=0$. Given that restriction, each set of curves $\big(z_1, \psi_{\ell}\big( z_1^{\frac{1}{N_{\ell}}}\big) \big)$ will be in $C$ for $z_1$ sufficiently close to zero. 

\end{remark}

\subsection{Contact order of a singularity}

Let $\phi=\frac{\tilde{p}}{p}$ be a rational inner function on $\mathbb{D}^2$ with $\deg \phi = (m,n)$ and fix $\zeta_2 \in \mathbb{T}$. Define the one-variable function $\phi_{\zeta_2}(z_1):= \phi(z_1, \zeta_2).$ Then $\phi_{\zeta_2}$ is a Blaschke product. If $\phi$ does not have a singularity on $\T^2$ with $z_2$-coordinate $\zeta_2$, then $\deg \phi_{\zeta_2} =m$ and $\phi_{\zeta_2}$ has $m$ zeroes $\alpha_1, \dots, \alpha_m$ in $\mathbb{D}.$ In this case, define the quantity 
\begin{equation} \label{eqn:dist} \epsilon(\phi, \zeta_2) : = \min  \big \{ 1- |\alpha_j|: 1 \le j \le m \big \} = \min \big  \{ 1 -|z_1|: \tilde{p}(z_1,\zeta_2) = 0 \big\}.\end{equation}
Observe that $\epsilon(\phi, \zeta_2)$ measures the distance from the zero set of $\phi_{\zeta_2}$ to $\mathbb{T}.$ If $\phi$ has a singularity on $\T^2$ with $z_2$-coordinate $\zeta_2$, define $\epsilon(\phi, \zeta_2):=0.$
We require the following result, which will allow us to rigorously define the contact order of a rational inner function.

\begin{theorem} \label{lem:contact} Let $\phi =\frac{\tilde{p}}{p}$ be a rational inner function on $\mathbb{D}^2$ with $\deg \phi = (m,n)$ and a singularity on $\T^2$ with  $z_2$-coordinate $\tau_2.$ Then there exists a rational number $K> 0$ such that 
\[   \epsilon(\phi, \zeta_2)  \approx | \tau_2- \zeta_2 |^K,\]
for all $\zeta_2 \in \mathbb{T}$ in a neighborhood $U$ of $\tau_2$.
The number $K$ is called the \emph{$(z_1, \tau_2)$-contact order of $\phi$.} 
\end{theorem}

 Now we prove Theorem \ref{lem:contact}:

\begin{proof}  
\textbf{Step 1.} We first need to reduce the problem to a local setting. Specifically, let  $(\tau_{11}, \tau_2), \dots, (\tau_{1J}, \tau_2)$ be the singularities of $\phi$ on $\T^2$ with $z_2$-coordinate $\tau_2.$ Fix a positive $r < \frac{1}{2}$. For each $j$, let $B_{r} (\tau_{1j})$ denote the part of the disk with radius $r$ and center $\tau_{1j}$ contained in $\overline{\mathbb{D}}.$ Wherever they are well defined, consider the following variants of $\epsilon(\phi, \zeta_2)$:
\[ \epsilon_{(\tau_{1j}, \tau_2)} (\phi, \zeta_2) :=  \min \left  \{ 1 -|z_1|: \tilde{p}(z_1,\zeta_2) = 0 \text{ and } z_1 \in B_{r} (\tau_{1j}) \right\}, \qquad 1 \le j \le J \]
and $ \epsilon_{\tau_2}(\phi, \zeta_2):= \min \{   \epsilon_{(\tau_{1j}, \tau_2)}(\phi,\zeta_2) : 1 \le j \le J\}.$ Since the zeros of $\tilde{p}(\cdot, \zeta_2)$ are continuous functions of $\zeta_2$  (see the proof of \cite[Theorem $3.7$]{bg17}) and since $\tilde{p}(\cdot, \zeta_2)$ only vanishes in $\overline{\D}$, there must be a neighborhood $V_r \subseteq \T$ of $\tau_2$ such that for every $\zeta_2 \in V_r$, each $\epsilon_{(\tau_{1j}, \tau_2)}(\phi, \zeta_2)$ is well defined 
\[ \epsilon_{\tau_2}(\phi, \zeta_2) = \epsilon(\phi,\zeta_2).\]
Thus, we can work directly with each $\epsilon_{(\tau_{1j}, \tau_2)} (\phi, \zeta_2)$ separately. Without loss of generality, fix $j$ and assume $(\tau_{1j}, \tau_2)=(1,1).$
Then we need only find a neighborhood $U \subseteq \T$ of $1$, a small $r>0$, and a rational number $K$ satisfying
\begin{equation} \label{eqn:11} \min \left  \{ 1 -|z_1|: \tilde{p}(z_1,\zeta_2) = 0 \text{ and } z_1 \in B_{r} (1) \right\} \approx | \zeta_2 -\tau_2|^K \quad \forall \zeta_2 \in U. \end{equation}
\\

\noindent \textbf{Step 2.} To simplify later computations, we reduce \eqref{eqn:11} to a problem on $\Pi$ near $(0,0)$. Recall that $\tilde{\beta}(w):=\frac{1+iw}{1-iw}:\Pi \rightarrow \mathbb{D}$ is a conformal map with $\tilde{\beta}(0)=1$ and define the polynomial
\[ q(w_1, w_2) : = (1-iw_1)^m(1-iw_2)^n \tilde{p}\left (\tilde{\beta}(w_1), \tilde{\beta}(w_2) \right) \quad \forall (w_1, w_2) \in \mathbb{C}^2.\]
Assume there is a small $r>0$, a rational number $K$, and a neighborhood $\widetilde{U} \subseteq \R$ of $0$ satisfying
\begin{equation} \label{eqn:qform}   \min \left \{ \text{Im}(w_1)  : q(w_1,x_2) = 0 \text{ and } \tilde{\beta}(w_1) \in B_{r}(1) \right\} \approx | x_2 |^{K} \quad \forall  x_2 \in \widetilde{U}.\end{equation}
We will show that \eqref{eqn:qform} implies \eqref{eqn:11}. First observe that if $\tilde{\beta}(w_j) \in B_{\frac{1}{2}}(1)$, then $|w_j|$ is bounded and $\frac{1}{|1-iw_j|} \approx 1.$ This allows one to compute
\[
\begin{aligned}
1 - |\tilde{\beta}(w_1)| \approx 1 - |\tilde{\beta}(w_1)|^2  &= \frac{4 \cdot  \text{Im}(w_1)}{|1-iw_1|^2} \approx \text{Im}(w_1); \\
\left | 1 - \tilde{\beta}(w_2) \right | &= \frac{2 |w_2|}{|1-iw_2|} \approx |w_2|.
\end{aligned}
\]
Define $U:= \tilde{\beta}(\widetilde{U})$ and fix $\zeta_2 \in U$. Then since each $(1-iw_j)$ is nonvanishing on $\overline{\Pi}$, 
we can use \eqref{eqn:qform} and the above estimates to conclude
\[
\begin{aligned}
&\left  \{ 1 -|z_1|: \tilde{p}(z_1,\zeta_2) = 0 \text{ and } z_1 \in B_{r} (1) \right\} \\
& \qquad \qquad \qquad = \left \{ 1 - |\tilde{\beta}(w_1)| : q\left (w_1, \tilde{\beta}^{-1}(\zeta_2)\right)  =0 \text{ and } \tilde{\beta}(w_1) \in B_{r}(1) \right \} \\
& \qquad \qquad \qquad  \approx \left \{  \text{Im}(w_1) : q\left (w_1, \tilde{\beta}^{-1}(\zeta_2)\right)  =0 \text{ and } \tilde{\beta}(w_1) \in B_{r}(1) \right \} \\
& \qquad \qquad \qquad  \approx  | \tilde{\beta}^{-1}(\zeta_2) |^{K} \\
& \qquad \qquad \qquad  \approx |1- \zeta_2|^K  \quad \forall  \zeta_2 \in U,
\end{aligned}
\] 
which establishes \eqref{eqn:11}.\\

\noindent \textbf{Step 3.} Now we prove \eqref{eqn:qform} using Puiseux's Theorem. Since $q(0,0)=0$, Remark \ref{rem:PT} gives neighborhoods $\widetilde{U}_1$ and $\widetilde{U}_2$ of $0$, natural numbers $L, N_1, \dots, N_L$, and power series $\psi_1, \dots, \psi_L$ converging on a neighborhood of zero and satisfying $\psi_{\ell}(0)=0$ such that the set of points 
\[ C := \left \{ (w_1, w_2) \in \widetilde{U}_1 \times \widetilde{U}_2: q(w_1, w_2) =0 \right\} \]
is contained in the set of points
\[\left \{  \Big(\psi_{1}\Big(w_2^{\frac{1}{N_1}} \Big), w_2 \Big), \dots,  \Big(\psi_{L} \Big(w_2^{\frac{1}{N_L}} \Big), w_2 \Big): w_2 \in \widetilde{U}_2\right \}, \]
where each $ w_2^{\frac{1}{N_{\ell}}}$ is $N_{\ell}$-valued and 
the curves $\big(\psi_{\ell} \big(w_2^{\frac{1}{N_{\ell}}} \big), w_2 \big)$ are contained in $C$ for sufficiently small $w_2$.
Write each 
\[ \psi_{\ell}(t)  = \sum_{k=0}^{\infty} a_{\ell k} t^k,\]
and let $\eta_1, \dots, \eta_{N_{\ell}}$ denote the $N_{\ell}^{th}$ roots of unity. Define the index set $\mathcal{J} = \{ (\ell, j): 1 \le \ell \le L, 1 \le j \le N_{\ell} \}.$
Then we can write the $w_1$ coordinates of points in $C$ as 
\[ W^{\ell,j}_1(w_2)  = \psi_{\ell} \left( \Big |w_2 ^{\frac{1}{N_{\ell}}} \Big| \eta_j \right) = \sum_{k=0}^{\infty} a_{\ell k} \left( \Big |w_2 ^{\frac{1}{N_{\ell}}} \Big| \eta_j\right)^k \quad \text{ for }(\ell, j) \in \mathcal{J}. \]
Now restrict to $x_2 \in \mathbb{R}$ and observe that for any $(\ell, j) \in \mathcal{J}$, the function $\text{Im } \Big( W^{\ell,j}_1(x_2) \Big) \not \equiv 0,$
because otherwise $\tilde{p}$ would vanish on an entire curve contained in $\mathbb{T}^2$. Let $K_{\ell, j}$ be the smallest integer $k$ such that 
$\text{Im} \left( a_{\ell k} \cdot \eta_j^k \right) \ne 0.$ Since each $\psi_{\ell}(0)=0$, we know every $K_{\ell,j} >0$ and so for $x_2$ sufficiently close to zero,
\[ \text{Im } \left( W^{\ell,j}_1(x_2) \right) \approx \Big |x_2 ^{\frac{K_{\ell,j}}{N_{\ell}}} \Big| \cdot \text{Im} \left( a_{\ell K_{\ell,j}}\cdot \eta_j^{K_{\ell,j}} \right) \approx \Big  |x_2 ^{\frac{K_{\ell,j}}{N_{\ell}}} \Big|.\]
Define $K := \max \left\{ \frac{K_{\ell, j}}{N_{\ell}}: (\ell, j) \in \mathcal{J}\right \}$.  Choose $r>0$ so that $B_r(1) \subseteq \tilde{\beta}(\widetilde{U}_1).$
Then there is a neighborhood $\widetilde{U} \subseteq \R$ of $0$ such that for all $x_2 \in \widetilde{U}$, 
\[
\begin{aligned}
 \min \left \{ \text{Im}(w_1)  : q(w_1,x_2) = 0 \text{ and } \tilde{\beta}(w_1) \in B_{r}(1) \right\} 
& = \min \left \{ \text{Im}\left ( W^{\ell,j}_1(x_2)\right) :(\ell, j) \in \mathcal{J}  \right \} \\
 & \approx \min\left \{  |x_2|^{\frac{K_{\ell,j}}{N_{\ell}}}  : (\ell, j) \in \mathcal{J} \right\} \\
  &= |x_2|^{K},
  \end{aligned}
  \]
  which proves \eqref{eqn:qform} as required.
\end{proof}

 We can now officially define the contact order of a rational inner function:

\begin{definition} \label{def:contact} Let $\phi =\frac{\tilde{p}}{p}$ be a rational inner function on $\mathbb{D}^2$ with $\deg \phi = (m,n)$ and let $\tau_{21}, \dots, \tau_{2J}$ denote the distinct $z_2$-coordinates of the singularities of $\phi$. Then the \emph{$z_1$-contact order of $\phi$} is the number $K_1$ defined by 
\[ K_1 := \max \Big \{  \text{$(z_1, \tau_{2j})$-contact order of $\phi$}:  1 \le j \le J \Big\},\]
where the $(z_1, \tau_{2j})$-contact order of $\phi$ is defined in Theorem \ref{lem:contact}. 
If $\phi$ has no singularities on $\overline{\mathbb{D}^2}$, then we set $K_1 =0.$
Then for each $j$, there is a small neighborhood $U_j\subseteq \T$ of $\tau_2$ such that  
\[  \epsilon(\phi, \zeta_2)  \gtrsim | \tau_{2j}- \zeta_2 |^{K_1} \quad \forall \zeta_2 \in U_j.\] 
And, there is at least one neighborhood, say $U_1$, where the equivalence holds:
\[  \epsilon(\phi, \zeta_2) \approx  | \tau_{21}- \zeta_2 |^{K_1} \quad \forall   \zeta_2 \in U_1 .\]
The \emph{$z_2$-contact order of $\phi$}, denoted $K_2$, is defined analogously. 
\end{definition}

An elementary argument shows that the above definition agrees with the definition from the introduction.
\section{RIF Derivatives in $H^\p(\mathbb{D}^2)$: Contact Order} \label{sec:Dcontact}

The goal of this section is to prove the following theorem:

\begin{theorem} \label{thm:contact} Let $\phi = \frac{\tilde{p}}{p}$ be a rational inner function on $\mathbb{D}^2$. Then for $1 \le \p < \infty$, $\frac{\partial \phi}{\partial z_i} \in H^\p(\mathbb{D}^2)$ if and only if the $z_i$-contact order of $\phi$, denoted $K_i$, satisfies $K_i < \frac{1}{\p-1}.$  
\end{theorem}

We first require the following two lemmas: 

\begin{lemma} \label{lem:blaschke} Consider a finite Blaschke product $b(z) := \prod_{j=1}^n b_{\alpha_j}(z), \text{ with } b_{\alpha_j}(z) = \frac{z-\alpha_j}{1-\bar{\alpha}_j z}$ for $\alpha_j \in \mathbb{D}.$
Then the modulus of the derivative of $b$ satisfies
\begin{equation} \label{eqn:derivnorm}  |b'(\zeta)| = \frac{ b'(\zeta)}{b(\zeta)} \zeta = \sum_{j=1}^n \left |b'_{\alpha_j}(\zeta) \right | \qquad \text{ for } \zeta \in \mathbb{T}. \end{equation}
\end{lemma}

\begin{proof} We first consider a single Blaschke factor 
\[ b_{\alpha} (z) : = \frac{z-\alpha}{1-\bar{\alpha}z}, \quad \text{ for some } \alpha \in \mathbb{D}.\]
Then we can compute 
\[ b_{\alpha}'(z) = \frac{1-|\alpha|^2}{(1-\bar{\alpha} z)^2} \  \text{ and  for $\zeta \in \T,$ } \ \frac{ b_{\alpha}'(\zeta)}{b_{\alpha}(\zeta)} \zeta = \frac{1-|\alpha|^2}{(1-\bar{\alpha}\zeta)(\zeta-\alpha)\bar{\zeta}} = \frac{1-|\alpha|^2}{|\zeta-\alpha|^2} >0.\]
Since $|b_{\alpha}'(\zeta) |  = \left | \frac{ b_{\alpha}'(\zeta)}{b_{\alpha}(\zeta)} \zeta \right |$ for $\zeta \in \mathbb{T}$, this
 immediately implies that \eqref{eqn:derivnorm} holds for Blaschke factors. Observe that
\[ b'(z) = \sum_{j=1}^n b_{\alpha_j}'(z) \left( \prod_{k \ne j} b_{\alpha_k}(z)\right)\]
and so for $\zeta \in \mathbb{T}$, 
\[ \frac{b'(\zeta)}{b(\zeta)} \zeta =  \sum_{j=1}^n \frac{ b_{\alpha_j}'(\zeta)}{b_{\alpha_j}(\zeta)} \zeta = \sum_{j=1}^n \left | b'_{\alpha_j}(\zeta) \right|,\]
which implies \eqref{eqn:derivnorm}.
\end{proof}

\begin{lemma} \label{lem:Hpderiv} Let $\phi = \frac{\tilde{p}}{p}$ be a rational inner function on $\mathbb{D}^2$ with $\deg \phi = (m,n)$ and fix $\zeta_2 \in \mathbb{T}$ such that $\tilde{p}(\cdot, \zeta_2)$ does not vanish on $\T$. Then $\phi_{\zeta_2}:=\phi(\cdot, \zeta_2)$ satisfies
\[ \left \| \phi'_{\zeta_2} \right \|_{H^\p(\mathbb{D})}^\p \approx \epsilon(\phi,\zeta_2)^{1-\p} \quad \text{ for } \p \ge 1,\]
where $ \epsilon(\phi,\zeta_2)$ is the distance from the zero set of $\phi_{\zeta_2}$ to $\mathbb{T}$, rigorously defined in \eqref{eqn:dist}.
\end{lemma}

\begin{proof}  First as $\deg \phi_{\zeta_2} = m$, we can write the finite Blaschke product $\phi_{\zeta_2} = \prod_{j=1}^m b_{\alpha_j}$, for $\alpha_1, \dots, \alpha_m \in \D.$ Then $ \epsilon(\phi,\zeta_2) = \min \{ 1-|\alpha_j| : 1\le j \le m\}.$
By Lemma \ref{lem:blaschke}, we have
\begin{equation} \label{eqn:Hpnorm} \left \| \phi'_{\zeta_2} \right \|_{H^\p(\mathbb{D})}^\p =
  \frac{1}{2\pi}\int_{\mathbb{T}} \left | \phi'_{\zeta_2}(\zeta_1) \right|^\p  |d \zeta_1| = \frac{1}{2\pi}  \int_{\mathbb{T}} \left | \sum_{j=1}^m \left |b'_{\alpha_j}(\zeta_1) \right | \right|^\p  |d \zeta_1| \approx  \sum_{j=1}^m \left \| b'_{\alpha_j} \right\|_{H^\p(\mathbb{D})}^\p,\end{equation}
where the implied constant depends on $\p$ and $m$. So, it makes sense to first study $\| b'_{\alpha} \|_{H^\p(\mathbb{D})}^\p $ for $\alpha \in \mathbb{D}$. Observe that if 
\[ b_{\alpha}(z) = \frac{z-\alpha}{1-\bar{\alpha}z}, \text{ then } |b_{\alpha}'(z)| = \frac{1 -|\alpha|^2}{|1-\bar{\alpha} z|^2} \approx  \frac{1 -|\alpha|}{|1-\bar{\alpha} z|^2}.\]
Then it should be clear that 
\begin{equation} \label{eqn:bd_1}  \| b'_{\alpha} \|_{H^\p(\mathbb{D})}^\p =\frac{1}{2\pi} \int_{\mathbb{T}} \frac{ (1-|\alpha|^2)^\p}{|1-\bar{\alpha}\zeta |^{2\p}} |d\zeta| \approx (1-|\alpha|)^\p \int_{\mathbb{T}} \frac{1}{|1-\bar{\alpha}\zeta|^{2\p}} |d\zeta| \end{equation}
only depends on $|\alpha|$ and indeed, by a change of variables, we can assume $\alpha = |\alpha|$. To emphasize that $\alpha$ is now real, we will write $\alpha =t$. 
We will show $\eqref{eqn:bd_1} \approx (1-t)^{1-\p}.$ Now if $t <\frac{1}{2}$, then this is basically immediate. So, assume $t \in [\frac{1}{2}, 1).$ 
Writing $\zeta = e^{i\theta}$, we have 
\[ \frac{1}{|1-\bar{\alpha} \zeta|^2} =  \frac{1}{(1-t\zeta)(1-t\bar{\zeta})} = \frac{1}{1 +t^2 -2t \cos \theta} \approx \frac{1}{(1-t)^2 + t\theta^2} \approx \frac{1}{(1-t)^2 + \theta^2},\]
for $\theta \in (-b, b)$, with $b$ small but independent of the choice of $t \in [\frac{1}{2},1]$.
 Then an application of trigonometric substitution with $\theta = (1-t) \tan x$ gives
\[ 
\begin{aligned}
\eqref{eqn:bd_1} & \approx (1-t)^\p \int_{[0,2\pi] \setminus (-b,b)} \frac{1}{|1- t e^{i\theta} |^{2\p}} d \theta  + (1-t)^\p \int_{-b}^b \frac{1}{\left( (1-t)^2 +\theta^2 \right)^\p} \ d  \theta \\
&\approx (1-t)^\p +(1-t)^\p \int_{\text{arctan}\left( \frac{-b}{1-t}\right)}^{\text{arctan}\left( \frac{b}{1-t}\right)} \frac{(1-t) \sec^2 (x)}{(1-t)^{2\p} \sec^{2\p}(x)} \ dx \\
&= (1-t)^\p + (1-t)^{1-\p}  \int_{-\text{arctan}\left( \frac{b}{1-t}\right)}^{\text{arctan}\left( \frac{b}{1-t}\right)} \cos^{2\p-2}(x) \ dx\\
& \approx (1-t)^\p + (1-t)^{1-\p} \\
& \approx (1-|\alpha|)^{1-\p},  
\end{aligned}
\]
since $\mathfrak{p} \ge 1$, where none of the implied constants depend on $t$.  By \eqref{eqn:Hpnorm} and again using $\p\ge 1,$ we can compute
\[  \left \| \phi'_{\zeta_2} \right \|_{H^{\mathfrak{p}}(\mathbb{D})}^\p \approx \sum_{j=1}^m (1- |\alpha_j|)^{1-\p} \approx \max \{ (1- |\alpha_j|)^{1-\p} \} = \left( \min \{ 1-|\alpha_j| \} \right)^{1-\p} =  \epsilon(\phi,\zeta_2)^{1-\p},\]
as desired.
\end{proof}

The proof of Theorem \ref{thm:contact} is almost immediate:

\begin{proof} We prove the result for $i=1.$ Assume $\p \ge 1$. Then by Lemma \ref{lem:Hpderiv},
\begin{equation} 
 4\pi^2\left \| \tfrac{ \partial \phi}{\partial z_1} \right \|_{H^\p(\mathbb{D}^2)}^\p = \iint_{\mathbb{T}^2} \left | \tfrac{ \partial \phi}{\partial z_1} (\zeta_1, \zeta_2 )\right |^\p |d \zeta_1| |d \zeta_2| 
 = 2 \pi \int_{\mathbb{T}} \left \| \phi_{\zeta_2}' \right \|_{H^\p(\mathbb{D})}^\p |d\zeta_2 | 
 \approx \int_{\mathbb{T}} \epsilon(\phi, \zeta_2)^{1-\p} |d\zeta_2|.\end{equation}
Let $K_1$ denote the $z_1$-contact order of $\phi$ and let $\tau_{21}, \dots, \tau_{2J}$ denote the distinct $z_2$-coordinates of the singularities of $\phi$. Then by the discussion in Definition \ref{def:contact}, there is a neighborhood $U_j \subseteq \T$ around each $\tau_{2j}$
 such that
\[  \epsilon(\phi, \zeta_2)  \gtrsim | \tau_{2j}- \zeta_2 |^{K_1} \ \text{ and so } \ \epsilon(\phi, \zeta_2)^{1-\p} \lesssim  | \tau_{2j}- \zeta_2 |^{(1-\p)K_1} \quad \forall \zeta_2 \in U_j.\] 
There is also at least one neighborhood, say $U_1$, where we have the equivalence
\[  \epsilon(\phi, \zeta_2)^{1-\p} \approx  | \tau_{21}- \zeta_2 |^{(1-\p)K_1} \quad \forall  \zeta_2 \in U_1.\]
Writing each $\tau_{2j} =e^{i \theta_j}$  and shrinking each $U_j$ if necessary, we can assume that the $U_j$ are disjoint and each $U_j = \{ e^{i \theta}: |\theta_j -\theta| < \epsilon_j\}$ for some $\epsilon_j>0$. Then we have 
\[ 
\begin{aligned}
 \int_{\mathbb{T}} \epsilon(\phi, \zeta_2)^{1-\p} |d\zeta_2| &\lesssim \sum_{j=1}^{J} \int_{U_j}   | \tau_{2j}- \zeta_2 |^{(1-\p)K_1} |d \zeta_2|  + \int_{\mathbb{T} \setminus \cup U_j}  \epsilon(\phi, \zeta_2)^{1-\p} |d\zeta_2| \\
 &\ \lesssim \sum_{j=1}^{J} \int_{U_j} | \cos(\theta_j) - \cos(\theta) +i(\sin(\theta_j) - \sin(\theta))|^{(1-\p)K_1}  |d \zeta_2|  +  C^{1-\p}\\
 & = \sum_{j=1}^{J}  \int_{\theta_j -\epsilon_j}^{\theta_j +\epsilon_j} \left |  2- 2 \cos(\theta_j-\theta) \right| ^{\frac{(1-\p)K_1}{2}} \ d \theta +  C^{1-\p} \\
 & \approx \sum_{j=1}^{J}  \int_{\theta_j -\epsilon_j}^{\theta_j +\epsilon_j} | \theta_j - \theta |^{(1-\p)K_1} d \theta +  C^{1-\p} \\
 &< \infty,
 \end{aligned} 
\]
as long as $(1-\p)K_1 > -1$ or equivalently, if $K_1 < \frac{1}{\p-1}$. Now, if $K_1 \ge \frac{1}{\p-1}$, we have
\[  \int_{\mathbb{T}} \epsilon(\phi, \zeta_2)^{1-\p} |d\zeta_2|  \ge \int_{U_1} \epsilon(\phi, \zeta_2)^{1-\p} |d\zeta_2|   \approx \int_{\theta_1 -\epsilon_1}^{\theta_1 +\epsilon_1} | \theta_1 - \theta |^{(1-\p)K_1} d \theta = \infty,\]
which completes the proof.
\end{proof}

\begin{example} \label{ex:favcontact} Consider the simple rational inner function 
\[\phi(z) = \frac{2z_1z_2 -z_1-z_2}{2-z_1-z_2},\] whose only singularity on $\mathbb{T}^2$ is at $(1,1)$. Its associated Pick function defined as in \eqref{eqn:pick} is
\[f(w)=\frac{1}{2}(w_1+w_2).\] 
To compute the $z_1$-contact order of $\phi$, fix $\zeta_2 \in \mathbb{T}$ and observe that 
\[ \tilde{p}(z_1,\zeta_2) = 2z_1 \zeta_2 - z_1 - \zeta_2 =0  \ \ \text{ implies } \ \  z_1 = \frac{\zeta_2}{2\zeta_2-1}.\]
Writing $\zeta_2 = e^{i \theta}$ and assuming $\zeta_2$ is close to $1$, so $\theta$ is close to zero,  we can compute
\[ |z_1|^2 = \frac{1}{ 1 + 4 - 4\cos(\theta)}  \approx \frac{1}{1+2 \theta^2} \approx 1-2\theta^2.\]
Then we have 
\[ \epsilon(\phi, \zeta_2) = 1-|z_1| \approx 1-|z_1|^2 \approx \theta^2.\]
Similarly, we can compute
\[ |1 -\zeta_2|^2 = (1-\cos(\theta))^2 + \sin(\theta)^2 \approx  \theta^2.\]
 Combining the two equations gives $\epsilon(\phi, \zeta_2) \approx |1 -\zeta_2|^2,$ for $\zeta_2$ near $1$. 
Thus, the $z_1$-contact order of $\phi$ satisfies $K_1 =2$ and by symmetry, $K_2 = 2.$ 

Theorem \ref{thm:contact} now immediately implies that $\frac{\partial \phi}{\partial z_1}, \frac{\partial \phi}{\partial z_2}   \in H^\p(\mathbb{D}^2)$ if and only if $\p < \frac{3}{2}.$
\end{example}

\begin{example}\label{ex:agmcy}
The rational inner function
\[\phi(z_1,z_2)=\frac{4z_1z_2^2-z_2^2-3z_1z_2-z_2+z_1}{4-z_1-3z_2-z_1z_2+z_2^2}\]
from \cite{AMY12} is an example with higher contact order at its singularity at $(1,1)$. Namely, fixing $\zeta_2 \in \T$ and solving for $z_1$, we find that
\[z_1=\zeta_2\frac{\zeta_2+1}{4\zeta_2^2-3\zeta_2+1}.\]
Setting $\zeta_2=e^{i \theta}$ and computing $|\zeta_2+1|^2=2+2\cos \theta$ and $|4\zeta_2^2-3\zeta_2+1|^2=16\cos^2\theta-30\cos \theta+18$, we obtain
\[|z_1|^2=\frac{1+\cos \theta}{8\cos^2\theta-15\cos \theta+9}\approx 1-\theta^4,\]
meaning that $K_1=4$ for this example. Using computer algebra, one can also show that $K_2=4$. Thus $\frac{\partial \phi}{\partial z_1}, \frac{\partial \phi}{\partial z_2}\in  H^\p(\D^2)$ if and only if $\p<\frac{5}{4}$.
\end{example}

\begin{remark} \label{rem:Hp} Theorem \ref{thm:contact} immediately implies that the derivatives $\frac{\partial \phi}{\partial z_1}, \frac{\partial \phi}{\partial z_2}   \in H^1(\mathbb{D}^2)$ for every rational inner function $\phi$.  Because the $H^\p(\mathbb{D}^2)$ spaces are nested, this immediately implies that $\frac{\partial \phi}{\partial z_1}, \frac{\partial \phi}{\partial z_2}   \in H^\p(\mathbb{D}^2)$ for every $0 < \p \le 1.$ 
\end{remark} 

Moreover, it is worth pointing out that a particularly nice formula holds for each $\| \frac{\partial \phi}{\partial z_i} \|_{H^1(\mathbb{D}^2)}$.

\begin{proposition} \label{prop:H1} Let $\phi = \frac{\tilde{p}}{p}$ be a rational inner function on $\mathbb{D}^2$ with $\deg \phi=(m, n)$. Then
\[  \left \| \tfrac{\partial \phi}{\partial z_1} \right \|_{H^1(\mathbb{D}^2)} =m \ \ \text{ and } \ \  \left \| \tfrac{\partial \phi}{\partial z_2} \right \|_{H^1(\mathbb{D}^2)} =n.\]
\end{proposition}

\begin{proof} We prove the result for $\tfrac{\partial \phi}{\partial z_1}$. Fix $\zeta_2 \in \mathbb{T}$, so that $\phi_{\zeta_2}: = \phi(\cdot, \zeta_2)$ is a finite Blaschke product. Moreover, $\deg \phi_{\zeta_2} = m,$ except at the finite number of points $\zeta_2$ corresponding to the $z_2$-coordinates of singularities of $\phi.$ By Lemma \ref{lem:blaschke}, we know
\[ \left | \tfrac{\partial \phi}{\partial z_1}(\zeta_1,\zeta_2) \right| =  | \phi_{\zeta_2}'(\zeta_1)| = \frac{ \phi_{\zeta_2}'(\zeta_1)}{\phi(\zeta_1)} \zeta_1, \quad \text{ for } \zeta_1 \in \mathbb{T}.\]
Using that and the argument principle, we can compute
\[
\begin{aligned}
 \left \| \tfrac{\partial \phi}{\partial z_1} \right \|_{H^1(\mathbb{D}^2)} 
& =   \frac{1}{4\pi^2}\int_{\mathbb{T}} \left( \int_{\mathbb{T}} \frac{ \phi_{\zeta_2}'(\zeta_1)}{\phi_{\zeta_2}(\zeta_1)} \zeta_1  |d\zeta_1| \right) |d\zeta_2| \\
 & = \frac{1}{2\pi} \int_{\mathbb{T}} \left( \frac{1}{2 \pi i} \int_{\mathbb{T}} \frac{ \phi_{\zeta_2}'(\zeta_1)}{\phi_{\zeta_2}(\zeta_1)} d\zeta_1 \right) |d\zeta_2|  \\
&= \frac{1}{2\pi}\int_{\mathbb{T}} m \ |d\zeta_2| 
= m,
 \end{aligned}
 \] 
 as desired. \end{proof}


\section{RIF Derivatives in $H^\p(\mathbb{D}^2)$: Spokes and Horns} \label{sec:Dspokes}

The goal of this section is to prove the following theorem:

\begin{theorem} \label{thm:Hp} Let $\phi = \frac{\tilde{p}}{p}$ be a rational inner function on $\mathbb{D}^2$ that does not extend continuously to $\overline{\mathbb{D}^2}.$ Then $\frac{\partial \phi}{\partial z_j} \not \in H^\p(\mathbb{D}^2)$ for $j=1,2$ and $\frac{3}{2} \le \p \le \infty.$
\end{theorem}

\begin{remark}In Example \ref{ex:favcontact}, we considered $\phi(z) = \frac{2z_1z_2 -z_1-z_2}{2-z_1-z_2}.$ By computing $\phi$'s contact orders and applying 
Theorem \ref{thm:contact}, we deduced $\frac{\partial \phi}{\partial z_1}, \frac{\partial \phi}{\partial z_2}   \in H^\p(\mathbb{D}^2)$ if and only if $\p < \frac{3}{2}.$
 Thus Theorem \ref{thm:Hp} is sharp in the sense that it cannot be improved for the entire class of rational inner functions.
\end{remark}

Combining Theorems \ref{thm:contact} and \ref{thm:Hp}, we obtain a lower bound on the contact orders of a RIF in the situation where the function does not extend continuously to $\overline{\mathbb{D}^2}$.
\begin{corollary}[Geometric Julia Inequality]\label{cor: geojulia}
Let $\phi= \frac{\tilde{p}}{p}$ be a rational inner function on $\mathbb{D}^2$ that does not extend continuously to $\overline{\mathbb{D}^2}.$ Then the contact orders of $\phi$ satisfy 
$K_i\geq 2$ for $i=1,2$.
\end{corollary}
This complements results of Agler, M\McC Carthy, and Young, see \cite[Theorem 4.9, Corollary 4.12]{AMY12}, which can be interpreted as saying that the maximum of the contact orders of a RIF with a singularity on $\T^2$ is at least $2$.  

To prove Theorem \ref{thm:Hp}, we require the following lemma. It clarifies how a rational inner function $\phi$ can fail to extend continuously to $\overline{\mathbb{D}^2}.$ 

\begin{lemma}  \label{lem:limitpoint} Let $\phi =\frac{\tilde{p}}{p}$ be a rational inner function on $\mathbb{D}^2$ that does not extend continuously to $\overline{\mathbb{D}^2}.$ 
Then, there is some $\tau \in \mathbb{T}^2$ such that $p(\tau)=0$ and a sequence $\{\zeta_n \} \subseteq \mathbb{T}^2$  converging to $\tau$ such that $\{ \phi(\zeta_n)\}$ does not converge to $\phi(\tau)$, where $\phi(\tau)$ is the non-tangential limit that exists because  every $\tau \in \T^2$ is a $B$-point of $\phi$ (see Lemma \ref{lem:Bpoint}).
\end{lemma}

\begin{proof} Proceeding by contradiction, assume that no such $\tau$ exists. Then at each $\tau \in \mathbb{T}^2$, either $p(\tau) \ne0$ or for every $\{\zeta_n \} \subseteq \mathbb{T}^2$ converging to $\tau$, the sequence $\{ \phi(\zeta_n)\}$ converges to $\phi(\tau)$. These conditions imply that $\phi$ is a continuous function on $\mathbb{T}^2$.  
Because $\phi$ is rational inner, we know that the Fourier coefficients of $\phi$ satisfy $\widehat{\phi}(n_1, n_2)=0$ for $(n_1, n_2) \in \mathbb{Z}^2 \setminus \mathbb{N}^2.$ But then by Theorem $2.2.1$ in \cite{Rud69}, we can conclude that $\phi$ is continuous on $\overline{\mathbb{D}^2},$  a contradiction. 
\end{proof}

For the remainder of this section, we make the following assumptions: 
\begin{itemize}
\item[(A1)]  $\phi = \frac{\tilde{p}}{p}$ is a rational inner function on $\mathbb{D}^2$ that does not extend continuously to $\overline{\mathbb{D}^2}$ and $f$ is its associated rational inner Pick function as defined in \eqref{eqn:pick}.
\item[(A2)]  $(1,1)$ is the badly-behaved point guaranteed by Lemma \ref{lem:limitpoint} and $\phi(1,1) = 1$, where equality means $1$ is the non-tangential limit guaranteed because $(1,1)$ is a B-point of $\phi$.
\end{itemize}

\subsection{Spokes and Horns}

We first study the geometry of the sets where a rational Pick function $f$ can be large, earlier called spokes. In particular, we will show that certain level sets can become trapped within a spoke and forced to approach $(\infty, \infty).$ See Subsection \ref{ss:spokes} for a review of the notation.

\begin{proposition} \label{prop:spoke} Assume $\phi$ satisfies assumptions $(A1)$ and $(A2)$. Then there is an $r>0$ and a point $p_0$ on a nontrivial $S(f,\frac{r}{2})$-spoke such that $r<|f(p_0)|<\infty$ and $f$ is continuous on the cut-off spoke $S(f,\frac{r}{2}, p_0)$ as a function mapping into $\mathbb{C}_{\infty}$. Moreover, the connected component of the set $\mathcal{C}$ defined by
\[ \mathcal{C} := \{ (x,y)\in \R^2: f(x,y) = f(p_0)\} \]
that contains $p_0$, and is called $\mathcal{C}_{p_0}$, goes to $(\infty, \infty)$ within $S(f,\frac{r}{2}, p_0)$.
\end{proposition}

\begin{proof}
By Lemma \ref{lem:limitpoint}, there exists a sequence $\{ \zeta_n \} \subseteq \mathbb{T}^2$ converging to $(1,1)$ such that $\{ \phi(\zeta_n)\}$ converges to some $\zeta \ne 1$ with $\zeta \in \T.$ Shifting to the upper-half plane, define the sequence $\{(x_n,y_n)\} = \{\beta^{-1}(\zeta_n)\}$. Then $\{(x_n, y_n)\}$ converges to $(\infty, \infty)$ and $\{f(x_n, y_n)\}$ converges to some nonzero $\tilde{r} \in \mathbb{R} \cup \{ \infty\}$. In what follows, we assume $\tilde{r}$ is finite, but similar arguments handle the $\tilde{r} =\infty$ case. 

Let $ r = \frac{\tilde{r}}{2}$ and consider the spoke region $SR(f, \frac{r}{2})$. For $n$ sufficiently large,  $|f(x_n,y_n)|$ is larger than $r$ and $(x_n,y_n)$ is not in a trivial $S(\frac{r}{2},f)$-spoke. This implies that there must be a nontrivial $S(\frac{r}{2},f)$-spoke containing a subsequence of $\{(x_n, y_n)\}$.  Choose $N$ so that $p_0 =(p_{01}, p_{02}):= (x_N, y_N)$ is in this nontrivial spoke and satisfies $r<|f(p_0)|<\infty$, and $f$ is continuous (as a function into $\mathbb{C}_{\infty}$) on $S(f, \frac{r}{2}, p_0)$, the cut-off spoke associated to $p_0$.

Now, we show $\mathcal{C}_{p_0}$  goes to $(\infty, \infty)$  in $S(f, \frac{r}{2}, p_0)$. Without loss of generality, assume $S(f, \frac{r}{2}, p_0)$ extends in the negative-$x$, positive-$y$ direction.  Fix any point $q_0 \in S(f, \frac{r}{2}, p_0)$ such that $f(q_0 ) \ne \infty$. Then $f$ is continuously differentiable at $q_0$.
For $w \in \Pi$, define the one-variable function
\[ \tilde{f}(w):= f(q_0 +(1,1)w) = f(q_{01} +w, q_{02} +w).\]
Then $\tilde{f}$ is a one-variable rational Pick function and $f(x)$ is real for all $x$ in some real interval containing $0$. By \cite[pp. 19]{Don74}, we can write 
\[ \tilde{f}(w) = aw + b + \sum_{j=1}^d \frac{m_j}{c-w},\]
for $d$ a nonnegative integer, $m_j$ and $a$ positive, and $b, c$ real. Then $\tilde{f}'(w) = \frac{\partial f}{\partial w_1} \frac{\partial w_1}{\partial w} + \frac{\partial f}{\partial w_2} \frac{\partial w_2}{\partial w}$. In particular, at $w=0$, one can compute
\[ \tilde{f}'(0) = \frac{ \partial f }{\partial w_1}(q_0) \cdot 1 +  \frac{ \partial f }{\partial w_2}(q_0) \cdot 1 = a + \sum_{j=1}^d \frac{m_j}{(\lambda-0)^2 } >0.\]
Therefore, $\nabla f(q_0) = \left \langle  \frac{ \partial f }{\partial w_1}(q_0) ,  \frac{ \partial f }{\partial w_2}(q_0) \right \rangle \ne \vec{0}$.

In particular, $\nabla f(p_0) \ne \vec{0}$. Without loss of generality, assume $\frac{ \partial f }{\partial w_2}(p_0) \ne 0.$ Then by the Implicit Function Theorem, there exist open intervals $I$ and $J$ containing $p_{01}$ and $p_{02}$ respectively and a continuously differentiable $g(x):I \rightarrow J$ such that 
\[ \{ (x,g(x)):  x\in I\} = \{ (x,y) \in I \times J: f(x,y) =f(p_0)\}.\]
In particular, we can extend $\mathcal{C}_{p_0}$ in the negative $x$-direction from the point $p_0$. Further, observe that  $\mathcal{C}_{p_0}$ cannot cross the boundary of the spoke because $|f(p_0)| \le \frac{r}{2}$ immediately across the boundary of the spoke. 

Now, we need to show   $\mathcal{C}_{p_0}$ approaches $(\infty,\infty)$ in $S(f,\frac{r}{2}, p_0)$. Because $\mathcal{C}_{p_0}$
is trapped within a nontrivial spoke, it suffices to show the $x$-coordinate goes to $\infty$ in the negative $x$-direction. 
By contradiction, assume  the part of $\mathcal{C}_{p_0}$ in $S(f,\frac{r}{2}, p_0)$ 
has bounded $x$-coordinate. Because $f$ is continuous, this implies that there is some point $\hat{p}_0$ on both $\mathcal{C}_{p_0}$ and $S(f,\frac{r}{2}, p_0)$ 
with smallest x-coordinate (so the $x$-coordinate will be negative with $|x|$ large).

By our earlier arguments, we can apply the Implicit Function Theorem at $\hat{p}_0$. If $\frac{ \partial f }{\partial w_2}(\hat{p}_0) \ne 0,$ 
we can extend $\mathcal{C}_{p_0}$ in the negative $x$-direction and so, obtain the needed contradiction. If $\frac{ \partial f }{\partial w_2}(\hat{p}_0) = 0,$ we can extend $\mathcal{C}_{p_0}$ in the positive $y$-direction.  If the $x$-coordinate is affected, we obtain the needed contraction. Therefore, assume that the points in this extension of $\mathcal{C}_{p_0}$ have $x$-coordinate equal to that of $\hat{p}_0$.

Because $\mathcal{C}_{p_0}$ will eventually cross the spoke boundary otherwise, 
there must be a point with maximum $y$-value obtained by this extension of $\mathcal{C}_{p_0}.$
Applying the Implicit Function Theorem at this point gives the needed contradiction.
\end{proof}
Now, we consider transformations of nontrivial spokes.

\begin{definition}  Let $r>0$ and let $h$ be a rational inner Pick function with a Type $1$ Nevanlinna representation as in Lemma \ref{lem:Spoke}. Assume there is a nontrivial $S(h,r)$-spoke with boundary curves $y = mx \pm b$, for $m$ negative and $b$ positive.
 Consider the conformal map $\gamma\colon \Pi \rightarrow \Pi$ defined by $\gamma(w) = -\frac{1}{w}$. Under $\gamma$, that is, sending $(x,y) \mapsto (-\frac{1}{x}, - \frac{1}{y})$, the nontrivial  $S(h,r)$-spoke is transformed into  a \emph{nontrivial  $H(h,r)$-horn}. The horn has boundary curves $y = \frac{x}{m \pm bx}$ and at a fixed $x$, the horn has height
\[ \left | \frac{x}{m-bx}- \frac{x}{m+bx} \right | = \frac{2b x^2}{|m^2-b^2 x^2|} \approx \frac{ bx^2}{m^2},\]
as long as $x$ is sufficiently close to zero, say in the interval $(\frac{m}{2b}, - \frac{m}{2b})$.
\end{definition}

\begin{center}
\begin{figure}[!htbp]
  \includegraphics[scale=1.3]{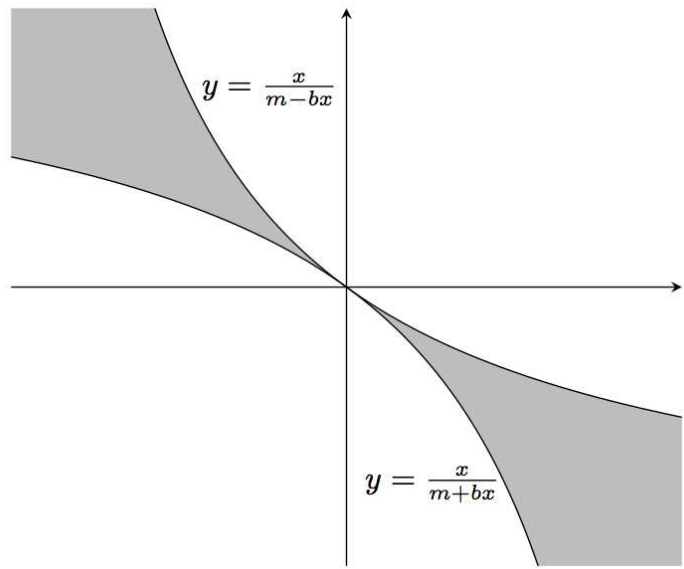}
  \caption{A nontrivial Horn}
  \label{fig:horn}
\end{figure}
\end{center}

Observe that the conformal map $\tilde{\beta}: \Pi \rightarrow \mathbb{D}$, defined earlier by $\tilde{\beta}(w) = \frac{1+iw}{1-iw}$ also satisfies $\tilde{\beta}  = \beta \circ \gamma.$ In what follows, we often write $\gamma(x,y)$ for $(\gamma(x), \gamma(y))$ and use similar notation for $\tilde{\beta}$.  
Now, we can transfer Proposition \ref{prop:spoke} to the associated horn as follows:

 \begin{lemma} \label{lem:horn} Assume $\phi$ satisfies $(A1)$ and $(A2)$. Then there is a non-trivial horn with boundaries $ y = \frac{x}{m \pm bx}$ and a point $\tilde{p}_0$ in the horn such that the connected component of the level set
\[ \tilde{\mathcal{C}} : = \left \{ (x,y) \in \R^2 \colon f(\gamma(x,y)) = f(\gamma(\tilde{p}_0)) \right \}\]
containing $\tilde{p}_0$, which is called $\tilde{\mathcal{C}}_{\tilde{p}_0}$, 
goes to $(0,0)$ within the horn. Moreover if $e^{i \theta_0}: = \phi(\tilde{\beta}(\tilde{p}_0))$, then $e^{i \theta_0} = \phi(\tilde{\beta}(x,y))$ for all $(x,y) \in \tilde{\mathcal{C}}$ and 
\[ \left |\phi( \tilde{\beta}(x,y)) - e^{i\theta_0} \right | \ge \frac{1}{2}\left | 1 - e^{i \theta_0} \right|, \]
for all $(x,y)$ on the boundary of the horn.
\end{lemma}

\begin{proof} Applying Proposition \ref{prop:spoke} to $\phi$ gives an $r>0$ and a point $p_0 \in \R^2$ such that $r < |f(p_0)|<\infty$. Moreover, there is a  nontrivial $S(f, \frac{r}{2})$-spoke containing $p_0$ such that the connected component $\mathcal{C}_{p_0}$ of the level set $ \mathcal{C}:= \{ (x,y) \in \R^2:f(x,y)=f(p_0)\}$ containing $p_0$ goes to $(\infty, \infty)$ in the cut-off spoke $S(f, \frac{r}{2}, p_0).$ Set $M = |1+\phi(\beta(p_0))| >0$ and consider $S(f, \frac{Mr}{4})$. As $\frac{Mr}{4} \le \frac{r}{2}$, all of the $S(f, \frac{r}{2})$-spokes are contained in $S(f, \frac{Mr}{4})$-spokes. Thus, there is a nontrivial $S(f, \frac{Mr}{4})$-spoke containing $p_0$ such that $\mathcal{C}_{p_0}$ goes to $(\infty, \infty)$ within the spoke.   Let $y = mx \pm b$ denote the boundaries of this nontrivial $S(f, \frac{Mr}{4})$-spoke containing $p_0$. Clearly on $\mathcal{C}$, $|f(x,y)| = |f(p_0)|>r$, and on the boundary of the nontrivial $S(f,\frac{Mr}{4})$-spoke containing $p_0$, we have $|f(x,y)| \le \frac{Mr}{4}$.

Now consider the transformation of these objects under $\gamma$. We obtain a nontrivial horn with boundary curves $ y = \frac{x}{m \pm bx}$, a point  $\tilde{p}_0 := \gamma(p_0)$ in the horn, and the level set
 \[ \tilde{\mathcal{C}} := \left \{ \gamma(x,y) \colon (x,y) \in \mathcal{C} \right\} = \left\{ (x,y) \in \R^2 \colon f(\gamma(x,y) ) = f(p_0)=f(\gamma(\tilde{p}_0)) \right\},\]
since $\gamma$ is its own inverse. Let $\tilde{\mathcal{C}}_{\tilde{p}_0}$ be the connected component of $\mathcal{C}$ containing $\tilde{p}_0$. By moving $p_0$ further along the 
 $S(f, \frac{Mr}{4})$ if necessary (so that the part of $\mathcal{C}_{p_0}$ going towards $(\infty, \infty)$ from $p_0$ remains connected under $\gamma$), we can assume  $\tilde{\mathcal{C}}_{\tilde{p}_0}$ 
 approaches $(0,0)$ within the horn. Recalling the definitions of $\tilde{\beta}$ and $f$ from \eqref{eqn:pick} and setting $e^{i\theta_0} := \phi( \beta(p_0)) = \phi(\tilde{\beta}(\tilde{p}_0))$, we have 
\[ \left | f(\gamma(x,y)) \right| = |f (p_0)|  = \left | \frac{ 1-e^{i\theta_0} }{ 1+e^{i\theta_0} } \right| >r \text{ on } \tilde{\mathcal{C}}.\] 
Moreover, one can easily show that $e^{i \theta_0} = \phi(\tilde{\beta}(x,y))$ for all $(x,y) \in \tilde{\mathcal{C}}$. Similarly, for $(x,y)$ on the 
horn boundary, we have 
\[  \left| f(\gamma(x,y)) \right| =\left | \frac{ 1- \phi(\tilde{\beta}(x,y))}{ 1+ \phi(\tilde{\beta}(x,y))} \right| \le \frac{r M}{4} = \frac{r}{4}|1+e^{i\theta_0}|.\]
The previous equations can be combined to obtain:
\[ \left | 1- \phi(\tilde{\beta}(x,y)) \right| \le \frac{r}{2} |1+e^{i\theta_0}| \le \frac{1}{2}  \left | 1 - e^{i\theta_0} \right| \]
and so 
\begin{equation}  \left| \phi(\tilde{\beta}(x,y)) - e^{i \theta_0} \right | \ge \left | 1- e^{i\theta_0} \right | -  \left| 1- \phi(\tilde{\beta}(x,y)) \right |  \ge \frac{1}{2} \left | 1 - e^{i\theta_0} \right|,\end{equation}
as desired.
\end{proof}


\subsection{Integral Estimates}

In this section, we obtain the integral estimates necessarily to prove Theorem \ref{thm:Hp}.  It is easier to integrate over horns in $\mathbb{R}^2$ than over regions in $\mathbb{T}^2$. Thus we define  $\tilde{\phi}:= \phi \circ \tilde{\beta}\colon \Pi^2 \rightarrow \overline{\mathbb{D}}$. Here $\tilde{\phi}$ extends continuously to $\R^2$, except at a finite number of singular points. 

To specify the integration region,  let $ y = \frac{x}{m\pm bx}$ denote the boundaries of the nontrivial horn guaranteed by Lemma \ref{lem:horn}.  Let $\tilde{p}_0=(\tilde{p}_{01}, \tilde{p}_{02})$ denote the point from Lemma \ref{lem:horn} and set
 $x_0 = \min \{-\frac{m}{2b}, -\tilde{p}_{01}\}$. Then for $ x\in [0, x_0]$, the height of the horn at $x$ is approximately $\frac{b x^2}{m^2}$ and there is a point on $\tilde{\mathcal{C}}_{\tilde{p}_0}$ with $x$-coordinate $x$. 
Define $\Omega$ to be the region in the horn with $x$ values in the interval $[0, x_0].$
Then we have:

\begin{lemma}  For $\phi$ satisfying $(A1)$ and $(A2)$ and $\Omega$ defined above, if 
\[ \iint_{\Omega} \left | \tfrac{ \partial \tilde{\phi}}{\partial y} (x,y) \right|^\p  dx \ dy = \infty, \ \text{ then }  \ \iint_{\mathbb{T}^2} \left | \tfrac{\partial \phi}{\partial z_2} (\zeta_1, \zeta_2) \right|^\p |d \zeta_1| |d\zeta_2| = \infty,\]
for $0 < \p <\infty$. 
\end{lemma}
\begin{proof} The proof is basically just changing variables. First writing $\zeta_1 = e^{i \theta_1}$ and $\zeta_2 = e^{i \theta_2}$, we have
\[\iint_{\mathbb{T}^2} \left | \tfrac{\partial \phi}{\partial z_2} (\zeta_1, \zeta_2) \right|^\p |d \zeta_1| |d\zeta_2| =  \int_0^{2\pi} \int_0^{2\pi}   \left | \tfrac{\partial \phi}{\partial z_2} (e^{i \theta_1}, e^{i \theta_2}) \right|^\p d \theta_1 d\theta_2.\]
For each $(\theta_1, \theta_2) \in [0, 2\pi)^2$, there are unique $x,y \in \mathbb{R}$ such that
\[ e^{i \theta_1} = \tilde{\beta}(x) = \frac{1 + i x}{1-ix} \ \text{ and } \ e^{i \theta_2} = \tilde{\beta}(y) = \frac{1 + i y}{1-iy}.\]
Solving for $\theta_1$ gives $\theta_1 = \text{arctan}\left( \frac{2x}{1-x^2}\right) (+\pi)$ and so $\frac{d \theta_1}{d x}  = \frac{2}{1+x^2}.$ Analogous formulas hold for $
\theta_2.$ Changing variables to $x$ and $y$ gives
\begin{equation} \label{eqn:integral1} \int_0^{2\pi} \int_0^{2\pi}   \left | \tfrac{\partial \phi}{\partial z_2} (e^{i \theta_1}, e^{i \theta_2}) \right|^\p d \theta_1 d\theta_2 = \iint_{\mathbb{R}^2} \left | \tfrac{\partial \phi}{\partial z_2} ( \tilde{\beta}(x,y)) \right |^\p \left( \frac{2}{1 +x^2} \right) \left( \frac{2}{ 1 +y^2} \right) \ dx \ dy.\end{equation} 
By the chain rule, 
\[ \tfrac{\partial \tilde{\phi}}{\partial y}(x,y) = \tfrac{\partial \phi}{\partial z_2}(\tilde{\beta}(x,y)) \tilde{\beta}'(y) =  \tfrac{\partial \phi}{\partial z_2}(\tilde{\beta}(x,y)) \frac{2i(1-y^2)-4y}{(1-y^2)^2 +4y^2}.\]
Then, using the boundedness of $\Omega$, we can conclude that 
\[ 
\begin{aligned}
\iint_{\Omega} \left | \tfrac{ \partial \tilde{\phi}}{\partial y} (x,y) \right|^\p  dx \ dy &=\iint_{\Omega} \left | \tfrac{ \partial \phi}{\partial y} (\tilde{\beta}(x,y)) \right|^\p \left|  \frac{2i(1-y^2)-4y}{(1-y^2)^2 +4y^2} \right|^p dx \ dy \\
& \lesssim \iint_{\Omega}  \left | \tfrac{ \partial \phi}{\partial y} (\tilde{\beta}(x,y)) \right|^\p\left( \frac{2}{1 +x^2} \right) \left( \frac{2}{ 1 +y^2} \right) \ dx \ dy,
\end{aligned}\]
which, coupled with \eqref{eqn:integral1}, gives the desired result. \end{proof} 

Now Theorem \ref{thm:Hp} simply requires the following result: 

\begin{proposition} For $\phi$ satisfying $(A1)$ and $(A2)$ and $\Omega$ defined above,
\[ \iint_{\Omega} \left | \tfrac{ \partial \tilde{\phi}}{\partial y} (x,y) \right|^\p  dx  \ dy = \infty, \qquad  \text{ for } \tfrac{3}{2} \le \p < \infty.\]
\end{proposition}

\begin{proof} First, observe that if $g: [c,d] \rightarrow \mathbb{R}$ is a twice continuously differentiable function, then 
\begin{equation} \label{eqn:EL}  \int_c^d |g'(y)|^\p \ dy \ge \frac{ |g(d) - g(c)|^\p}{|d-c|^{\p-1}}, \quad \text{ for } \p>1.\end{equation}
This follows from the calculus of variations and in particular, the Euler-Lagrange equations. See \cite[Chapter 8]{EvansBook}.  Specifically, (up to an $\epsilon$ modification ensuring the Lagrangian function is twice continuously differentiable) one can show that the integral $ \int_c^d |\Gamma'(y)|^\p \ dy$ with the constraints $\Gamma \in C^2[c,d],$ $\Gamma(c)=g(c),$ $\Gamma(d)=g(d)$ is minimized by the straight line 
\[ \Gamma(y) =  \left( \frac{g(d)-g(c)}{d-c}\right) ( x-c) + g(c).\]
Using this $\Gamma$ gives the lower bound in \eqref{eqn:EL}.

Now, recall that $\Omega$ is the region in $\mathbb{R}^2$ satisfying the bounds $0 \le  x \le x_0$ and  $\frac{x}{m+bx} \le y \le \frac{x}{m-bx}.$  Observe that for all but a finite number of $x$-values in $\R$, the function $\tilde{\phi}(x, \cdot)$ has a non-vanishing denominator on $\R$ and so is clearly twice continuously differentiable. Fix such an $x \in (0, x_0)$ and let 
$c_x$ denote the $y$-coordinate of the point with $x$-coordinate on $\tilde{\mathcal{C}}_{\tilde{p}_0}$. Then by the Lemma \ref{lem:horn}, $\tilde{\phi}(x,c_x) =e^{i \theta_0}.$ 
Applying \eqref{eqn:EL} to the real and imaginary parts of $\tilde{\phi}(x, \cdot)$ separately,  we obtain: 
\[
\begin{aligned}
 \int_{\frac{x}{m+bx}}^{\frac{x}{m-bx}}  \left | \tfrac{ \partial \tilde{\phi}}{\partial y} (x,y) \right|^\p \ dy &= \int_{\frac{x}{m+bx}}^{c_x}  \left | \tfrac{ \partial \tilde{\phi}}{\partial y} (x,y) \right|^\p \ dy + \int_{c_x}^{\frac{x}{m-bx}}  \left |\tfrac{ \partial \tilde{\phi}}{\partial y} (x,y) )\right|^\p \ dy \\
 & \gtrsim  \frac{ \left | e^{i \theta_0} - \tilde{\phi} \left(x, \frac{x}{m+bx}\right) \right|^\p}{ \left| c_x -  \frac{x}{m+bx}\right |^{\p-1}}  +  \frac{ \left | \tilde{\phi} \left(x, \frac{x}{m-bx}\right) -e^{i \theta_0} \right|^\p}{ \left|  \frac{x}{m-bx} - c_x \right |^{\p-1}}  \\
 & \gtrsim \frac{ \left | e^{i \theta_0} - 1 \right|^\p}{ \left| \frac{x}{m-bx} -  \frac{x}{m+bx}\right |^{\p-1}}  +  \frac{ \left | 1 -e^{i \theta_0} \right|^\p}{ \left|  \frac{x}{m-bx} - \frac{x}{m+bx} \right |^{\p-1}}  \\
& \approx \frac{1}{x^{2\p-2}},
 \end{aligned}\]
where the implied constants depends only on $\p$, $m$, and $b$. In the second-to-last line, we also used Lemma \ref{lem:horn} and the fact that the points $(x, \frac{x}{m-bx}), (x, \frac{x}{m+bx})$ are on the boundary of the horn.  Then
\[ \iint_{\Omega} \left | \tfrac{ \partial \tilde{\phi}}{\partial y} (x,y) \right|^\p  dx  \ dy \gtrsim \int_0^{x_0} \frac{1}{x^{2\p-2}} \ dx = \infty,\]
as long as $\frac{3}{2} \le \p < \infty.$ 
\end{proof}

It is worth mentioning that the above proofs do not explicitly discuss the case $\p=\infty$. However, by contrapositive, if $\frac{\partial \phi}{\partial z_i} \not \in H^\p(\mathbb{D}^2)$ for some $0<\p < \infty$, we also have $\frac{\partial \phi}{\partial z_i} \not \in H^{\infty}(\mathbb{D}^2)$ immediately.


\section{RIF Derivatives in $H^2(\mathbb{D}^2)$: Agler Decompositions} \label{sec:DAgler}

In this section, we provide an alternate proof that, for certain rational inner functions $\phi$, a partial derivative $\frac{\partial \phi}{\partial z_1}$ or $\frac{\partial \phi}{\partial z_2}$ is not in $H^2(\mathbb{D}^2)$. This proof rests on Agler decompositions and related results from \cite{Kne15} involving polynomial ideals, intersection multiplicities, and vanishing orders; these were defined in Section \ref{sec:prelim}. This alternative proof has the advantage of being short, and it also illustrates, along with our work on local Dirichlet integrals in subsequent sections, the usefulness of Agler decompositions in the context of examining the integrability of rational inner functions and their derivatives.

First, for $j=1,2$ define $q_j$ so that 
\[ \frac{\partial}{\partial z_j} \left( \frac{\tilde{p}}{ p} \right) = \frac{ p \frac{\partial \tilde{p}}{\partial z_j} -\tilde{p} \frac{ \partial p}{\partial z_j}}{p^2} =\frac{q_j}{p^2}.\]
Then the statement $\frac{\partial \phi}{\partial z_1}, \frac{\partial \phi}{\partial z_2} \in H^2(\mathbb{D}^2)$ is equivalent to $q_1, q_2 \in \mathcal{I}_{p^2},$ where $\mathcal{I}_{p^2}$ denotes a polynomial ideal defined in Subsection \ref{ssec:ideal}.
In the case where the intersection multiplicity $N_{\mathbb{T}^2}(p, \tilde{p})$ is maximal, we have a complete and attractive description of $\mathcal{I}_{p^2}$ in terms of $\mathcal{I}_{p}$.

\begin{lemma} \label{lem:p2ideal} Assume $p \in \mathbb{C}[z_1, z_2]$ is an atoral polynomial with no zeros on the bidisk and $\deg p = (m,n)$. Further, assume $N_{\mathbb{T}^2}(\tilde{p}, p) = 2mn.$ 
Then 
\[ \mathcal{I}_{p^2} = \{ q_1 p + q_2 \tilde{p} : q_1, q_2 \in \mathcal{I}_p \}. \] 
\end{lemma}
\begin{proof} Let $(K_1, K_2)$ be the Agler kernels of $\phi := \frac{\tilde{p}}{p}$ given in \eqref{eqn:Akernels}. By Corollary $13.6$ in \cite{Kne15}, $\phi$ has a unique pair of Agler kernels, so (up to left multiplication by a unitary matrix), $\vec{E}_1=\vec{F}_1$ and $\vec{E}_2 = \vec{F}_2.$ By Theorem $7.1$ in \cite{Kne15}, the elements of $\vec{E}_1, \vec{F}_2$ generate $\mathcal{I}_p$. Now observe that
\[
\begin{aligned}
 |p(z)^2|^2 - |\tilde{p}(z)^2|^2 &= (|p(z)|^2 +|\tilde{p}(z)|^2) ( |p(z)|^2 - |\tilde{p}(z)|^2 )\\
&= (1-|z_1|^2) (|p(z)|^2 +|\tilde{p}(z)|^2)  \vec{E}_1(z)^*\vec{E}_1(z) \\
\quad & + (1-|z_2|^2) (|p(z)|^2 +|\tilde{p}(z) |^2) \vec{F}_2(z)^*\vec{F}_2(z). 
\end{aligned}
\]
A moment's thought should reveal that if $\{ p_1, \dots, p_m\}$ are the entries of $\vec{E}_1$ and $\{ r_1, \dots, r_n\}$ are the entries of $\vec{F}_2$, then one can find Agler kernels of $\phi^2$ corresponding to polynomial vectors  with entries $\{p p_1, \dots, p p_m, \tilde{p}p_1, \dots, \tilde{p}p_m\}$ and $\{p r_1, \dots, p r_n, \tilde{p}r_1, \dots, \tilde{p}r_n\}$ respectively. However, $N_{\mathbb{T}^2}(\tilde{p}^2, p^2) =4 N_{\mathbb{T}^2}(\tilde{p}, p) =  2(2m\cdot2n),$ the maximum intersection multiplicity of $p^2$ and $\tilde{p}^2$ on $\mathbb{T}^2.$ Then another application of Corollary 13.6 in \cite{Kne15} implies that 
$\phi^2$ has a unique pair of Agler kernels. Thus, by Theorem 7.1 in \cite{Kne15},
$\mathcal{I}_{p^2} $ is generated by 
\[ \{p p_1, \dots, p p_m, \tilde{p}p_1, \dots, \tilde{p}p_m\} \cup \{p r_1, \dots, p r_n, \tilde{p}r_1, \dots, \tilde{p}r_n\},\]
which proves the claim. 
\end{proof}

We can now prove the following derivative result:

\begin{theorem}\label{thm:notinH2}  Let $\phi = \frac{\tilde{p}}{p}$ be a rational inner function on $\D^2$ with $\deg \phi = (m,n)$. Further, assume $N_{\mathbb{T}^2}(\tilde{p}, p) = 2mn.$ 
Then at least one of  $\frac{\partial \phi}{\partial z_1}$, $\frac{\partial \phi}{\partial z_2}$ is not in $H^2(\mathbb{D}^2).$
\end{theorem}
\begin{proof} Without loss of generality, assume $p$ has a zero at $(\tau_1,\tau_2) \in \mathbb{T}^2$ and vanishes there to order $M$. This means we can write
\[ p(\tau_1-\lambda_1, \tau_2-\lambda_2) = \sum_{i=M}^{n+m} P_i(\lambda_1, \lambda_2),\]
where each $P_i$ is a homogeneous polynomial in $\lambda_1, \lambda_2$ of total degree $i$ and $P_M \not \equiv 0$. Then $P_M$ contains a nonzero term involving at least one of $\lambda_1$, $\lambda_2$. Assume, without loss  of generality, that there is a nonzero term involving $\lambda_1$. Then, since $z_1=\tau_1-\lambda_1$, $\frac{\partial p}{\partial z_1} = \frac{\partial p}{\partial \lambda_1} \frac{\partial \lambda_1}{\partial z_1} = -  \frac{\partial p}{\partial \lambda_1}$, we have
\[ \left(\frac{\partial p}{\partial z_1}\right)(\tau_1 -\lambda_1, \tau_2-\lambda_2) = -  \left(\frac{\partial p}{\partial \lambda_1}\right)(\tau_1 -\lambda_1, \tau_2-\lambda_2) =- \sum_{i=M}^{n+m}  \frac{\partial P_i}{\partial \lambda_1}(\lambda_1, \lambda_2).\]
The derivative of the nonzero term from $P_M$ involving $\lambda_1$ is a nonzero homogeneous polynomial of total degree $M-1$, and it will not cancel with any other terms in the derivative, since it is the unique monomial with a given power in $\lambda_1, \lambda_2$. This  implies that $\frac{\partial p}{\partial z_1}$ vanishes to at most order $M-1$ at $(\tau_1,\tau_2)$ and so by Theorem 14.10 in \cite{Kne15}, $\frac{\partial p}{\partial z_1} \not \in \mathcal{I}_p$.

Proceeding towards a contradiction, assume $\frac{\partial \phi}{\partial z_1} \in H^2(\mathbb{D}^2)$ and so $q_1 \in \mathcal{I}_{p^2}.$ By Lemma \ref{lem:p2ideal}, this means we can write
\[ p \tfrac{\partial \tilde{p}}{\partial z_1} -\tilde{p} \tfrac{ \partial p}{\partial z_1} = \tilde{p} r_1 + pr_2\]
for some $r_1, r_2 \in \mathcal{I}_p.$ Rearranging terms gives
\[ p(  \tfrac{\partial \tilde{p}}{\partial z_1}  - r_2) = \tilde{p}( r_1 +\tfrac{ \partial p}{\partial z_1} ).\]
As $p$ and $\tilde{p}$ share no common factors (this follows from $p$ atoral), we can conclude that $p$ divides $r_1 + \tfrac{ \partial p}{\partial z_1} $ and so there is a polynomial $R_1$ such that
\[ r_1 + \tfrac{ \partial p}{\partial z_1}  = p R_1.\]
As $p \in \mathcal{I}_p$, this immediately implies that $ \tfrac{ \partial p}{\partial z_1}  \in \mathcal{I}_p,$ a contradiction. \end{proof}


\section{RIF Derivatives: $H^\p(\mathbb{D}^2)$ Inclusion vs Non-tangential Regularity} \label{sec:reg}

The goal of this section is to prove the following theorem:

\begin{theorem} \label{thm:BJ}  Let $\phi = \frac{\tilde{p}}{p}$ be a rational inner function on $\mathbb{D}^2$ with a singularity at $\tau \in \mathbb{T}^2$. If $\tau$ is a $B^J$ point of $\phi$, then 
\[ \max \big \{ K_1, K_2 \big \} \ge 2J, \]
where each $K_i$ is the $z_i$-contact order of $\phi.$
\end{theorem}
Let $\phi = \frac{\tilde{p}}{p}$. Then Theorem \ref{thm:BJ} can be compared to Corollary 14.7 in \cite{Kne15}, which gives a relationship
between the intersection multiplicity of  $\tilde{p}$ and $p$ at a singularity of $\phi$ and the level of non-tangential
regularity  of $\phi$ at the given singularity. Moreover, we may view Theorem \ref{thm:BJ} as a partial higher-order analogue of the Geometric Julia Inequality in Corollary \ref{cor: geojulia}.
Theorem \ref{thm:BJ} coupled with Theorem \ref{thm:contact} yields the following:

\begin{corollary} \label{cor:BJc} Let $\phi = \frac{\tilde{p}}{p}$ be a rational inner function on $\mathbb{D}^2$ with a singularity at $\tau \in \mathbb{T}^2$.  If $\tau$ is a $B^J$ point of $\phi$, then at least one of $\frac{\partial \phi}{\partial z_1}$ and  $\frac{\partial \phi}{\partial z_2}$ fails to belong to $H^\p(\D^2)$ for $\p \ge \frac{1}{2J} +1.$
\end{corollary}


\subsection{Proof of Theorem \ref{thm:BJ}} 
Before proceeding, we require the following lemma, which is likely well-known:

\begin{lemma} \label{lem:ineq} If $A$ and $B$ are self-adjoint $n \times n$ matrices and $B$ is positive-definite, then 
\[ \left \| (A -i B)^{-1} \right \| \le \left \| B^{-1} \right \|. \]
\end{lemma}

\begin{proof} Observe that since $A$ is self-adjoint and $B$ is positive-definite, we have: 
\[  
\begin{aligned}
\left \| (A -i B)^{-1} \right \|^2 &= \left \| B^{-\frac{1}{2}} \left( B^{-\frac{1}{2}} A B^{-\frac{1}{2}} -i I \right)^{-1}  B^{-\frac{1}{2}} \right \|^2  \\
& \le \left \| B^{-\frac{1}{2}}  \right \|^4 \left \|  \left( B^{-\frac{1}{2}} A B^{-\frac{1}{2}} -i I \right)^{-1}  \right \|^2 \\
& \le \left \| B^{-1}  \right \|^2,
\end{aligned}
\]
where we used the fact that $B^{-\frac{1}{2}} A B^{-\frac{1}{2}}$ is self-adjoint to conclude that the moduli of the eigenvalues of $\left( B^{-\frac{1}{2}} A B^{-\frac{1}{2}} -i I\right)^{-1} $ are all less than $1$.
\end{proof}

Now we can proceed with the proof of Theorem \ref{thm:BJ}:

\begin{proof} Let $\phi$ be a rational inner function on $\mathbb{D}^2$. Moreover, assume there is a point $\tau \in \mathbb{T}^2$ such that $\tau$ is both a singular point and $B^J$ point of $\phi$. Without loss of generality, assume $\tau = (1,1)$. \\

\noindent \textbf{Step 1.} We will first show that if $z=(z_1, z_2) \in \mathbb{D}^2$ satisfies $\tilde{p}(z)=0$ and is sufficiently near $(1,1)$, then 
\begin{equation} \label{eqn:contactJ} \min \left\{ 1-|z_1|, 1-|z_2| \right\} \lesssim \left \| z - (1,1) \right\|^{2J}.
 \end{equation}
It is easier to prove \eqref{eqn:contactJ} on $\Pi^2$. So, define $g$ by
\[ g(w) :=  \tilde{\alpha}(\phi(\tilde{\beta}(w))) =  i \left[ \frac{1-\phi( \tilde{\beta}(w))}{1+\phi( \tilde{\beta}(w))} \right].\]
Then Lemma \ref{lem:BJ} implies that there is a finite-dimensional Hilbert space $\mathcal{H}$, a self-adjoint operator $A$ and a positive contraction on $Y$ on $\mathcal{H}$ satisfying $0 \le Y \le I$ such that 
\[ g (w) = q(w) + \left \langle \left(A+ (W^{-1})_Y \right)^{-1} R_{J-1}(w),  R_{J-1}(\bar{w}) \right \rangle_{\mathcal{H}}, \]
where $q$ is a polynomial with real coefficients satisfying $q(0,0)=0$ and $R_{J-1}$ is a homogeneous vector-valued polynomial of degree $J-1$. If $g(w) =i$ and $w \in \Pi^2$ is near $(0,0)$, then using Lemma \ref{lem:ineq}, we have 
\[ 
\begin{aligned}
1 &= |g(w)| \lesssim \left \|  \left(A+ (W^{-1})_Y \right)^{-1}  \right \| \cdot \| w \|^{2J-2} \\
& = \left \|  \left(A+ \tfrac{1}{w_1}Y +\tfrac{1}{w_2}\left(I -Y \right) \right)^{-1}  \right \| \cdot \| w \|^{2J-2} \\
& \le \left \| \left( \text{Im}\left(\tfrac{1}{w_1} \right)Y + \text{Im}\left(\tfrac{1}{w_2}\right)(I-Y) \right)^{-1} \right \| \cdot \| w \|^{2J-2} \\
& \lesssim \frac{1}{\min \left\{ \left|  \text{Im}\left( \tfrac{1}{w_1}\right) \right| , \left | \text{Im}\left(\tfrac{1}{w_2}\right) \right|\right\}} \| w\|^{2J-2}.
 \end{aligned}
 \]
From here, we can conclude that 
\begin{equation} \label{eqn:contactJ2} \min \left \{ \text{Im}(w_1), \text{Im}(w_2) \right \} \le \| w\|^2 \min \left \{  \left | \text{Im}\left(\tfrac{1}{w_1}\right) \right|, \left | \text{Im}\left(\tfrac{1}{w_2}\right) \right| \right\} \le \| w\|^{2J},\end{equation}
where the first inequality is trivial and the second follows from the above string of inequalities. 

Now we return to $\D^2$. Let $U$ be a small neighborhood of $(1,1)$ such that if $z \in U \cap \D^2$, then $w:=\tilde{\beta}^{-1}(z)$ satisfies \eqref{eqn:contactJ2} and 
\[ 1- |z_j| \approx \text{Im}(w_j) \ \text{ and } \ |1-z_j | \approx |w_j| \quad \text{ for } j=1,2\]
as in the proof of Theorem \ref{lem:contact}. Then since $\tilde{p}(z) =0$ on $\D^2$ precisely when $g(w) =i$ on $\Pi^2$, we obtain \eqref{eqn:contactJ} for $z \in U \cap \D^2.$\\

\noindent \textbf{Step 2.} Now define $q(\lambda_1, \lambda_2):= p(\lambda_1, \lambda_1 \lambda_2).$ Then by Theorem 2.4 in \cite{AMS06}, $q$ is atoral and $\theta : = \frac{\tilde{q}}{q}$ is rational inner. Since by definition, $q(1,1)=0$, we also have $\tilde{q}(1,1)=0.$ 

Define the polynomial $r(w_1, w_2):=\tilde{q}(1-w_1, 1-w_2)$.  Then $r(0,0)=0$ and by the discussion of Puiseux's Theorem in Remark \ref{rem:PT}, there are neighborhoods $\widetilde{U}_1$ and $\widetilde{U}_2$ containing zero, natural numbers $L, N_1, \dots, N_L$, and power series $\psi_1, \dots, \psi_L$ converging on a neighborhood of zero and satisfying $\psi_{\ell}(0)=0$ such that the set of points 
\[ C := \left \{ (w_1, w_2) \in \widetilde{U}_1 \times \widetilde{U}_2: r(w_1, w_2) =0 \right\} \]
is contained in the set of points
\[\left \{  \Big(\psi_{1}\Big(w_2^{\frac{1}{N_1}} \Big), w_2 \Big), \dots,  \Big(\psi_{L} \Big(w_2^{\frac{1}{N_L}} \Big), w_2 \Big): w_2 \in \widetilde{U}_2\right \}.\]
Here each $ w_2^{\frac{1}{N_{\ell}}}$ is $N_{\ell}$-valued and 
the curves $\big(\psi_{\ell} \big(w_2^{\frac{1}{N_{\ell}}} \big), w_2 \big)$ are contained in $C$ for sufficiently small $w_2$.
 Writing each $\psi_{\ell}(t)  = \sum_{k=0}^{\infty} a_{\ell k} t^k$ and interpreting $w_2 ^{\frac{1}{N_{\ell}}}$ as an $N_{\ell}$-valued function, we can write the $w_1$-coordinates of points in $C$ as 
\[ W^{\ell}_1(w_2)  = \psi_{\ell} \left( w_2 ^{\frac{1}{N_{\ell}}} \right) = \sum_{k=0}^{\infty} a_{\ell k} \left( w_2 ^{\frac{k}{N_{\ell}}}\right)  \approx a_{\ell} \left(w_2 ^{\frac{k_{\ell}}{N_{\ell}}} \right) + O \left( \| w_2 \|^{\frac{k_{\ell}+1}{N_{\ell}}}\right) \quad \text{ for } 1 \le \ell \le L,  \]
for $w_2$ sufficiently small and $a_{\ell} \ne 0$ and $k_{\ell} \in \mathbb{N}.$

This immediately gives a parameterization of the zero set of $\tilde{q}$ near $(1,1)$. Indeed, if $\widetilde{V}_1$ and $\widetilde{V}_2$ are the obvious translations of $\widetilde{U}_1$ and $\widetilde{U}_2$, the set of points  
\[\tilde{C}:= \left\{ (\lambda_1,\lambda_2)  \in \widetilde{V}_1 \times \widetilde{V}_2: \tilde{q}(\lambda_1,\lambda_2)=0 \right\}\]
  must have $\lambda_1$-coordinates satisfying
\[ \Lambda^{\ell}_1 \big(\lambda_2 ) = 1- \sum_{k=0}^{\infty} a_{\ell k} \left( 1-\lambda_2 \right)^{\frac{k}{N_{\ell}}}  \approx 1- a_{\ell} \left( 1-\lambda_2 \right)^{\frac{k_{\ell}}{N_{\ell}}}  + O \left( \| 1-\lambda_2 \|^{\frac{k_{\ell}+1}{N_{\ell}}}\right) \quad \text{ for } 1 \le \ell \le L,  \]
for nonzero constants $a_{\ell}$. Furthermore, each $\psi_{\ell}(0)=0$ implies that each $k_{\ell} >0$. Fix $\ell$ and for ease of notation, write
\begin{equation} \label{eqn:approx2} \lambda_1 = \Lambda^{\ell}_1 \big(\lambda_2 ) \approx 1- a \left( 1-\lambda_2 \right)^{\frac{k}{N}}  + O \left( \| 1-\lambda_2 \|^{\frac{k+1}{N}}\right).\end{equation}
Without loss of generality, assume $\gcd(k,N)=1$. We will show that $N=1.$ Set $\lambda_2 = 1 +\epsilon$, where $\epsilon >0$ will be chosen later. Then
\[ \lambda_1 \approx 1- a   \left( -\epsilon \right)^{\frac{k}{N}} = 1 + \left |\epsilon^{\frac{k}{N}} \right| b \cdot \eta_j^k \qquad 1 \le j \le N,\]
for some nonzero $b \in \mathbb{C}$ and $\eta_1, \dots, \eta_{N}$ the $N^{th}$ roots of unity. 
Because $\gcd(k,N)=1$, the list $\eta^k_1, \dots, \eta^k_{N}$ contains exactly the  N$^{th}$ roots of unity.  Moreover each quadrant in $\R^2$ contains at least one N$^{th}$ root of unity, where we say adopt the convention that points on an axis are said to lie in both quadrants. Thus, we can choose $\eta_j$ such that either
\[ \text{Re}\left( b \cdot \eta_j^k \right) >0 \ \text{ or }  \text{Re}\left( b \cdot \eta_j^k \right) =0 \text{ and } \text{Im}\left( b \cdot \eta_j^k \right) \ne 0.\]
Given this $\eta_j$ and $\epsilon >0$ chosen small enough, we can guarantee $| \lambda_1| , |\lambda_2| >1.$ But since $\tilde{q}(\lambda_1, \lambda_2)=0$, then $q$ vanishes at $\left(\frac{1}{\bar{\lambda}_1}, \frac{1}{\bar{\lambda}_2} \right) \in \mathbb{D}^2$, a contradiction. 

Now shrinking $\widetilde{V}_1, \widetilde{V}_2$ if necessary and using \eqref{eqn:approx2} and the conclusion that each $N_{\ell}=1$,  
there is some nonzero $K \in \mathbb{N}$ such that 
\begin{equation} \label{eqn:bound} \left |1-\lambda_1 \right| \lesssim \left | 1 - \lambda_2 \right|^K \le  \left|1 -\lambda_2 \right|,\end{equation}
for all $(\lambda_1, \lambda_2) \in \widetilde{V}_1 \times \widetilde{V}_2$ satisfying 
$\tilde{q}(\lambda_1, \lambda_2)=0$. \\

 \noindent \textbf{Step 3.} Choose $ \zeta_2$ close to $1$ and further, assume $\zeta_2 \in \mathbb{T}$. We will show
 \begin{equation} \label{eqn:contact3} \epsilon(\theta, \zeta_2) = \min \left \{ 1-|\lambda_1| : \tilde{q}(\lambda_1, \zeta_2)=0 \right\}  \lesssim \left | \zeta_2 - 1 \right|^{2J},\end{equation}
implying that the $z_1$-contact order of $\theta = \frac{\tilde{q}}{q}$ is at least $2J$. For $\lambda_1 \in \D$, set $z_1 = \lambda_1$ and $z_2 = \lambda_1 \zeta_2$. Then using $\zeta_2 \in \mathbb{T}$, we have
\[ \| z- (1,1)\|^2  \approx |\lambda_1 -1|^2 + |\lambda_1 \zeta_2 -1 |^2  \lesssim |\lambda_1 -1|^2  + |\zeta_2 -1 |^2 \approx \| (\lambda_1, \zeta_2) - (1,1) \|^2.\]
 Thus, for  $(\lambda_1, \zeta_2)$ in a sufficiently small neighborhood of $(1,1)$, we can assume $(z_1, z_2)$ is in the neighborhood $U \cap \D^2$, defined in Step $1.$ Moreover, $\tilde{q}(\lambda_1, \zeta_2) =0$ if and only if $\tilde{p}(z_1,z_2)=0$. Then for all $(\lambda_1, \zeta_2)$ in a small neighborhood of $(1,1)$ satisfying $\tilde{q}(\lambda_1, \zeta_2)=0$, we can apply \eqref{eqn:contactJ} and \eqref{eqn:bound} to obtain
\[ 
\begin{aligned}
1 - \left | \lambda_1 \right | &=  \min\big \{ 1-|z_1|, 1-|z_2| \big \} \lesssim \| z- (1,1)\|^{2J} \lesssim \| (\lambda_1, \zeta_2) - (1,1) \|^{2J} \\
&\approx \left( |1-\lambda_1|^2 + |1-\zeta_2|^2 \right)^J 
\lesssim |1-\zeta_2 |^{2J}.
\end{aligned}
\]
This finishes the proof of \eqref{eqn:contact3}. 

By Theorem \ref{thm:contact}, we can conclude that $\frac{\partial \theta}{\partial \lambda_1} \not \in H^\p(\mathbb{D}^2)$ for all $p \ge \frac{1}{2J} +1.$ A computation using a change of variables and the chain rule now shows that if $\frac{\partial \theta}{\partial \lambda_1}  \not \in H^\p(\mathbb{D}^2)$, then at least one of the statements 
\[\frac{\partial \phi}{\partial z_1} \not \in H^\p(\mathbb{D}^2) \quad \textrm{or} \quad \frac{\partial \phi}{\partial z_2} \not \in H^\p(\mathbb{D}^2)\] 
must hold. Then another application of Theorem \ref{thm:contact} gives
\[ \max \big \{ K_1, K_2 \big \} \ge 2J, \]
where $K_i$ is the $z_i$-contact order of $\phi.$ \end{proof}


\section{Preliminaries: Dirichlet-Type Spaces} \label{sec:dirichlet}

\subsection{Basics of $\Da$ and $\Dveca$ spaces}
Recall that the {\it Dirichlet-type spaces} $\Da$ are spaces of functions $f(z)=\sum_{k,\ell}a_{k\ell}z_1^kz_2^{\ell}$ that are holomorphic in the bidisk and for $\alpha \in \R$ fixed, satisfy 
\[\|f\|^2_{\Da}=\sum_{k, \ell=0}^{\infty}(k+1)^{\alpha}(\ell+1)^{\alpha}|a_{k{\ell}}|^2<\infty.\] 
We are particularly interested in the {\it isotropic} spaces $\Da$ but it is frequently useful to consider the 
more general scale of {\it anisotropic Dirichlet-type spaces} $\Dveca$, furnished with the norm
\begin{equation}\label{def:dveca}
\|f\|^2_{\Dveca}=\sum_{k, \ell=0}^{\infty} (k+1)^{\alpha_1}(\ell+1)^{\alpha_2}|a_{k\ell}|^2,
\end{equation}
where now $\veca \in \R^2$. 

We collect some basic material on Dirichlet-type spaces here, and refer the reader to \cite{Hed88, Kap94, JR06, BCLSS15, BKKLSS15, KKRS} for more information 
concerning $\Da$ and $\Dveca$. The one-variable versions of the Dirichlet spaces, consisting of holomorphic functions $f\colon \D \to \C$,  will be denoted by $D_{\alpha}$ and are discussed in the textbook \cite{EKMRBook}.
The norm of $f=\sum_ka_kz^k\in D_{\alpha}$ is given by 
\[\|f\|^2_{D_{\alpha}}=\sum_{k=0}^{\infty} (k+1)^{\alpha}|a_k|^2.\] 
From the norm definition in \eqref{def:dveca}, one can see that the polynomials are dense in every $\Dveca$. In addition, any holomorphic function that extends to a strictly larger bidisk belongs to each $\Dveca$. 
 Moreover, if each $\beta_j\leq \alpha_j$, then $\Dveca \subset \mathfrak{D}_{\vec{\beta}}$, see \cite{JR06}, and each one variable $D_{\alpha_j}$ embeds in $\Dveca$ in a natural way. When $\max\{\alpha_1,\alpha_2\}>1$, the spaces $\Dveca$ are Banach algebras of functions continuous on the 
closed bidisk. Functions in weighted Dirichlet spaces are not necessarily in $C(\overline{\D}^2)$, but $\Da$ is contained in $H^2(\mathbb{D}^2)$ whenever $\alpha\geq 0$. Thus each $f\in \Da$ possesses a radial limit at almost every point of $\T^2$, see \cite{Rud69, Kap94}.

If $\vec{\alpha}\in \R^2$ satisfies $\max(\alpha_1, \alpha_2) \le 2,$ then  $\Dveca$ has an equivalent integral norm given by
\begin{equation} \label{eqn:intnorm} \int_{\D^2}\left|\frac{\partial^2}{\partial z_1\partial z_2}(z_1z_2f(z_1,z_2))\right|^2dA_{\alpha_1}(z_1)dA_{\alpha_2}(z_2),\end{equation}
where $dA_{\alpha_j}(z_j)=(1-|z_j|^2)^{1-\alpha_j}dA(z_j)$, and $dA(z_j)$ is normalized area measure in the disk,  see \cite[pp. 2]{KKRS}.
Later, we shall make use of the fact that the expression $\int_{\D^2}|\frac{\partial^2 f}{\partial z_1 \partial z_2}|^2dAdA$ is invariant under the transformation
$f\mapsto f\circ m_{a,b}$, where for $a,b\in \D$,
\begin{equation}
 \label{eqn:mobius} m_{a,b}(z_1,z_2)=\left(\frac{a-z_1}{1-\overline{a}z_1}, \frac{b-z_2}{1-\overline{b}z_2}\right)\in\mathrm{Aut}(\D^2).
\end{equation}
We say that a (necessarily holomorphic) function $m\colon \D^2\to\C$ belongs to the {\it multiplier space} $M(\Dveca)$ if $mf\in \Dveca$ for all $f\in \Dveca$. The multiplier spaces do not seem to admit a simple characterization for all $\veca\in \R^2$, but it is known that $M(\Dveca)=H^{\infty}(\mathbb{D}^2)$ when $\max\{\alpha_1,\alpha_2\}\leq 0$ (see \cite{JR06}), and 
furthermore, every function holomorphic on a neighborhood of the closed bidisk is a multiplier for each $\Dveca$.


\subsection{Local Dirichlet Integrals}

An important tool in the study of the Dirichlet space on the unit disk is the {\it local Dirichlet integral} introduced by Richter and Sundberg in \cite{RS91}; see 
\cite{EKMRBook} for applications and further developments. For $\zeta \in \T$ and $f\in H^2(\D)$, set
\[D_{\zeta}(f):=\frac{1}{2\pi}\int_{\T}\left|\frac{f(\eta)-f(\zeta)}{\eta-\zeta}\right|^2 |d\eta|,\]
with the understanding that $D_{\zeta}(f)=\infty$ if the radial limit $f(\zeta)$ does not exist. Integrating the local Dirichlet integral with respect to $|d\zeta|/(2\pi)$, normalized arclength measure, we recover 
the classical {\it Douglas formula},
\[D(f):=\frac{1}{4\pi^2}\int_{\T^2}\left|\frac{f(\eta)-f(\zeta)}{\eta-\zeta}\right|^2|d\zeta||d\eta|,\]
which expresses the Dirichlet integral $\int_{\D}|f'|^2dA$ in terms of boundary values of the function $f$.  In \cite[Lemma 3.2]{RS91}, Richter and Sundberg make the useful observation that if 
$f\in H^2(\mathbb{D})$ and $z\in \D$, then
\begin{equation}
D_{z}(f):= \frac{1}{2\pi} \int_{\T}\frac{|f(\eta)-f(z)|^2}{|\eta-z|^2}|d\eta|=  \frac{1}{2\pi}  \int_{\T}\frac{|f(\eta)|^2-|f(z)|^2}{|\eta-z|^2}|d\eta|.
\label{ortholocal}
\end{equation}
It follows, in particular, that if $\phi$ is an inner function on the disk, then
\[D_z(\phi)=\frac{1}{2\pi}\int_{\T}\frac{1-|\phi(z)|^2}{|\eta-z|^2}|d\eta|.
\]
We will develop similar objects on the bidisk but first, derive an equivalent norm for the two-variable Dirichlet space $\mathfrak{D}$:

\begin{lemma} \label{lem:norm}
Let $f$ be a function in $\mathfrak{D}$ with  $f= \sum_{k, \ell=0}^{\infty} a_{k\ell}z_1^k z_2^{\ell}.$ Then the expression
\[
\begin{aligned}
 \mathrm{Doug}(f)&:= |f(0,0)|^2\\
\nonumber &+\frac{1}{4\pi^2}\int_{\T^2}\frac{|f(\zeta_1, 0)-f(\eta_1,0)|^2}{|\zeta_1-\eta_1|^2}|d\zeta_1||d\eta_1|+\frac{1}{4\pi^2}\int_{\T^2}\frac{|f(0,\zeta_2)-f(0, \eta_2)|^2}{|\zeta_2-\eta_2|^2}|d\zeta_2||d\eta_2|\\
\nonumber &+ \frac{1}{(2\pi)^4}\int_{\T^4}\frac{|f(\zeta_1,\zeta_2)-f(\zeta_1,\eta_2)-f(\eta_1,\zeta_2)+f(\eta_1,\eta_2)|^2}{|\zeta_1-\eta_1|^2|\zeta_2-\eta_2|^2}|d\zeta||d\eta| \\
\nonumber&= |a_{00}|^2 + \sum_{k=1}^{\infty} k | a_{k0}| ^2 +  \sum_{\ell=1}^{\infty} \ell | a_{0\ell}| ^2  + \sum_{k, \ell =1}^{\infty} k \ell | a_{k\ell}|^2.
\end{aligned}
\]
This implies that Doug($f$)  defines an equivalent norm on $\mathfrak{D}$.
\end{lemma}
\begin{proof} Clearly $|f(0,0)|^2 = |a_{00}|^2$ and as $f(\cdot, 0) \in D$, then Lemma $1.2.5$ in  \cite{EKMRBook} implies that
\[ \frac{1}{4\pi^2}\int_{\T^2}\frac{|f(\zeta_1, 0)-f(\eta_1,0)|^2}{|\zeta_1-\eta_1|^2}|d\zeta_1||d\eta_1| =  \sum_{k=1}^{\infty} k | a_{k0}| ^2.\]
The formula for $f(0,\cdot)$ is similar. Now we implement a two-variable version of the proof of Lemma $1.2.5$ from  \cite{EKMRBook}. Specifically, 
perform the change of variables $\zeta_j = e^{i(s_j +t_j)}$ and $\eta_j =e^{it_j}$ for $j=1,2.$ Then the integral over $\T^4$ in the formula for Doug$(f)$ becomes
\[ \frac{1}{(2\pi)^4} \int_{[0,2\pi]^2}  \int_{[0,2\pi]^2}  \frac{ \left | f(e^{i(s_1+t_1)} ,e^{i(s_2+t_2)} )-f(e^{i(s_1+t_1)} ,e^{it_2})-f(e^{it_1},e^{i(s_2+t_2)} )+f(e^{it_1},e^{it_2}) \right|^2}{|e^{is_1}-1|^2|e^{is_2}-1|^2} dt  ds.\]
Just as in the proof of  Lemma $1.2.5$, we can apply Parseval's formula to the function $(\eta_1, \eta_2) \mapsto f(\eta_1 e^{is_1},\eta_2 e^{is_2} )-f(\eta_1e^{is_1} , \eta_2)-f(\eta_1,\eta_2e^{is_2} )+f(\eta_1,\eta_2)$ to reduce the above integral to 
\[ \frac{1}{ 4 \pi^2} \sum_{k,\ell=0}^{\infty} |a_{k\ell}|^2  \int_{[0,2\pi]^2}   \frac{| e^{is_1k} -1|^2 | e^{is_2\ell}-1|^2}{|e^{is_1}-1|^2|e^{is_2}-1|^2} ds =  \sum_{k, \ell =1}^{\infty} k \ell | a_{k\ell}|^2,\]
  via another application of Parseval's formula.
 \end{proof}

We define the {\it local Dirichlet integral for the bidisk} at a point $(\zeta_1,\zeta_2)\in \T^2$ by setting
\begin{equation}
\mathfrak{D}_{(\zeta_1,\zeta_2)}(f):=\frac{1}{4\pi^2}\int_{\T^2}\frac{|f(\zeta_1,\zeta_2)-f(\zeta_1,\eta_2)-f(\eta_1,\zeta_2)+f(\eta_1,\eta_2)|^2}{|\zeta_1-\eta_1|^2|\zeta_2-\eta_2|^2}|d\eta|.
\label{douglasbidisk}
\end{equation}
The integral above is to be interpreted as being infinite if the radial limit $f(\zeta_1, \zeta_2)$ does not exist. As in the one-variable case, $f\in \mathfrak{D}$ implies $\int_{\T^2}\mathfrak{D}_{(\zeta_1,\zeta_2)}(f)|d\zeta|<\infty$.   We immediately obtain an analog of \eqref{ortholocal}.

 \begin{lemma}\label{lem:RIFlocal}
Let $\phi = \frac{\tilde{p}}{p}$ be a rational inner function on $\mathbb{D}^2$ and let $(z_1, z_2) \in \mathbb{D}^2$. Then 
\begin{equation} \label{eqn:RIFlocal}\mathfrak{D}_{(z_1,z_2)}(\phi)=\frac{1}{4\pi^2}\int_{\T^2}\frac{1-|\phi(\eta_1,z_2)|^2-|\phi(z_1,\eta_2)|^2+|\phi(z_1,z_2)|^2}{|\eta_1-z_1|^2|\eta_2-z_2|^2}|d\eta|.\end{equation}
\end{lemma}
\begin{proof} Fix $(z_1, z_2) \in \mathbb{D}^2$. Then for almost every $\eta_1 \in \T$, the function
\[g_{(\eta_1,z_1)}(\lambda_2):=\phi(\eta_1,\lambda_2)-\phi(z_1,\lambda_2) \quad \forall \ \lambda_2 \in \D,\]
belongs to the Hardy space of the unit disk.  Hence by \eqref{ortholocal}, 
\[\int_{\T}\frac{|g_{(\eta_1,z_1)}(\eta_2)-g_{(\eta_1,z_1)}(z_2)|^2}{|\eta_2-z_2|^2}|d\eta_2|=\int_{\T}\frac{|g_{(\eta_1,z_1)}(\eta_2)|^2-|g_{(\eta_1,z_1)}(z_2)|^2}{|\eta_2-z_2|^2}|d\eta_2|.\]
Repeating this argument using the functions $h_{\eta_2}(\lambda_1):=\phi(\lambda_1,\eta_2)$ and $h_{z_2}(\lambda_1):=\phi(\lambda_1,z_2)$ and integrating with respect to $|d\eta_1|$, we arrive at the formula
\[ \mathfrak{D}_{(z_1,z_2)}(\phi)=\frac{1}{4\pi^2}\int_{\T^2}\frac{|\phi(\eta_1, \eta_2)|^2-|\phi(\eta_1,z_2)|^2-|\phi(z_1,\eta_2)|^2+|\phi(z_1,z_2)|^2}{|\eta_1-z_1|^2|\eta_2-z_2|^2}|d\eta|. 
\] 
Now, simply note that for every $(\eta_1, \eta_2) \in \T^2$, the radial limit $\phi(\eta_1, \eta_2)$ exists and satisfies $|\phi(\eta_1, \eta_2)|=1$; this follows from Lemma \ref{lem:Bpoint} and Proposition $2.8$ in \cite{AMY12}.
\end{proof}


\section{RIFs in Dirichlet Spaces: Contact Order}\label{sect: dirichletcontact}

The goal of this section is to prove the following theorem, which connects the inclusion of rational inner functions in $\Da$ spaces to  the inclusion of their derivatives in related $H^\p(\mathbb{D}^2)$ spaces.

\begin{theorem}  \label{thm:Dp} Let $\phi = \frac{\tilde{p}}{p}$ be a rational inner function on $\mathbb{D}^2$ with $\deg p=(m,n)$. Then for $0 < \p < \infty$, if $\frac{\partial \phi}{\partial z_1} \in H^\p(\mathbb{D}^2)$, then $\phi \in \mathfrak{D}_{(\p,0)}$ and if $\frac{\partial \phi}{\partial z_2} \in H^\p(\mathbb{D}^2)$, then $\phi \in \mathfrak{D}_{(0,\p)}$ .
\end{theorem}

\begin{remark} \label{rem:Prange} In Theorem \ref{thm:Dp}, the important $\p$ range is $1 \le \p < \frac{3}{2}.$ Indeed, if we have Theorem \ref{thm:Dp} on this range, for every $\phi$ and $\p=1$, Theorem \ref{thm:contact} implies that $\frac{\partial \phi}{\partial z_1}, \frac{\partial \phi}{\partial z_2} \in H^1(\mathbb{D}^2)$, and so, every $\phi \in \mathfrak{D}_{(1,0)} \cap  \mathfrak{D}_{(0,1)}$.  Due to the nested property of our spaces, for $0<\p<1$, we immediately obtain $\frac{\partial \phi}{\partial z_1}, \frac{\partial \phi}{\partial z_2} \in H^\p(\mathbb{D}^2)$ and $\phi \in \mathfrak{D}_{(\p,0)} \cap  \mathfrak{D}_{(0,\p)}$.

Similarly, if $p \ge \frac{3}{2}$, Theorem \ref{thm:Dp} is trivial. Namely, Theorem \ref{thm:Hp} implies that  $\frac{\partial \phi}{\partial z_1}, \frac{\partial \phi}{\partial z_2} \in H^\p(\mathbb{D}^2)$ if and only if $\phi$ is continuous on $\overline{\mathbb{D}^2}$. So, the statement is vacuous for $\phi$ with singularities on $\partial(\mathbb{D}^2)$. Conversely, by the continuous edge-of-the-wedge
theorem in \cite{Rud71}, if $\phi = \frac{\tilde{p}}{p}$ is continuous on $\overline{\mathbb{D}^2}$, then $p$ has no zeros on $\overline{\mathbb{D}^2}$. Thus, $\phi$ is holomorphic on an open set containing $\overline{\mathbb{D}^2}$ and belongs to every $\mathfrak{D}_{\vec{\alpha}}.$ 
 \end{remark}

By Remark \ref{rem:Prange}, the proof of Theorem \ref{thm:Dp} will only address the range $1 \le \p < \frac{3}{2}.$  Theorem \ref{thm:Dp} also has the following important corollary.

\begin{corollary} \label{cor:Da} Let $\phi = \frac{\tilde{p}}{p}$ be a rational inner function on $\mathbb{D}^2$. If $\frac{\partial \phi}{\partial z_1} \in H^\p(\mathbb{D}^2)$ and $\frac{\partial \phi}{\partial z_2} \in H^\q(\mathbb{D}^2)$, then $\phi \in \mathfrak{D}_{(\frac{\p}{2}, \frac{\q}{2})}$. More specifically, if  $\frac{\partial \phi}{\partial z_1}, \frac{\partial \phi}{\partial z_2} \in H^\p(\mathbb{D}^2),$ then $\phi \in \mathfrak{D}_{\frac{\p}{2}}.$
\end{corollary}

\begin{proof} Assume $\frac{\partial \phi}{\partial z_1} \in H^\p(\mathbb{D}^2)$ and $\frac{\partial \phi}{\partial z_2} \in H^\q(\mathbb{D}^2)$. By Theorem \ref{thm:Dp}, the assumptions imply $\phi \in  \mathfrak{D}_{(\p,0)} \cap  \mathfrak{D}_{(0,\q)}.$ Write $\phi(z) = \sum_{k, \ell =0}^{\infty} a_{k\ell}z_1^k z_2^{\ell}$. Then an application of the Cauchy-Schwarz inequality gives
\[ 
\begin{aligned}
\| \phi \|_{\mathfrak{D}_{\left(\frac{\p}{2}, \frac{\q}{2}\right)}}^2 &= \sum_{k, \ell=0}^{\infty}  (k+1)^{\frac{\p}{2}} (\ell+1)^{\frac{\q}{2}} |a_{k\ell}|^2  \\
&\le \left( \sum_{k, \ell=0}^{\infty}  (k+1)^{\p}|a_{k\ell}|^2 \right)^{\frac{1}{2}} \left( \sum_{k, \ell=0}^{\infty}  (\ell+1)^{\q} |a_{k\ell}|^2 \right)^{\frac{1}{2}} \\
& = \| \phi \|_{\mathfrak{D}_{(\p,0)}} \| \phi \|_{\mathfrak{D}_{(0,\q)}}< \infty,
\end{aligned}
\]
as needed. \end{proof}


\subsection{Proof of Theorem \ref{thm:Dp}}

We first require the following lemma. The lemma is likely true for $\p\ge2$, but as this is not required to prove Theorem \ref{thm:Dp}, we restrict to the interval $1 \le \p <2$ for an easier proof.

\begin{lemma} \label{lem:Dp} Let $b$ be a finite Blaschke product $b(z) := \prod_{j=1}^n b_{\alpha_j}(z), \text{ with } b_{\alpha_j}(z) = \frac{z-\alpha_j}{1-\bar{\alpha}_j z}$ for $\alpha_j \in \mathbb{D}.$ Then $b$ satisfies
\[ \| b \|^2_{D_\p} \lesssim \epsilon(b)^{1-\p} \quad \text{ for } 1 \le \p < 2,\]
where $ \epsilon(b):= \min\{1-|\alpha_j| : 1 \le j \le n\}$ is the distance from the zero set of $b$ to $\mathbb{T}$ and the implied constant depends on $\p$ and $\deg b =n$.
\end{lemma}

\begin{proof}Since $1 \le \p < 2$, we have the following equivalent norm for the one-variable space $D_\p$
\[ \| f \|^2_{D_\p} \approx |f(0)|^2 +  \int_{\mathbb{D}} |f'(z)|^2 (1-|z|^2)^{1-\p} \ dA(z),\]
where $d A(z)$ is normalized area measure on $\mathbb{D}$. See for example, \cite[Chapter 1.6]{EKMRBook}.
Applying this to $b$, computing $b'$, and using each $|b_{\alpha_j}(z)| \le 1$ on $\mathbb{D}$, we have
\[  
\begin{aligned}
 \| b \|^2_{D_\p}& \approx |b(0)|^2 +  \int_{\mathbb{D}} |b'(z)|^2 (1-|z|^2)^{1-\p} \ dA(z) \\
&\lesssim \sum_{j=1}^n |b_{\alpha_j}(0)|^2 + 
\int_{\mathbb{D}} \left( \sum_{j=1}^n | b_{\alpha_j}'(z)| \right)^2 (1-|z|^2)^{1-\p} \ dA(z) \\
&\lesssim \sum_{j=1}^{n} \| b_{\alpha_j} \|^2_{D_\p},
\end{aligned}
\]
where the implied constants depend on $n=\deg b.$
Now we simply need to show that for $\alpha \in \mathbb{D}$, $\| b_{\alpha} \|^2_{D_\p}  \lesssim (1-|\alpha|)^{1-\p}$ and apply the following 
computation, using $\p\ge 1$:
\[ \epsilon(b)^{1-\p}:= \left( \min\{1-|\alpha_j| \} \right)^{1-\p} = \max\{ (1-|\alpha_j| )^{1-\p} \} \approx \sum_{j=1}^{n} (1-|\alpha_j| )^{1-\p},\]
where the implied constants depend on $n$.
 
So we are done once we show $\| b_{\alpha} \|^2_{D_\p} \lesssim (1-|\alpha|)^{1-\p}$, for an implied constant depending on $\p$. We first go on a slight detour. Observe that 
\[ g(x) := \left( \frac{1}{1-x}\right)^{\p+1} = 1+ \sum_{k=1}^{\infty} \frac{(\p+1)(\p+2)\cdots (\p+k)}{k!} x^k  \ \text{ for } \ x\in (-1,1).\] 
Using well-known properties of the Gamma function $\Gamma(x)$, we have
\[ \frac{(\p+1)(\p+2)\cdots (\p+k)}{k!} = \frac{ \Gamma(\p+k)}{\Gamma(\p) \Gamma(k)} \approx k^\p\]
for large $k$, say $k \ge K_\p$, since $\frac{\Gamma(\p+k)}{\Gamma(k)} \approx k^\p$ for large $k$.
Returning to $b_{\alpha}$, observe that  
\[ b_{\alpha}(z) = \frac{z-\alpha}{1-\bar{\alpha}z} = (z-\alpha)\sum_{k=0}^{\infty} \bar{\alpha}^k z^k = -\alpha + \sum_{k=1}^{\infty} \left( 1- |\alpha|^2 \right)\bar{\alpha}^{k-1}z^k,\]
which implies that 
\begin{equation} \label{eqn:bDp1}  \|  b_{\alpha}\|^2_{D_\p} =  |\alpha|^2 + \sum_{k=1}^{\infty} (k+1)^\p (1-|\alpha|^2)^2 |\alpha|^{2k-2}.\end{equation}
Fix $\epsilon >0$ such that if $|\alpha| \ge 1-\epsilon$, then 
\[ 1 + (1-|\alpha|^2)^2  \sum_{k=1}^{K_\p-1} (k+1)^\p  \le  2(1-|\alpha|^2)^{1-\p},\]
where our choice of $\epsilon$ depends on $\p$. From this and \eqref{eqn:bDp1}, we can conclude that if $|\alpha| \ge 1 - \epsilon$, then
\[ 
\begin{aligned}
\|  b_{\alpha}\|^2_{D_\p} &\approx 1 + \sum_{k=1}^{\infty} (k+1)^\p (1-|\alpha|^2)^2 |\alpha|^{2k} \\
& \lesssim  1 +(1-|\alpha|^2)^2 \left( \sum_{k=1}^{K_\p-1} (k+1)^\p   +  \sum_{k=K_\p}^{\infty} \frac{(\p+1)(\p+2)\cdots (\p+k)}{k!} |\alpha|^{2k} \right)\\
&\le  1 + (1-|\alpha|^2)^2  \sum_{k=1}^{K_\p-1} (k+1)^\p +  (1-|\alpha|^2)^2 \left( \frac{1}{1-|\alpha|^2}\right)^{\p+1}\\
&\lesssim (1-|\alpha|^2)^{1-\p}, \\
\end{aligned}
\]
where the implied constants depend on $\p$ (and $\epsilon$, which depends on $\p$).
Similarly if $|\alpha| < 1 - \epsilon$, we have 
\[ \| b_{\alpha} \|^2_{D_\p} \le  1 + \sum_{k=1}^{\infty} (k+1)^\p  ( 1-\epsilon)^{2k-2} = C(\epsilon)  \lesssim  (1-|\alpha|^2)^{1-\p},\]
where the implied constant again depends on $\p$ (and $\epsilon$).
\end{proof}

From this, the proof of Theorem \ref{thm:Dp} is mostly an application of Theorem \ref{thm:contact}:

\begin{proof} As discussed in Remark \ref{rem:Prange}, we restrict attention to $1 \le \p <\frac{3}{2}$ and assume $\frac{\partial \phi}{\partial z_1} \in H^\p(\mathbb{D}^2)$. 
Then by Theorem \ref{thm:contact}, $K_1 < \frac{1}{\p-1}$. Define a sequence of functions $\{\phi_k\}$ by
\[ \phi(z) = \sum_{k, \ell=0}^{\infty} a_{k\ell} z_1^k z_2^{\ell} =  \sum_{k=0}^{\infty} \phi_{k}(z_2) z_1^k, \ \text{ so } \ \| \phi_k \|_{H^2(\mathbb{D})}^2 = \sum_{\ell=0}^{\infty} |a_{k\ell}|^2.\]
Then for almost every $\zeta_2 \in \mathbb{T}$,  each $\phi_k$ has a radial boundary value at $\zeta_2$, denoted $\phi_k(\zeta_2)$ and
the function $\phi_{\zeta_2}:=\phi(\cdot, \zeta_2) \in H^2(\mathbb{D})$ is a Blaschke product with $\deg \phi_{\zeta_2}=m.$ Furthermore by Lemma \ref{lem:Dp},
\[  \sum_{k=0}^{\infty} (k+1)^\p |\phi_k(\zeta_2)|^2 = \| \phi_{\zeta_2} \|^2_{D_\p}  \lesssim  \epsilon(\phi_{\zeta_2})^{1-\p} = \epsilon(\phi, \zeta_2)^{1-\p},\]
 where the implied constant depends on $\p$ but not on $\zeta_2.$Then we can compute
\[
\begin{aligned}
 \| \phi \|_{\mathfrak{D}_{(\p,0)}}^2 &= \sum_{k=0}^{\infty}(k+1)^\p \left( \sum_{\ell=0}^{\infty}  |a_{k\ell}|^2 \right) \\
& = \sum_{k=0}^{\infty} (k+1)^\p \| \phi_k \|^2_{H^2(\mathbb{D})}  \\
& = \frac{1}{2\pi}\int_{\mathbb{T}} \sum_{k=0}^{\infty} (k+1)^\p |\phi_k(\zeta_2)|^2 |d\zeta_2| \\
& =  \frac{1}{2\pi}\int_{\mathbb{T}} \| \phi_{\zeta_2}\|^2_{D_\p} |d\zeta_2| \\
&\lesssim \int_{\mathbb{T}} \epsilon(\phi, \zeta_2)^{1-\p} |d\zeta_2| <\infty,
\end{aligned}
\]
where the last statement uses $K_1 < \frac{1}{\p-1}$ and follows via the arguments used in the proof of Theorem \ref{thm:contact}. \end{proof}


\section{RIFs in Dirichlet Spaces: Agler Decompositions}\label{sect: aglerdirichlet}

In this section, we use local Dirichlet integrals and Agler decompositions to show that certain rational inner functions fail to be in the two-variable
Dirichlet space $\mathfrak{D}.$

In the following lemma, we obtain an alternate formula for $\mathfrak{D}_{(z_1, z_2)}(\phi)$, when $\phi$ is rational inner. To obtain the result,
let $K_1$ and $K_2$ be the Agler kernels defined in \eqref{eqn:Akernels} with expansions given  in \eqref{eqn:Akernels2}. Here, one can take $\{ \frac{r_1}{p}, \dots, \frac{r_m}{p} \}$ and $\{ \frac{q_1}{p}, \dots, \frac{q_n}{p} \}$ to be orthonormal bases of $\mathcal{H}(K_1)$ and $\mathcal{H}(K_2)$ respectively.

\begin{lemma} \label{lem:local2} For rational inner $\phi = \frac{\tilde{p}}{p}$ with $\deg \phi = (m,n)$, we have
\begin{equation} \label{eqn:RIFlocal2}\mathfrak{D}_{(z_1, z_2)}(\phi) =  \frac{1}{2\pi} \left( \sum_{j=1}^n \int_{\mathbb{T}} \frac{ | \frac{q_j}{p}(\eta_1, z_2) - \frac{q_j}{p}(z_1, z_2) |^2}{|z_1 - \eta_1|^2} | d \eta_1| + \sum_{k=1}^m \int_{\mathbb{T}} \frac{ | \frac{r_k}{p}(z_1, \eta_2) - \frac{r_k}{p}(z_1, z_2) |^2}{|z_2- \eta_2|^2} | d \eta_2| \right),   \end{equation}
for all $(z_1, z_2) \in \mathbb{D}^2$, where the functions $\{q_j\}$, $\{r_k\}$ are from \eqref{eqn:Akernels2}. Moreover,  if $\phi$ does not have a singularity at $(\zeta_1, \zeta_2) \in \mathbb{T}^2$, then \eqref{eqn:RIFlocal2} also holds for $\mathfrak{D}_{(\zeta_1, \zeta_2)}(\phi)$. 
\end{lemma} 

\begin{proof} First fix $z=(z_1, z_2) \in \mathbb{D}^2.$ Since $p$ has no zeros on $(\mathbb{D} \times \mathbb{T}) \cup (\mathbb{T} \times \mathbb{D})$, for every $\eta = (\eta_1, \eta_2) \in \mathbb{T}^2$, we can extend the Agler kernel formulas to the points $(z_1, \eta_2)$ and $(\eta_1, z_2)$ as follows:
\[ 
\begin{aligned}
1- |\phi(z)|^2 &= (1- |z_1|^2) \sum_{k=1}^m \Big | \frac{r_k}{p}(z)\Big|^2 + (1-|z_2|^2) \sum_{j=1}^n \Big| \frac{q_j}{p}(z) \Big |^2; \\
1- |\phi(z_1, \eta_2)|^2 &= (1- |z_1|^2) \sum_{k=1}^m \Big | \frac{r_k}{p}(z_1, \eta_2)\Big|^2; \\
1- |\phi(\eta_1, z_2)|^2 &=  (1-|z_2|^2) \sum_{j=1}^n \Big| \frac{q_j}{p}(\eta_1, z_2) \Big |^2.
\end{aligned}
\]
By manipulating these equations, we can conclude that 
\[ 
\begin{aligned}
1- |\phi(\eta_1, z_2)|^2 - |\phi(z_1, \eta_2)|^2 + |\phi(z)|^2& = (1-|z_1|^2) \left( \sum_{k=1}^m \Big | \frac{r_k}{p}(z_1, \eta_2)\Big|^2 - \sum_{k=1}^m \Big | \frac{r_k}{p}(z)\Big|^2   \right) \\
&+ (1-|z_2|^2) \left( \sum_{j=1}^n \Big| \frac{q_j}{p}(\eta_1, z_2) \Big |^2 -  \sum_{j=1}^n \Big| \frac{q_j}{p}(z) \Big |^2 \right).
\end{aligned}
\] 
Plugging this equation into the numerator of \eqref{eqn:RIFlocal} and using the fact that 
\[  \frac{1}{2 \pi } \int_{\mathbb{T}} \frac{1-|z_j|^2}{|z_j-\eta_j|^2}  |d \eta_j | = 1,\]
for $j=1,2$, we obtain 
\[ \mathfrak{D}_{(z_1, z_2)}(\phi) = \frac{1}{2 \pi} \left( \sum_{j=1}^n \int_{\mathbb{T}} \frac{ | \frac{q_j}{p}(\eta_1, z_2)|^2 - |\frac{q_j}{p}(z_1, z_2) |^2}{|z_1-\eta_1|^2} | d \eta_1| + \sum_{k=1}^m \int_{\mathbb{T}} \frac{ | \frac{r_k}{p}(z_1, \eta_2)|^2 - |\frac{r_k}{p}(z_1, z_2) |^2}{|z_2- \eta_2|^2} | d \eta_2| \right).\]
Since $p$ is nonzero on $(\mathbb{D} \times \overline{\D}) \cup ( \overline{\D} \times \mathbb{D})$, the slice functions $\frac{q_j}{p}(\cdot, z_2)$, $\frac{r_k}{p}(z_1, \cdot)$ are all in $H^2(\mathbb{D}),$ and we can use  \eqref{ortholocal} to obtain \eqref{eqn:RIFlocal2}.

Now, we extend \eqref{eqn:RIFlocal2} to points on $\mathbb{T}^2$. Fix $(\zeta_1, \zeta_2) \in \mathbb{T}^2$ and assume $\phi$ does not have a singularity at $(\zeta_1, \zeta_2)$. We will prove the following string of equalities:

\begin{align}
\notag \mathfrak{D}_{(\zeta_1, \zeta_2)}(\phi)
\notag & = \frac{1}{4\pi^2} \int_{\T^2}\frac{|\phi(\zeta_1,\zeta_2)-\phi(\zeta_1,\eta_2)-\phi(\eta_1,\zeta_2)+\phi(\eta_1,\eta_2)|^2}{|\zeta_1-\eta_1|^2|\zeta_2-\eta_2|^2}|d\eta| \\
\label{eqn:limit1} & = \lim_{r \nearrow 1}  \frac{1}{4 \pi^2}\int_{\T^2}\frac{|\phi(r\zeta_1,r\zeta_2)-\phi(r\zeta_1,\eta_2)-\phi(\eta_1,r\zeta_2)+\phi(\eta_1,\eta_2)|^2}{|r\zeta_1-\eta_1|^2|r\zeta_2-\eta_2|^2}|d\eta|  \\
\notag & =\lim_{r \nearrow 1}  \mathfrak{D}_{(r \zeta_1, r \zeta_2)}(\phi) \\
\label{eqn:local3} & =  \lim_{r \nearrow 1} \frac{1}{2\pi} \left( \sum_{j=1}^n \int_{\mathbb{T}} \frac{ | \frac{q_j}{p}(\eta_1, r \zeta_2) - \frac{q_j}{p}(r\zeta_1, r\zeta_2) |^2}{|r\zeta_1 - \eta_1|^2} | d \eta_1|\right.\\ \notag & \left.\qquad \qquad \qquad+\sum_{k=1}^m \int_{\mathbb{T}} \frac{ | \frac{r_k}{p}(r \zeta_1, \eta_2) - \frac{r_k}{p}(r \zeta_1, r\zeta_2) |^2}{|r \zeta_2- \eta_2|^2} | d \eta_2| \right) \\
 \label{eqn:limit2} & = \frac{1}{2\pi} \left( \sum_{j=1}^n \int_{\mathbb{T}} \frac{ | \frac{q_j}{p}(\eta_1,  \zeta_2) - \frac{q_j}{p}(\zeta_1, \zeta_2) |^2}{|\zeta_1 - \eta_1|^2} | d \eta_1| + \sum_{k=1}^m \int_{\mathbb{T}} \frac{ | \frac{r_k}{p}(\zeta_1, \eta_2) - \frac{r_k}{p}(\zeta_1, \zeta_2) |^2}{|\zeta_2- \eta_2|^2} | d \eta_2| \right).
 \end{align}
Equation \eqref{eqn:local3} follows from the derived formula \eqref{eqn:RIFlocal2}  for points in $\mathbb{D}^2$. So, to prove the string of equalities, we just need to justify the switching of the limits in \eqref{eqn:limit1} and \eqref{eqn:limit2}.

First consider \eqref{eqn:limit1} and define the rational function $G$ as follows:
\[ 
G(z_1, z_2, \eta_1, \eta_2) := \phi(z_1,z_2)-\phi(z_1,\eta_2)-\phi(\eta_1,z_2)+\phi(\eta_1,\eta_2)
= \frac{Q(z_1, z_2, \eta_1, \eta_2)}{ p(z_1, z_2)p(z_1, \eta_2) p(\eta_1, z_2)p(\eta_1, \eta_2)},
\]
for a polynomial $Q$ of four variables. Moreover, $Q$ vanishes whenever $z_{\ell} = \eta_{\ell}$ for $\ell =1,2$ so $Q$ is divisible by both $z_1-\eta_1$ and $z_2 -\eta_2$. This is a common conclusion we make and follows from an  application of Hilbert's Nullstellensatz. Specifically, Hilbert's Nullstellensatz implies that $Q^N$ is in the ideal generated by  each $z_{\ell}-\eta_{\ell}$ for some positive integer $N$. This means $z_{\ell}-\eta_{\ell}$ divides $Q^N$ and since each $z_{\ell}-\eta_{\ell}$ is irreducible, it must actually divide $Q$. 
This means we can write
\begin{equation} \label{eqn:G} \frac{G(z_1, z_2, \eta_1, \eta_2)}{(z_1-\eta_1)(z_2-\eta_2)} := \frac{R(z_1, z_2, \eta_1, \eta_2)}{  p(z_1, z_2)p(z_1, \eta_2) p(\eta_1, z_2)p(\eta_1, \eta_2)},
\end{equation}
for a polynomial $R$. Notice that for $0 \le r \le 1$,  
\[ \left | \frac{G(r\zeta_1, r\zeta_2, \eta_1, \eta_2)}{(r\zeta_1-\eta_1)(r\zeta_2-\eta_2)} \right |^2 = \left |\frac{R(r \zeta_1, r\zeta_2, \eta_1, \eta_2)}{  p(r\zeta_1, r\zeta_2)p(r\zeta_1, \eta_2) p(\eta_1, r\zeta_2)p(\eta_1, \eta_2)} \right |^2 \]
is exactly the function appearing in the integral \eqref{eqn:limit1}. 
Then, for all $0 \le r \le 1$, this function is bounded (independent of $r$) as long as we restrict the $(\eta_1, \eta_2)$ to closed subsets of $\mathbb{T}^2$ such that $p$ does not vanish at $(\eta_1, \eta_2)$, $(\zeta_1, \eta_2)$, $(\eta_1, \zeta_2)$ for any $(\eta_1, \eta_2)$ in the set.  Since $p$ has a finite number of zeros on $\T^2,$ for each $\epsilon >0$, there is an open set $S_{\epsilon} \subset \mathbb{T}^2$, which contains any troublesome $(\eta_1, \eta_2)$ values and satisfies $|S_{\epsilon}|< \epsilon.$ Then, we can use the Dominated Convergence Theorem on $\mathbb{T}^2 \setminus S_{\epsilon}$ to conclude that  
\[
\begin{aligned} \int_{\T^2 \setminus S_{\epsilon}} &\frac{|\phi(\zeta_1,\zeta_2) -\phi(\zeta_1,\eta_2)-\phi(\eta_1,\zeta_2)+\phi(\eta_1,\eta_2)|^2}{|\zeta_1-\eta_1|^2|\zeta_2-\eta_2|^2}|d\eta| \\
&\qquad  = \lim_{r \nearrow 1}  \int_{\T^2 \setminus S_{\epsilon}} \frac{|\phi(r\zeta_1,r\zeta_2)-\phi(r\zeta_1,\eta_2)-\phi(\eta_1,r\zeta_2)+\phi(\eta_1,\eta_2)|^2}{|r\zeta_1-\eta_1|^2|r\zeta_2-\eta_2|^2}|d\eta|,
\end{aligned}
\]
which in turn gives \eqref{eqn:limit1}. Analogous arguments (writing out the functions of interest, factoring out linear terms to get rational functions that are bounded away from a finite set, and using the Dominated Convergence Theorem) give \eqref{eqn:limit2}. The desired formula for $\mathfrak{D}_{(\zeta_1, \zeta_2)}(\phi)$ follows immediately.
%
\end{proof}

We can use Lemma \ref{lem:local2} to show that certain rational inner functions are not in the two-variable Dirichlet space. 

\begin{theorem} \label{thm:1D} Let $\phi = \frac{\tilde{p}}{p}$ be a rational inner function with $\deg p = \deg \phi= (1,n)$, with $n\ge 2.$ Furthermore assume that $N_{\mathbb{T}^2}(p ,\tilde{p}) =2n$, $p$ is irreducible, and all of the zeros of $p$ on $\mathbb{T}^2$ occur at $(1,1)$. Then $\phi$ is not in the two-variable Dirichlet space $\mathfrak{D}.$
\end{theorem}

\begin{proof} By Lemma \ref{lem:norm}, the function $\phi$ will fail to be in $\mathfrak{D}$ if the integral
\begin{align*}
&\frac{1}{(2\pi)^4}\int_{\T^4} \frac{|\phi(\zeta_1,\zeta_2)-\phi(\zeta_1,\eta_2)-\phi(\eta_1,\zeta_2)+\phi(\eta_1,\eta_2)|^2}{|\zeta_1-\eta_1|^2|\zeta_2-\eta_2|^2}|d\zeta||d\eta| \\
&= \frac{1}{4 \pi^2} \int_{\mathbb{T}^2} \mathfrak{D}_{(\zeta_1, \zeta_2)}(\phi) |d \zeta| \\
&=\frac{1}{(2\pi)^3} \int_{\mathbb{T}^2} \left( \sum_{j=1}^n \int_{\mathbb{T}} \frac{ | \frac{q_j}{p}(\eta_1,  \zeta_2) - \frac{q_j}{p}(\zeta_1, \zeta_2) |^2}{|\zeta_1 - \eta_1|^2} | d \eta_1| + \ \int_{\mathbb{T}} \frac{ | \frac{r_1}{p}(\zeta_1, \eta_2) - \frac{r_1}{p}(\zeta_1, \zeta_2) |^2}{|\zeta_2- \eta_2|^2} | d \eta_2| \right) |d \zeta|
 \end{align*}
 is infinite. Specifically, we claim that at least one 
\begin{equation} \label{eqn:infinite} \int_{\mathbb{T}^2} \int_{\mathbb{T}} \frac{ | \frac{q_k}{p}(\eta_1,  \zeta_2) - \frac{q_k}{p}(\zeta_1, \zeta_2) |^2}{|\zeta_1 - \eta_1|^2} | d \eta_1| |d \zeta|= \infty,\end{equation}
for $\frac{q_k}{p}$ a basis vector of $\mathcal{H}(K_2)$ from \eqref{eqn:Akernels2}. This will prove the desired claim.

First, for $1 \le j \le n$, let $\frac{q_j}{p}$ be a basis vector of $\mathcal{H}(K_2)$ as in \eqref{eqn:Akernels2}. Then $\deg q_j \le (1, n-1)$ and 
\[ \frac{q_j}{p}(\lambda_1, z_2) - \frac{q_j}{p}(z_1, z_2) = \frac{R_j(\lambda_1, z_1, z_2)}{p(\lambda_1, z_2)p(z_1, z_2)}.\] 
for a three-variable polynomial $R_j$ with $\deg R_j \le (1, 1, 2n-1).$ Since $R_j$ vanishes whenever $z_1 = \lambda_1$, it must be divisible by $\lambda_1 -z_1$. This implies that we can write
\[  \frac{\frac{q_j}{p}(\lambda_1, z_2) - \frac{q_j}{p}(z_1, z_2) }{ \lambda_1 - z_1} = \frac{Q_j(z_2)}{p(\lambda_1, z_2)p(z_1, z_2)},\]
for a one-variable polynomial $Q_j$ with $\deg Q_j \le 2n-1.$ Then we can write
\[ Q_j(z_2) = (z_2-1)^{m_j} r_j (z_2),\]
for some $m_j \le 2n-1$ and polynomial $r_j$ with $r(1) \ne 0.$ Observe that $Q_j \not \equiv 0$ because that would imply $\frac{q_j}{p}$ is a function of only one variable. This would give a contradiction because $\deg p = (1,n)$ and $p$ is irreducible.

We claim that for some $\frac{q_k}{p}$, we actually have $m_k \le 2n-2.$ By way of contradiction, assume not. Fix any $\frac{q_j}{p},$ $\frac{q_i}{p}$ with $j \ne i$ from \eqref{eqn:Akernels2}. Since $n \ge 2$, there are at least two such basis vectors of $\mathcal{H}(K_2)$.
Then there are nonzero constants $C_1$ and $C_2$ such that 
\[ \frac{\frac{q_i}{p}(\lambda_1, z_2) - \frac{q_i}{p}(z_1, z_2) }{ \lambda_1 - z_1}  = \frac{C_1(z_2-1)^{2n-1}}{p(\lambda_1, z_2)p(z_1, z_2)} = C_2\frac{\frac{q_j}{p}(\lambda_1, z_2) - \frac{q_j}{p}(z_1, z_2) }{ \lambda_1 - z_1}.\]
Taking limits as $\lambda_1 \rightarrow z_1$, this says
\[ \frac{ \partial}{\partial z_1}  \left( \frac{q_i}{p} \right)= C_2 \frac{ \partial}{\partial z_1} \left( \frac{q_j}{p} \right), \]
and so
\[  \frac{q_i}{p}(z) =C_2  \frac{q_j}{p}(z) + h(z_2),\]
for some rational function $h = \frac{h_1}{h_2},$ with $h_1, h_2$ polynomials. Since $\frac{q_i}{p}$ and $\frac{q_j}{p}$ are linearly independent, $h \not \equiv 0.$
Then rearranging terms implies that
\[ \left( q_i(z) - C_2q_j(z) \right) h_2(z_2) = h_1(z_2) p(z).\]
Since $p$ is irreducible with $\deg p =(1,n)$, this implies $p$ divides $q_i - C_2 q_j$. Then we obtain $(1, n-1) \ge \deg (q_i - C_2 q_j ) \ge \deg p = (1,n) $, a  contradiction.
Thus, there is some $\frac{q_{k}}{p}$ such that
\[  \frac{\frac{q_k}{p}(\lambda_1, z_2) - \frac{q_k}{p}(z_1, z_2) }{ \lambda_1 - z_1} = \frac{(z_2-1)^{m_k} r_k (z_2)}{p(\lambda_1, z_2)p(z_1, z_2)},\]
for $m_k \le 2n-2$ and a polynomial $r_k$ with $\deg r_k \le 2n-1-m_k$ and $r_k(1) \ne 0.$ Moreover, since $r_k$ does not vanish near $\phi$'s singularity, its presence does not affect whether the integral given below converges or diverges. This implies that 
\begin{align*}
\int_{\mathbb{T}^2} \int_{\mathbb{T}} \frac{ | \frac{q_k}{p}(\eta_1,  \zeta_2) - \frac{q_k}{p}(\zeta_1, \zeta_2) |^2}{|\zeta_1 - \eta_1|^2} | d \eta_1| |d \zeta| &= 
\int_{\mathbb{T}^2} \int_{\mathbb{T}} \left |  \frac{(\zeta_2-1)^{m_k} r_k (\zeta_2)}{p(\eta_1, \zeta_2)p(\zeta_1, \zeta_2)} \right |^2 | d \eta_1| |d \zeta| \\
&\approx \int_{\mathbb{T}} \left( \int_{\mathbb{T}}   \frac{|\zeta_2-1| ^{m_k}}{ |p(\eta_1, \zeta_2)|^2} |d\eta_1|\right) \left( \int_{\mathbb{T}}   \frac{|\zeta_2-1| ^{m_k}}{ |p(\zeta_1, \zeta_2)|^2} |d\zeta_1|\right) |d \zeta_2| \\
&  = \int_{\mathbb{T}}  \left( \int_{\mathbb{T}}   \frac{|\zeta_2-1| ^{m_k}}{ |p(\zeta_1, \zeta_2)|^2} |d\zeta_1|\right)^2 |d \zeta_2| \\
& \gtrsim \left( \int_{\mathbb{T}^2}   \frac{|\zeta_2-1| ^{m_k}}{ |p(\zeta_1, \zeta_2)|^2} |d\zeta|   \right)^2\\
& \gtrsim \left( \int_{\mathbb{T}^2}   \frac{|\zeta_2-1| ^{2n-2}}{ |p(\zeta_1, \zeta_2)|^2} |d\zeta|  \right)^2,
\end{align*}
where we used H\"older's inequality. 
If this last integral converged, then $ (1-\zeta_2)^{n-1}$ would be in the polynomial ideal $\mathcal{I}_p$. But, $\deg \left((1-\zeta_2)^{n-1}\right) = (0,n-1)$ and so
using the notation of \cite[Section 5]{Kne15}, this would imply that the set of polynomials $\mathcal{G} \ne \{0\}$, which by \cite[Cor 13.6]{Kne15} and its proof, contradicts  that $N_{\mathbb{T}^2} (p, \tilde{p}) = 2n$ is maximal. Thus, the integral must diverge, implying that $\phi$ is not in the Dirichlet space.
\end{proof}

\begin{remark}
 We believe that the conclusion of Theorem \ref{thm:1D} persists under much weaker assumptions. For instance, if $\phi$ is a RIF that satisfies the hypotheses of Theorem \ref{thm:1D} and $M,N\in \mathbb{N}$ are fixed, then the rational inner function
\[\phi_{M,N}(z_1,z_2)=\phi(z_1^M,z_2^N)\]
has $\deg \phi=(M,Nn)$ and has $MNn$ zeros on $\T^2$ located at roots of unity, yet also fails to belong to $\mathfrak{D}$. To see this, observe that the Taylor coefficients of $\phi_{M,N}$ satisfy
\[\widehat{\phi_{M,N}}(Mk,N\ell)=\widehat{\phi}(k,\ell)\]
and $\widehat{\phi_{M,N}}(n_1,n_2)=0$ if $M \nmid n_1$ or $N \nmid n_2$. Thus
\[\|\phi_{M,N}\|_{\mathfrak{D}}^2=\sum_{k, \ell=0}^{\infty}(Mk+1)(N\ell+1)|\widehat{\phi}(k,\ell)|^2\geq \sum_{k, \ell}^{\infty}(k+1)(\ell+1)|\widehat{\phi}(k,\ell)|^2,\]
and the series on the right diverges by Theorem \ref{thm:1D}.

Furthermore, in Section \ref{sect: examples}, we shall exhibit a RIF $\phi=\frac{\tilde{p}}{p}$
with a single singularity on $\T^2$ that does not belong to the Dirichlet space despite having $N_{\T^2}(p,\tilde{p})<N(p,\tilde{p})$ .
\end{remark}

\section{RIFs in Dirichlet-type Spaces: Inclusion criteria} \label{sec:Dincl}
We first state some elementary inclusion results for $\Dveca$ spaces. We will use some of these results when we explore specific examples in Section \ref{sect: examples}.
\begin{lemma}\label{lem: inclusion}
Let $\veca =(\alpha_1, \alpha_2) \in \R^2$. Then \begin{itemize}
\item $f\ \in \Dveca$ \ if and only if \ $\tfrac{\partial f}{\partial z_1} \in \mathfrak{D}_{(\alpha_1-2, \alpha_2)}$ \ and  \ $ f(0,\cdot)\in D_{\alpha_2}$; 
\item $f\ \in \Dveca$ \ if and only if \  $\tfrac{\partial f}{\partial z_2}  \in \mathfrak{D}_{(\alpha_1, \alpha_2-2)}$ \ and \ $f(\cdot,0)\in D_{\alpha_1}.$
\end{itemize}
Furthermore, if $f=\frac{q}{p}$, where $p$ is an atoral polynomial with no zeros in $\D^2$ and is not a polynomial in one variable only, then
\begin{itemize}
\item $f\ \in \Dveca$ \ if and only if \ $\tfrac{\partial f}{\partial z_1}  \in \mathfrak{D}_{(\alpha_1-2, \alpha_2)}$; 
\item $f\ \in \Dveca$ \ if and only if \ $\tfrac{\partial f}{\partial z_2}  \in \mathfrak{D}_{(\alpha_1, \alpha_2-2)}.$
\end{itemize}
\end{lemma}
\begin{proof}
The first two assertions follow upon considering the power series expansions of first derivatives. 
To obtain the second assertion, note that $p$ cannot vanish in $\D\times \T$ or $\T\times \D$.  This implies $\frac{1}{p(0,\cdot)}$ and $\frac{1}{p(\cdot, 0)}$ are both holomorphic on a neighborhood of $\overline{\mathbb{D}}$, and hence the functions  $f(0, \cdot), f(\cdot,0)$ are immediately in every $D_{\alpha}$.
\end{proof}

\begin{lemma}\label{lem: CSlemma}
If $f\in \mathfrak{D}_{(\alpha-2, \alpha)}\cap\mathfrak{D}_{(\alpha, \alpha-2)}$, then $f\in \mathfrak{D}_{\alpha-1}$.
\end{lemma}
\begin{proof}
For a given $\alpha$, write 
\[(k+1)^{\alpha-1}(\ell+1)^{\alpha-1}|a_{k\ell}|^2=\left[(k+1)^{\alpha/2-1}(\ell+1)^{\alpha/2}|a_{k\ell}|\right]\cdot
 \left[(k+1)^{\alpha/2}(\ell+1)^{\alpha/2-1}|a_{k\ell}|\right].\]
Writing $f(z) = \sum_{k, \ell=0}^{\infty} a_{k \ell}z_1^k z_2^{\ell}$ and applying the Cauchy-Schwarz inequality, we obtain 
 \begin{multline*}\|f\|^2_{\mathfrak{D}_{\alpha-1}}=\sum_{k,\ell=0}^{\infty}(k+1)^{\alpha-1}(\ell+1)^{\alpha-1}|a_{k\ell}|^2\leq 
 \left(\sum_{k,\ell=0}^{\infty}(k+1)^{\alpha-2}(\ell+1)^{\alpha}|a_{k\ell}|^2\right)^{1/2}\\
 \cdot\left(\sum_{k,\ell=0}^{\infty}(k+1)^{\alpha}(\ell+1)^{\alpha-2}|a_{k\ell}|^2\right)^{1/2},
 \end{multline*}
and the expression on the right-hand side is equal to $\|f\|_{\mathfrak{D}_{(\alpha-2,\alpha)}}\|f\|_{\mathfrak{D}_{(\alpha, \alpha-2)}}$.
\end{proof}
We can now state a sufficient condition for a rational inner function to belong to a weighted Dirichlet space. Sometimes it is easier to check this condition for a range of exponents than to determine the exact contact orders of a given RIF.
\begin{proposition} \label{prop:da}
Let $\phi=\frac{\tilde{p}}{p}$ be a rational inner function on $\mathbb{D}^2$ and suppose $\frac{1}{p} \in \mathfrak{D}_{\alpha}$ for some $\alpha\leq 0$. Then $\phi \in \mathfrak{D}_{\alpha+1}$.
\end{proposition}
\begin{proof}
We can write $\frac{\partial \phi}{\partial z_1} =\frac{q}{p^2}=\frac{q}{p}\cdot\frac{1}{p}$. Since $q\in \langle p, \tilde{p}\rangle$ by the quotient rule, we have  $\frac{q}{p} \in H^{\infty}(\D^2)$.  Hence $\frac{q}{p}$ is a multiplier of $\mathfrak{D}_{\alpha}$ and so $\frac{\partial \phi}{\partial z_1}$ is in $\mathfrak{D}_{\alpha}$. The same reasoning applies to $\frac{\partial \phi}{\partial z_2}$.
Next, Lemma \ref{lem: inclusion} implies that $\frac{\partial^2 \phi}{\partial z_2 \partial z_1} \in \mathfrak{D}_{(\alpha, \alpha-2)}$ and $\frac{\partial^2 \phi}{\partial z_1 \partial z_2}  \in \mathfrak{D}_{(\alpha-2, \alpha)}$. Then by Lemma \ref{lem: CSlemma}, 
$\frac{\partial^2 \phi}{\partial z_2 \partial z_1} \in \mathfrak{D}_{\alpha-1}$, and another application of Lemma \ref{lem: inclusion} gives $\phi \in \mathfrak{D}_{\alpha+1}$.
\end{proof}
We use this result in the next section to exhibit some concrete examples of rational inner functions that belong to a range of weighted Dirichlet spaces.

\section{Examples}\label{sect: examples}
We now illustrate our theorems and questions with several examples.

\begin{example} \label{faveex}Let $\phi = \frac{2z_1z_2-z_1 - z_2}{2-z_1-z_2}.$  We will now show that  $\phi \in \mathfrak{D}_{\alpha}$ if and only if $\alpha < \frac{3}{4}.$ 

First observe that Example \ref{ex:favcontact} coupled with Corollary \ref{cor:Da} implies that $\phi \in \mathfrak{D}_{\alpha}$ for $\alpha < \frac{3}{4}.$ 
To obtain the rest of the claim, observe that
\[ 
\begin{aligned}
\phi(z_1, z_2)& = \frac{z_1z_2}{1- \frac{z_1+z_2}{2}} -  \frac{ \frac{z_1+z_2}{2}}{1- \frac{z_1+z_2}{2}} \\
& =\sum_{n=0}^{\infty} \sum_{j=0}^n \left( \begin{array}{c} n \\ j \end{array} \right) 2^{-n} z_1^{j+1}z_2^{n-j+1} + \sum_{n=0}^{\infty} \sum_{j=0}^{n+1} \left( \begin{array}{c} n +1\\ j \end{array} \right) 2^{-(n+1)} z_1^j z_2^{n+1-j} \\
\end{aligned}
\]
where the coefficients $a_{k \ell}$ are given by
\[
\begin{aligned}
 a_{(k+1)(\ell+1)} &= \left( \begin{array}{c} k+\ell \\ k\end{array} \right) 2^{-(k+\ell)} - \left( \begin{array}{c} k+\ell +2 \\ k +1 \end{array} \right) 2^{-(k+\ell +2)} \\ 
&=  \left( \begin{array}{c} k+\ell \\ k\end{array} \right) 2^{-n} \left( 1 -  \frac{(n+1)(n+2)}{4(k+1)(\ell+1)} \right), 
\end{aligned}
\]
for $n = k + \ell.$  Set $x =  \frac{n}{2} -k$. Then for $n$ sufficiently large and $|x| \le \frac{\sqrt{n}}{3}$, we can use an asymptotic estimate for the binomial coefficients, see \cite[pp. 66]{Sp14}, to obtain:
\[
\begin{aligned}
 a_{(k+1)(\ell+1)} & \approx \frac{2^n e^{\frac{-2x^2}{n}}}{\sqrt{ \frac{1}{2}\pi n}}   2^{-n} \left( 1 -  \frac{(n+1)(n+2)}{4(k+1)(\ell+1)} \right) \\
&\approx \frac{1}{\sqrt{ n}} \left( 1 -  \frac{(n+1)(n+2)}{4(k+1)(\ell+1)} \right) \\
& = \frac{1}{\sqrt{ n}} \left( \frac{(n+2) - 4x^2}{ (n+2+2x)(n+2-2x)} \right) \\
& \approx \frac{1}{n^{\frac{3}{2}}}.
\end{aligned}
\]
Fix $n$ and  restrict to $k,\ell$ with $|x| \le \frac{\sqrt{n}}{3}$. For $n$ sufficiently large, there are approximately $\frac{2}{3} \sqrt{n}$ positive integers $k$ satisfying $ \frac{n}{2} -  \frac{\sqrt{n}}{3}< k < \frac{n}{2} +  \frac{\sqrt{n}}{3}.$ Then for $n$ large, we have
\[  \sum_{\substack{0 \le  k\le n \\ |x| \le \frac{\sqrt{n}}{3}}}  (k+2)^{\alpha}(n-k+ 2)^{\alpha} | a_{(k+1)(n-k+1)}|^2  \approx 
\sum_{\substack{ 0\le k\le n \\ |x| \le \frac{\sqrt{n}}{3}}}  n^{\alpha}n^{\alpha} \frac{1}{n^3} \approx n^{2\alpha -3 + \frac{1}{2}} .\]
 Choosing $N$ sufficiently large, we can then compute
\[
\begin{aligned}
\| \phi \|^2_{\mathfrak{D}_{\alpha}} 
\ge  \sum_{n=N}^{\infty} \sum_{\substack{0 \le k\le n \\ |x| \le \frac{\sqrt{n}}{3}}}  (k+2)^{\alpha}(n-k+ 2)^{\alpha} | a_{(k+1)(n-k+1)}|^2 
\approx \sum_{n=N}^{\infty} n^{2\alpha -3 + \frac{1}{2}}.
\end{aligned}
\]
This final series diverges for $\alpha \ge \frac{3}{4}.$ This shows that $\phi \not \in \mathfrak{D}_{\alpha}$ for $\alpha \ge \frac{3}{4},$ as desired.

 We now provide a second method using Agler kernel arguments to show $\phi \not \in \mathfrak{D}$. By Lemma \ref{lem:norm}, we can work with the Douglas integral $\text{Doug}(\phi)$ and by Lemma \ref{lem:local2}, this involves Agler kernels.  First, in the notation of Lemma \ref{lem:local2}, we have $r_1(z_1,z_2)=1-z_1$. After a short computation, we find that 
\[\int_{\T}\frac{\left|\frac{r_1}{p}(\zeta_1,\eta_2)-\frac{r_1}{p}(\zeta_1,\zeta_2)\right|^2}{|\zeta_2-\eta_2|^2}|d\eta_2|=\int_{\T}\frac{|\zeta_1-1|^2}{|2-\zeta_1-\zeta_2|^2|2-\zeta_1-\eta_2|^2}|d\eta_2|.\]
By Lemma \ref{lem:local2}, this implies that
\begin{align*}\mathrm{Doug}(\phi)& \gtrsim  \int_{\T^2} \mathfrak{D}_{(\zeta_1, \zeta_2)}(\phi) |d\zeta_1| \ |d\zeta_2|  \\
&\gtrsim \int_{\T}\left(\int_{\T}\frac{|\zeta_1-1|}{|2-\zeta_1-\eta_2|^2}|d\eta_2|\right)\left(\int_{\T}\frac{|\zeta_1-1|}{|2-\zeta_1-\zeta_2|^2}|d\zeta_2|\right)|d\zeta_1|\\&=\int_{\T^2}\left(\frac{|\zeta_1-1|}{|2-\zeta_1-\zeta_2|^2}|d\zeta_2|\right)^2|d\zeta_1|\\
& \gtrsim \left(\int_{\T^2}\frac{|\zeta_1-1|}{|2-\zeta_1-\zeta_2|^2}|d\zeta| \right)^2.
\end{align*}
It remains to show that the last double integral is infinite. First, note that $k_{\lambda}(z_2):=(1-\overline{\lambda}z_2)^{-1}$, $\lambda \in \D$, is the reproducing kernel for $H^2(\mathbb{D}^2)$. Since $|2-\zeta_1|>1$ for $\zeta_1\in \T\setminus\{1\}$, this observation implies that
\[\frac{1}{2\pi}\int_{\T}\frac{1}{|2-\zeta_1-\zeta_2|^2}|d\zeta_2|=\frac{1}{|2-\zeta_1|^2-1}.\]
Setting $\zeta_1=\cos \theta+i\sin \theta$ gives
\[\frac{1}{4\pi^2}\int_{\T^2}\frac{|\zeta_1-1|}{|2-\zeta_1-\zeta_2|^2}|d\zeta|=\frac{1}{2\pi}\int_{\T}\frac{|\zeta_1-1|}{|2-\zeta_1|^2-1}|d\zeta_1| \approx
\frac{1}{2\pi}\int_{-\pi}^{\pi}\frac{1}{(1-\cos \theta)^{1/2}}d\theta.\]
The latter integral diverges, which implies that $\mathrm{Doug}(\phi)=\infty$ and $\phi \notin \mathfrak{D}$.
\end{example}

\begin{remark}\rm{
Any degree $(1,1)$ inner function of the form
\[\phi_{A,B}(z_1, z_2):=\frac{z_1z_2-\overline{B}z_1-\overline{A}z_2}{1-Az_1-Bz_2},\]
with $A,B\in\D\setminus\{0\}$ having $|A|+|B|=1$, is in $\Da$ if and only if $\alpha<\frac{3}{4}$.
This can be seen by transforming $\phi_{A,B}$ into the RIF treated in Example \ref{faveex} by precomposing with a suitable pair of M\"obius transformations, 
and exploiting the fact that such composition operators are bounded above and below in the integral version of the $\Da$ norm.}
\end{remark}

\begin{example}
The second example is based on $p(z_1,z_2)=2-z_1z_2-z_1^2z_2$. The associated rational inner function is 
\begin{equation}\label{ex:nonuniqeag}
\psi(z_1,z_2)=\frac{2z_1^2z_2-z_1-1}{2-z_1z_2-z_1^2z_2}.
\end{equation}
The relevant Pick function here is
\[f(w_1, w_2)=\frac{2w_1^2+3w_1w_2-1}{w_1+w_2}.\]

Since $\deg p=(2,1)$, B\'ezout's theorem  says that $N(p, \tilde{p})=4$. Unlike the previous example, however, $N_{\T^2}(p, \tilde{p})<N(p, \tilde{p})$. Namely, the polynomials $p$ and $\tilde{p}$ have a common zero of multiplicity $2$ at $(1,1)\in \T^2$, and further common zeros at $(0,\infty)$ and $(\infty,0)$, each of multiplicity $1$.
Theorem \ref{thm:Hp} guarantees that $\frac{\partial \psi}{\partial z_2}, \frac{\partial \psi}{\partial z_2}  \notin H^2(\D^2)$. In fact, the contact orders of $\psi$ can be computed and are both equal to $2$. Then by Theorem \ref{thm:contact}, the partial derivatives $\frac{\partial \psi}{\partial z_1}, \frac{\partial \psi}{\partial z_2} $ are in  $H^p(\D^2)$ precisely when $p<3/2$. 
 
We can also address $H^2(\D^2)$ membership using Agler decompositions. by \cite[Cor.13.6]{Kne15}, $\psi$ does not have a unique Agler decomposition. However, using the methods discussed in Appendix $B$ from \cite{Kne15}, we can compute its canonical Agler pairs $(\vec{E}_1, \vec{F}_2)$ and $(\vec{F}_1, \vec{E}_2)$ as follows:
\[ \vec{E}_1(z) = \begin{bmatrix} \sqrt{2}(1-z_1 z_2) \\ 1-z_1 \end{bmatrix} \qquad \vec{F}_1(z) = \begin{bmatrix} \sqrt{2}(1-z_1 z_2) \\ z_2(1-z_1) \end{bmatrix} \]
and 
\[ \vec{E}_2(z) = 1-z_1 \qquad \vec{F}_2(z) = z_1(1-z_1).\]
Using Theorem $7.1$ in \cite{Kne15}, it is clear that ${\mathcal{I}_p} = \langle 1-z_1, 1-z_1z_2\rangle.$ One can show that $ \mathcal{P}_{p^2, (3,1)}$ has the following basis:
\[ \{ (1-z_1)^3, z_2(1-z_1)^3, p (1-z_1), \tilde{p}(1-z_1)\}. \]
 Moreover, one can compute
\[\frac{ \partial \psi}{\partial z_1} (z_1,z_2)=-\frac{2z_1^2z_2^2+z_1^2z_2-6z_1z_2+z_2+2}{(2-z_1z_2-z_1^2z_2)^2} \quad\textrm{and}\quad\frac{ \partial \psi}{\partial z_2} (z_1,z_2)=-z_1\frac{(z_1-1)^2}{(2-z_1z_2-z_1^2z_2)^2}.\]
One can show that $z_1(z_1-1)^2 \not \in \mathcal{P}_{p^2, (3,1)}$, and so conclude $\frac{ \partial \psi}{\partial z_2}  \not \in H^2(\mathbb{D}^2)$  using the given basis. The derivative $\frac{ \partial \psi}{\partial z_1} $ appears less amenable to direct analysis as we are lacking a description of the corresponding ideal in the case of deficient intersection multiplicity.

We turn to the question of membership of $\psi$ in weighted Dirichlet spaces.  We address this question by determining a range of values of $\alpha$ for which $1/p \in \Da$ and applying Proposition \ref{prop:da}. We have
\[\frac{2}{p(z_1,z_2)}=\sum_{k=0}^{\infty}2^{-k}(z_1z_2)^k(1+z_1)^k=\sum^{\infty}_{k=0}2^{-k}z_2^k\sum_{j=0}^k\left(\begin{array}{c} k\\j\end{array}\right)z_1^{k+j}.\]
Now, since there are constants $c, C$ satisfying: $c(k+1)^{\alpha}\leq (k+j+1)^{\alpha}\leq C(k+1)^{\alpha}$ for $j=0,\ldots, k$, we have
\[\left \|\tfrac{1}{p}\right\|_{\Da}^2\approx \sum_{k=0}^{\infty}(k+1)^{2\alpha} 2^{-2k} \sum_{j=0}^k\left(\begin{array}{c} k\\j\end{array}\right)^2 \approx \sum_{k=0}^{\infty}(k+1)^{2\alpha-1/2} .\]
Thus $1/p\in\Da$ if and only if $\alpha<-1/4$, and moreover by Proposition \ref{prop:da}, we have  $\psi\in\Da$ for all $\alpha<3/4$. This means that we obtain membership for the same range of parameters that we would by using contact order and invoking Corollary \ref{cor:Da}.

On the other hand, $\psi \notin \mathfrak{D}$ even though $N_{\mathbb{T}^2}(p, \tilde{p})$ is not maximal. The argument is similar to that in the previous example. Namely, by the formula for $\vec{E}_1$, the kernel $K_1(z,\lambda)$ in \eqref{eqn:Akernels2} includes the term $r(z)\overline{r(\lambda)}/p(z)\overline{p(\lambda)},$ where  $r(z_1,z_2)=1-z_1$ and by direct computation,
\[\frac{r}{p}(z_1,\eta_2)-\frac{r}{p}(z_1,z_2)=\frac{z_1(1-z_1)(1+z_1)(\eta_2-z_2)}{p(z_1,z_2)p(z_1,\eta_2)}.\]
Then using Lemma \ref{lem:local2}, we have
\[\mathfrak{D}_{(\zeta_1,\zeta_2)}(\psi)\gtrsim \int_{\T}\frac{|1-\zeta_1|^2}{|p(\zeta_1,\zeta_2)|^2|p(\zeta_1,\eta_2)|^2}|d\eta_2|.\]
This means that 
\[\mathrm{Doug}(\psi)\gtrsim\int_{\T}\left(\int_{\T}\frac{|1-\zeta_1|}{|p(\zeta_1,\zeta_2)|^2}|d\zeta_2|\right)^2|d\zeta_1|\gtrsim \left(\int_{\T^2}\frac{|1-\zeta_1|}{|2-\zeta_2(\zeta_1+\zeta_1^2)|^2}|d\zeta|\right)^2.\]
Again invoking the reproducing property in $H^2(\mathbb{D}^2)$, we see that for $\zeta_1\neq 1$,
\[\frac{1}{2\pi}\int_{\T}\frac{1}{\left|2-\zeta_2(\zeta_1+\zeta_1^2)\right|^2}|d\zeta_2|=\frac{1}{4-\left|\zeta_1+1\right|^2}.\]
Finally, we obtain the desired estimate
\[\mathrm{Doug}(\psi)\gtrsim \int_{\T}\frac{|1-\zeta_1|}{4-|1+\zeta_1|^2}|d\zeta_1|\gtrsim \int_{0}^{2\pi} \frac{1}{(1-\cos \theta)^{1/2}}d\theta=\infty.\]

\end{example}

\section{Further remarks and open problems}
In this last section, we collect some observations and pose several problems concerning inner functions and their derivatives.
\subsection{Comparing $H^p(\D^2)$ and $\Da$ norms}
In Section \ref{sect: dirichletcontact}, we related one-variable Dirichlet norms to the contact order of a RIF on the bidisk and used this to deduce inclusion in $\Da$ from inclusion of derivatives in $H^\p(\D^2)$. One could go further and ask whether this implication is reversible, that is, whether membership of $\phi$ in $\mathfrak{D}_{\frac{\p}{2}}$ implies that $\frac{\partial \phi}{\partial z_1}, \frac{\partial \phi}{\partial z_2}\in H^\p(\D^2)$. If such a statement were true, then we would have the attractive characterization that $\phi \in \mathfrak{D}_{\frac{p}{2}}$ if and only if the contact orders of a RIF $\phi$ satisfied $K_i<\frac{1}{\p-1}$. In particular, $\Da$ would contain no rational inner functions with singularities on the two-torus when $\alpha\geq \frac{3}{4}$. 

\subsection{Properties of contact order}
In this paper, we have based many of our arguments on the notion of contact orders of a RIF, and it seems to the authors that this concept deserves to be explored further. Let us make a few observations here.

First, we note that there exist RIFs with arbitrarily high contact order: this follows from the existence of functions belonging to arbitrarily high intermediate L\"owner classes $\mathcal{L}^{J_{-}}$ \cite{Pas}. Moreover, contact order is a conformally invariant quantity as can be seen from the fact that $H^\p$-membership of the derivative of a RIF is conformally invariant. This implies that all level sets of a RIF, not just $\phi^{-1}(0)$, exhibit the same type of geometry. A futher study of such geometric issues seems to be of interest.

As we noted in the Geometric Julia Inequality \ref{cor: geojulia}, the contact orders are at least $2$ when the RIF has a singularity. However, it is not immediately apparent what 
general values the contact orders of a RIF can take. Specifically, in principle, it might be possible for a RIF to have non-integer contact orders, but we do not have an example of such a function.
\subsection{Singular inner functions}
One can pose the same kind of questions that we have addressed in this paper for RIFs for more general classes of bounded functions, for instance, for singular inner functions.
Such functions can be constructed as follows. Consider a measure $\mu$, singular with respect to Lebesgue measure on $\T^2$,  whose Fourier coefficients satisfies
\[\hat{\mu}(k,\ell)=0, \quad (k,\ell)\notin (\mathbb{Z}_+)^2\cup (-\mathbb{Z}_+)^2.\]
Compute its Poisson integral
\[P[\mu](z_1,z_2)=\int_{\T^2}P(r_1, s_1-t_1)P(r_2, s_2-t_t)d\mu(t_1,t_2),\quad 
(z_1,z_2)=(r_1e^{it_1}, r_2e^{it_2})\in \D^2,\]
and then find its harmonic conjugate $Q[\mu]$ in $\D^2$ to obtain a holomorphic function by setting $S[\mu]=\exp(-(P[\mu]+iQ[\mu]))$. The resulting function is 
bounded and has $|S[\mu](e^{is}, e^{it})|=1$ almost everywhere by Fatou's theorem for Poisson integrals in polydisks, see \cite[Chapter 2]{Rud69}.  A discussion of the basic properties of singular inner functions, as well as other subclasses of inner functions, can be found in \cite{Rud69}. Further developments can be found in, for instance, \cite{AhRud72}.

We consider a simple example.
\begin{example}\rm{
Let $\mu_{\mathcal{Z}}$ be the measure induced by the integration current associated with the subvariety 
\[\mathcal{Z}=\{(e^{it}, e^{-it})\colon t\in [0,2\pi)\}\subset \T^2,\]
and define, for $\sigma>0$, the singular measure $\mu_\sigma=\sigma \mu_{\mathcal{Z}}.$
We have $\hat{\mu}_\sigma(k,k)=\sigma$ and $\hat{\mu}_\sigma(k,\ell)=0$ otherwise, so this is an admissible measure. 
The corresponding Poisson integral is
\[P[\mu_{\sigma}](z_1,z_2)=\sigma+2\sigma\sum_{k=1}^{\infty}(r_1r_2)^k\cos[k(t_1+t_2)], \quad (z_1,z_2)=(r_1e^{it_1},r_2e^{it_2})\in \D^2.\]
The harmonic conjugate is
\[Q[\mu_\sigma](z_1,z_2)=2\sigma\sum_{k=1}^{\infty}(r_1r_2)^k\sin[k(t_1+t_2)],\]
and we obtain the analytic function 
\[F[\mu_\sigma](z_1,z_2)=P[\mu_\sigma](z_1,z_2)+iQ[\mu_\sigma](z_1,z_2)=\sigma\left(1+2\sum_{k=1}^{\infty}(z_1z_2)^k\right)=\sigma\frac{1+z_1z_2}{1-z_1z_2}.\]
After simplifying and exponentiating, we are left with the singular inner function
\[S[\mu_\sigma](z_1,z_2)=\exp\left(-\sigma\frac{1+z_1z_2}{1-z_1z_2}\right).\]
We now note that by construction, the singularity set of $S[\mu_{\sigma}]$ is a set of positive $\alpha$-capacity whenever $\alpha>\frac{1}{2}$, see \cite[Section 5]{BCLSS15},  and hence immediately deduce $S[\mu_{\sigma}] \notin \Da$ for $\alpha>\frac{1}{2}$ since exceptional sets of functions in $\mathfrak{D}_{\alpha}$ cannot have positive capacity \cite{Kap94}.

On the other hand, Newman and Shapiro \cite{NewSha62} have observed that the Taylor coefficients of the (one-variable) inner function $s_{\sigma}(z)=\mathrm{exp}\left(-\sigma\frac{1+z)}{1-z}\right)=\sum_{k\geq 0}a_kz^k$ can be expressed in terms of certain special functions. Namely, we have
\[a_k=e^{-\sigma}L_k^{-1}(2\sigma),\]
where $L_n^{m}(x)$ are associated Laguerre polynomials, and it is known that
\[a_k
\asymp \frac{\cos(k^{1/2})}{k^{3/4}}\]
for $\sigma>0$ fixed.
Using this, we find that $S[\mu_{\sigma}]=\sum_{k\geq 0}a_k(z_1z_2)^k$ has
\[\|S[\mu_{\sigma}]\|^2_{\alpha}\approx \sum_{k=0}^{\infty}(k+1)^{-3/2+2\alpha},\]
and thus, for all $\sigma>0$, $S[\mu_{\sigma}]$ belongs to $ \Da$ if and only if $\alpha< \frac{1}{4}$.
}
\end{example}
As the example shows, it is not the size of the support of a singular measure alone that determines whether the associated $S[\mu]$ belongs to $\Da$. We hope to return to membership problems for singular inner functions on the bidisk in future work.

\subsection{Inner functions in polydisks}
Finally, all  of the problems addressed in this paper can be stated for the $n$-dimensional polydisk $\D^n$ and are, to the best of the authors' knowledge, completely open. We should stress that many of the techniques used in this paper are strictly two-dimensional in nature: B\'ezout's theorem, Puiseux expansions, Agler decompositions, and Nevanlinna representations simply do not have exact analogs in higher dimensions. Nevertheless, we suspect that at least some of our results
should have $n$-dimensional counterparts in terms of  objects  like the general resolution of singularities, see \cite{Hau03} for a survey.
\subsection*{Acknowledgements}
The authors thank Greg Knese and Constanze Liaw for numerous conversations, comments, and suggestions during the initial stages of this work. The second and third authors thank Institute Mittag-Leffler, where they spent a productive week engaged in this project. The third author thanks Stefan Richter for useful remarks concerning local Dirichlet integrals, as well as Dan Petersen for helpful discussions about algebraic geometric concepts.

 
\end{document}